\documentclass[a4paper, intlimits, 10pt]{amsart}
\calclayout
\usepackage{enumerate}
\usepackage{color}
\usepackage{amssymb}
\usepackage{graphicx}
\usepackage{breakcites}
\usepackage{dsfont}
\definecolor{dark green}{rgb}{0.09, 0.45, 0.27}
\usepackage{hyperref}

\newtheorem{theorem}{Theorem}[section]
\newtheorem{proposition}[theorem]{Proposition}
\newtheorem{lemma}[theorem]{Lemma}
\newtheorem{corollary}[theorem]{Corollary}
\newtheorem{definition}[theorem]{Definition}
\newtheorem{remark}[theorem]{Remark}
\newtheorem{example}[theorem]{Example}
\newtheorem{assumption}[theorem]{Assumption}

\newtheorem{algorithm}[theorem]{Algorithm}

\newcommand{\R}{\mathbb{R}}

\newcommand{\N}{\mathds{N}}

\newcommand{\coloneq}{\mathrel{\mathop:}=}

\usepackage{tikz}
\usetikzlibrary{decorations.markings}
\usetikzlibrary{shapes.geometric}

\pgfdeclarelayer{edgelayer}
\pgfdeclarelayer{nodelayer}
\pgfsetlayers{edgelayer,nodelayer,main}

\tikzstyle{none}=[inner sep=0pt]
\tikzset{new/.style={thick}}

\definecolor{cccccc}{rgb}{0.8,0.8,0.8}
\definecolor{azure}{rgb}{0.0, 0.5, 1.0}
\definecolor{cqcqcq}{rgb}{0.75,0.75,0.75}
\definecolor{ao}{rgb}{0.0, 0.5, 0.0}
\definecolor{amber}{rgb}{1.0, 0.49, 0.0}
\definecolor{babyblue}{rgb}{0.54, 0.81, 0.94}
\definecolor{darkpastelgreen}{rgb}{0.01, 0.75, 0.24}
\definecolor{darkspringgreen}{rgb}{0.09, 0.45, 0.27}

\begin{document}
\title[Continuity of martingale optimal transport]{Continuity of  the martingale optimal transport problem on the real line}

\date{\today}

\author{Johannes Wiesel}

\address{
Johannes Wiesel\newline
Columbia University\newline
Department of Statistics\newline
1255 Amsterdam Avenue\newline
New York, NY 10027, USA}
\email{johannes.wiesel@columbia.edu}


\begin{abstract}
We show continuity of the martingale optimal transport optimisation problem as a functional of its marginals. This is achieved via an estimate on the projection in the nested/causal Wasserstein distance  of an arbitrary coupling on to the set of martingale couplings with the same marginals. As a corollary we obtain an independent proof of sufficiency of the monotonicity principle established in [Beiglb\"ock, M., \& Juillet, N. (2016). On a problem of optimal transport under marginal martingale constraints. Ann. Probab., 44 (2016), no. 1, 42-106] for cost functions of polynomial growth.
\end{abstract}

\thanks{Support from the European Research Council under the European Union's Seventh Framework Programme (FP7/2007-2013) / ERC grant agreement no. 335421, St. John's College in Oxford and the German Scholarship foundation are gratefully acknowledged. We thank Benjamin Jourdain, Jan Ob{\l}{\'o}j and Gudmund Pammer for helpful comments.}
\maketitle

\section{Introduction: The Martingale optimal transport problem and nested Wasserstein distance}
The martingale optimal transport (MOT) problem, which was introduced in \cite{Beiglbock:2013cxa} in discrete time and in \cite{HenryLabordere:2014hta} in continuous time, is a version of the optimal transport problem, which was first posed by Gaspard Monge in \cite{Monge:3L-czg_j}, with an additional martingale constraint. In recent years it has received considerable attention in the field of robust mathematical finance, as it can be utilised to obtain no-arbitrage pricing bounds. For an overview of recent developments in the field we refer to \cite{Beiglbock:2016kt}, \cite{Beiglbock:2017hcc} and the references therein. \\
Given two measures $\mu$ and $\nu$ on the real line, let us denote by $\Pi(\mu, \nu)$ the set of probability measures on $\R^2$ with marginals $\mu$ and $\nu$. With this notation at hand the MOT problem reads
\begin{align}\label{eq:mot}
C(\mu,\nu)\coloneq \inf_{\pi\in \mathcal{M}(\mu,\nu)} \int c(x_1,x_2)\,\pi(dx_1,dx_2).
\end{align}
Here $\mathcal{M}(\mu, \nu)$ is the set of martingale couplings
\begin{align*}
\mathcal{M}(\mu,\nu)=\left\{\pi\in \Pi(\mu,\nu) \ :\ \int (x_2-x_1)\, \pi_{x_1}(dx_2) =0 \quad \mu\text{-a.s.}\right\},
\end{align*}
$(\pi_{x_1})_{x_1\in \R}$ denotes a regular disintegration of the coupling $\pi$ with respect to its first marginal $\mu$ and $c: \R^2 \to \R$ is a Borel measurable function. \\
In this paper we establish continuity of the mapping 
\begin{align}\label{eq:main}
(\mu,\nu)\mapsto C(\mu,\nu)
\end{align}
with respect to the Wasserstein metric and give a new proof of sufficiency of the monotonicity principle for martingale optimal transport, which was introduced in \cite[Lemma 1.11]{Beiglbock:2016kt}, as a consequence of this result. Such a continuity property is well known for classical optimal transport (see e.g. \cite[Theorem 5.20, p.77]{Villani:2009ha}), but has only quite recently  been proven for martingale optimal transport in the independent work \cite{BackhoffVeraguas:2019vwa}. Before, partial results have been obtained in \cite{Juillet:2016ksb} and \cite{Guo:2017txa}. Establishing continuity of $(\mu, \nu) \mapsto C(\mu, \nu)$ is clearly of paramount importance for any practical applications such as computational methods or statistical estimation, when approximations cannot be avoided or uncertainty in the underlying data is present.\\
Contrary to \cite{BackhoffVeraguas:2019vwa}, our main stability result is proved via an estimate of the nested $1$-Wasserstein distance $W_1^{nd}$ between a coupling $\pi \in \Pi(\mu, \nu)$ and the set $\mathcal{M}(\mu, \nu)$. More specifically, we show
\begin{align}\label{eq:main2}
\inf_{\tilde{\pi} \in \mathcal{M}(\mu,\nu)} W_1^{nd}(\pi, \tilde{\pi}) \approx \int \left|\int (x_2-x_1)\, \pi_{x_1}(dx_2) \right|\, \mu(dx_1),
\end{align}
where $W_1^{nd}$ is defined as (see \cite[Proposition 5.2]{Backhoff:2017gk})
\begin{align}\label{eq:defn_nested_wasserstein}
W_1^{nd}(\pi,\tilde{\pi})&=\inf_{\gamma^1\in \Pi(\pi^1,\tilde{\pi}^1)} \Bigg(\int|x_1-y_1|\,\gamma^1(dx_1,y_1) \\\
&\qquad\qquad+\int\inf_{\gamma^2\in \Pi(\pi_{x_1},\tilde{\pi}_{y_1})}\int |x_2-y_2|\,\gamma^2(dx_2, dy_2)\, \gamma^1(dx_1,dy_1)\Bigg) \nonumber.
\end{align}
Here $\pi^1, \tilde{\pi}^1$ denote the first marginals of $\pi$ and $\tilde{\pi}$ respectively, while $(\pi_{x_1})_{x_1\in \R}$, $(\tilde{\pi}_{y_1})_{y_1\in \R}$ denote the disintegrations of $\pi$ and $\tilde{\pi}$ with respect to the first marginal. On an intuitive level, the nested Wasserstein distance only considers those couplings $\gamma\in \Pi(\pi, \tilde{\pi})$, which respect the information flow formalised by the canonical (i.e. coordinate) filtration $(\mathcal{F}_t)_{t\in \{1,2\}}$: in \eqref{eq:defn_nested_wasserstein} this is achieved by first taking an infimum over  couplings of $\pi^1, \tilde{\pi}^1$ (i.e. ``couplings at time one") and then a second (nested) infimum with respect to the respective disintegrations (i.e. ``conditional couplings at time two"). This feature distinguishes $W_1^{nd}$ from the Wasserstein distance $W_1$, which also includes ``anticipative couplings". We refer to \cite[pp. 2-3]{BackhoffVeraguas:2019tnb} for a well-written introduction to this topic. The nested distance was introduced in \cite{Pflug:2009hl}, \cite{Pflug:2012bfa} in the context of multistage stochastic optimisation and was independently analysed in \cite{Lassalle:2018hfa}.\\
Our estimate \eqref{eq:main2} complements the results of \cite{ BackhoffVeraguas:2019vwa}, who essentially show continuity of the monotonicity principle for MOT without using the primal formulation~\eqref{eq:mot} directly. We believe that it is of independent interest as it implies uniform continuity of the mapping $\pi\mapsto \inf_{\tilde{\pi} \in \mathcal{M}(\pi^1,\pi^2)} W_1^{nd}(\pi, \tilde{\pi})$ under a uniform integrability constraint on the second marginal of $\pi$, which we denote by $\pi^2$. Furthermore we show that our estimate \eqref{eq:main2} is sharp for a class of couplings $\pi \in \Pi(\mu, \nu)$ satisfying a dispersion assumption in the spirit of \cite{Hobson:2015cka}, extending results obtained in \cite{Jourdain:2018teb}.\\
The remainder of this article is organised as follows: we state our main results in Section \ref{sec:main}. The proof of Proposition \ref{thm:dispersion} is given in Section \ref{sec:thm_dispersion}, while we prove Theorem \ref{thm:W1} in Section \ref{sec:W1}. In Section \ref{sec:proofs} we collect proofs of the remaining results announced in Section \ref{sec:main}. We also list some generic approximation results for $W_1^{nd}$ and $W_p$ in Section \ref{sec:generic}. These will be frequently used in Sections \ref{sec:thm_dispersion} and \ref{sec:W1} and are proved in the appendix.

\section{Main results}\label{sec:main}
\subsection{Notation}
Let us first outline the notation used in this paper. We denote by $\mathcal{P}(\R^d)$  the probability measures on $\R^d$ and write $\mathcal{P}_p(\R^d)=\{\mu\in \mathcal{P}(\R^d) \ : \ \int |x|^p\mu(dx)<\infty\}$, where $p \ge 1$ and $|\cdot|$ is the Euclidean distance on $\R^d$. For two probability measures $\mu,\nu\in \mathcal{P}(\R^d)$ let $\Pi(\mu, \nu)$ denote the set of couplings $\pi\in\mathcal{P}( \R^d\times \R^d)$ with marginals $\mu$ and $\nu$. Let $f_\ast \pi$ denote the push-forward measure of $\pi$ by the measurable function $f:\R^d\times \R^d\to \R^d$. For $\pi \in \mathcal{P}(\R^d\times \R^d)$ we denote by $\pi^1$ and $\pi^2$ the push-forward measures of $\pi$ under the canonical projection to the first coordinate $(x_1, x_2) \mapsto x_1$ and second coordinate $(x_1,x_2)\mapsto x_2$ respectively. Furthermore a disintegration (or regular conditional probability distribution) of $\pi$ is defined as a   family of probability measures $(\pi_{x_1})_{x_1\in \R^d}$ on $\R^d$  such that for every Borel set $B\subseteq \R^d$ the mapping $x_1\mapsto \pi_{x_1}(B)$ is Borel measurable and for all Borel sets $A,B\subseteq\R^d$ 
\begin{align}\label{eq:idiot2}
\pi(A\times B)=\int_A \pi_{x_1}(B)\,\pi^1(dx_1).
\end{align}
For a general existence result on Polish spaces and fundamental properties of disintegrations we refer to \cite[pp.12-19]{strvar}.
More generally, for a disintegration $(\pi_{x_1})_{x_1\in \R^d}$ on $\R^d$ and a measure $\mu$ on $\R^d$ we denote by $\mu \otimes \pi_{x_1}$ the measure obtained via $\mu\otimes \pi_{x_1}(A \times B)= \int_{A} \pi_{x_1}(B)\,\mu(dx_1)$ for Borel $A,B \subseteq \R^d$. The product coupling of $\pi\in \mathcal{P}(\R^d)$ and $\tilde{\pi}\in \mathcal{P}(\R^d)$ will be denoted by $\pi\times \tilde{\pi}$. We also write $\text{supp}(\pi)$ for the support of a measure $\pi\in \mathcal{P}(\R^d)$ and often use $\mu(x_1):=\mu(\{x_1\})$ to shorten notation. Given a set $\Gamma \in \R^d\times \R^d$ we write 
$\Gamma^1\coloneq \{x_1\in \R^d\ : \ \exists x_2\in \R^d \text{ such that } (x_1, x_2)\in \Gamma\}$ and $\Gamma_{x_1}\coloneq \{x_2\in \R^d \ : \ (x_1, x_2)\in \Gamma\}$. We also write $\int_{\{x_1 \in A\}}$ instead of $\int_{\{(x_1,x_2)\in  A\times \R^d\}}$ for a Borel set $A\subseteq \R^d$.
\\
For $\mu,\nu\in \mathcal{P}_1(\R)$ let $\mathcal{M}(\mu, \nu)$ be the set of martingale couplings
\begin{align*}
\mathcal{M}(\mu,\nu)=\left\{\pi\in \Pi(\mu,\nu) \ :\ \int (x_2-x_1)\, \pi_{x_1}(dx_2) =0 \quad \mu\text{-a.s.}\right\},
\end{align*}
where $(\pi_{x_1})_{x_1\in \R}$ is a disintegration of $\pi$.  We denote the convex order of $\mu$ and $\nu$ by $\mu \preceq_c \nu$, i.e. $\mu \preceq_c \nu$ iff $\int f d\mu \le \int fd\nu$ for all convex functions $f:\R\to \R$. It is well known that $\mu \preceq_c \nu$ is equivalent to $\mathcal{M}(\mu, \nu)\neq \emptyset$ (see \cite{Strassen:1965ha}).  We call a set of measures $\mathfrak{P}\subseteq \mathcal{P}(\R)$ uniformly integrable if $\lim_{K \to \infty} \sup_{\mu \in \mathfrak{P}}\int_{\{|x|\ge K\}} |x|\,\mu(dx)=0$.
Next we recall the $p$-Wasserstein distance on $\mathcal{P}(\R^2)$ given by
\begin{align*}
W_p(\pi,\tilde{\pi})= \left(\inf_{\gamma\in\Pi(\pi,\tilde{\pi})}\int |x_1-y_1|^p+|x_2-y_2|^p\,\gamma(dx,dy) \right)^{1/p},
\end{align*}
where $\Pi(\pi, \tilde{\pi})\subseteq \mathcal{P}(\R^2 \times \R^2)$ is the set of couplings with first marginal $\pi\in \mathcal{P}(\R^2)$ and second marginal $\tilde{\pi}\in \mathcal{P}(\R^2)$, 
and the nested $p$-Wasserstein distance
\begin{align*}
W_p^{nd}(\pi,\tilde{\pi})&= \Bigg(\inf_{\gamma^1\in \Pi(\pi^1,\tilde{\pi}^1)} \Bigg(\int|x_1-y_1|^p\,\gamma^1(dx_1,y_1) \\\
&\qquad\qquad+\int \inf_{\gamma^2\in \Pi(\pi_{x_1},\tilde{\pi}_{y_1})}\int |x_2-y_2|^p\,\gamma^2(dx_2, dy_2)\, \gamma^1(dx_1,dy_1)\Bigg)\Bigg)^{1/p}.
\end{align*}
For ease of notation we furthermore define 
\begin{align*}
\epsilon_{\pi} \coloneq \int \left|\int (x_2-x_1)\,\pi_{x_1}(dx_2)\right|\,\mu(dx_1)
\end{align*}
for $\pi \in \Pi(\mu,\nu)$.
\\\\
Fix $\mu,\nu\in \mathcal{P}_1(\R)$. We now investigate the nested distance $W_1^{nd}$ between a coupling $\pi\in \Pi(\mu,\nu)$ and its projection on to the set $\mathcal{M}(\mu,\nu)$. For the sake of clarity we first give a lower bound on $\inf_{\tilde{\pi}\in \mathcal{M}(\mu,\nu)} W^{nd}_1(\pi,\tilde{\pi})$ and then derive upper bounds under progressively less restrictive assumptions.

\subsection{Projection on to $\mathcal{M}(\mu, \nu)$: attainment of lower bound}
Let us first derive a lower bound on $\inf_{\tilde{\pi}\in \mathcal{M}(\mu,\nu)} W^{nd}_1(\pi,\tilde{\pi})$:
\begin{lemma}\label{lem:upperbound} Let $\mu,\nu \in \mathcal{P}_1(\R)$, $\mu \preceq_c \nu$ and $\pi \in \Pi(\mu,\nu)$. Then
\begin{align}\label{eq:upperbound}
\inf_{\tilde{\pi}\in \mathcal{M}(\mu,\nu)} W^{nd}_1(\pi,\tilde{\pi})\ge \epsilon_{\pi}.
\end{align}
\end{lemma}

We introduce the following assumption:

\begin{assumption}[Barycentre dispersion assumption]\label{def:dispersion}
Let $\mu \preceq_c \nu$ and $\pi\in \Pi(\mu,\nu)$. For all $x \in \R$
\begin{align}\label{eq:generaliseddispersion}
\int_{\{x_1 \ge x\}} (x_2-x_1)\, \pi(dx_1, dx_2) \ge 0.
\end{align}
\end{assumption}

In contrast to \cite{Hobson:2015cka} our dispersion assumption \ref{def:dispersion} is formulated for $\pi$ and not just the marginals $\mu$ and $\nu$. In order to motivate it, let us recall the Hoeffding--Frechet coupling $\pi_{HF}\in \Pi(\mu,\nu)$: it enjoys the property, that it is an optimiser for problems of the form
\begin{align*}
\inf_{\pi\in \Pi(\mu,\nu)}\int c(x_1,x_2)\, \pi(dx_1,dx_2),
\end{align*}
where $c(x_1,x_2)=h(x_2-x_1)$ and $h:\R\to \R$ is any convex function. In fact $\pi_{HF}$ is characterised by the following monotonicity property:
\begin{align}\label{eq:fh_montonone}
&\quad\text{There exists a Borel set }\Gamma_{HF}\subseteq \R^2 \text{ such that }\pi_{HF}(\Gamma_{HF})=1\nonumber\\
&\text{and whenever }(x_1,x_2),(y_1,y_2)\in \Gamma_{HF}\text{ and }x_1<y_1\text{ then also }x_2\le y_2. 
\end{align}
This characterisation directly implies the following lemma:
\begin{lemma}\label{lem:fh}
Let $\mu,\nu \in \mathcal{P}_1(\R)$ with $\mu \preceq_c \nu$ and $\pi_{HF}\in \Pi(\mu, \nu)$ be the Hoeffding--Frechet coupling. Then $\pi_{HF}$ satisfies the barycentre dispersion assumption \ref{def:dispersion}.
\end{lemma}
It turns out that our lower bound on $\inf_{\tilde{\pi}\in \mathcal{M}(\mu,\nu)} W_1^{nd}(\pi, \tilde{\pi})$ is tight under the barycentre dispersion assumption:
\begin{proposition}\label{thm:dispersion}
Let $\mu,\nu \in \mathcal{P}_1(\R)$ satisfy $\mu \preceq_c \nu$. Let $\pi\in \Pi(\mu,\nu)$ satisfy the barycentre dispersion assumption \ref{def:dispersion}. Then there exists a martingale measure $\pi_{mr}\in \mathcal{M}(\mu,\nu)$ such that
\begin{align}\label{eq:dispersion}
\inf_{\tilde{\pi}\in \mathcal{M}(\mu,\nu)} W^{nd}_1(\pi,\tilde{\pi})=W_1^{nd}(\pi,\pi_{mr})=\epsilon_{\pi}.
\end{align}
\end{proposition}
We call a martingale coupling $\pi_{mr}\in \mathcal{M}(\mu,\nu)$ satisfying \eqref{eq:dispersion} a \emph{($W_1^{nd}$-minimal) martingale rearrangement coupling} of $\pi$. We now discuss some basic properties of $\pi_{mr}$. Let us first remark that, as $W_1(\cdot, \cdot) \le W_1^{nd}(\cdot, \cdot)$,  we have  $W_1(\pi, \pi_{mr})\le \epsilon_{\pi}$ in Proposition \ref{thm:dispersion} and this inequality is strict in general. Furthermore, while existence of $\pi_{mr}$ is guaranteed in Proposition \ref{thm:dispersion}, uniqueness is not satisfied in general as the following example shows:
\begin{example}
Take $$\pi\coloneq\frac{1}{4}\left( \delta_{(-2,-3)}+\delta_{(-1,-2)}+\delta_{(1,2)}+\delta_{(2,3)}\right).$$ As $\epsilon_{\pi}=1$ it remains to check that both $$\bar{\pi}=\frac{1}{20} \left(4\delta_{(-2,-3)}+\delta_{(-2,2)}+4\delta_{(-1,-2)}+\delta_{(-1,3)}+4\delta_{(1,2)}+\delta_{(1,-3)}+4\delta_{(2,3)}+\delta_{(2,-2)}\right)$$
and
$$\hat{\pi}\coloneq \frac{1}{24}\left(5\delta_{(-2,-3)}+\delta_{(-2,3)}+5\delta_{(2,3)}+\delta_{(2,-3)}\right)+\frac{1}{16}\left(3\delta_{(-1,-2)}+\delta_{(-1,2)}+3\delta_{(1,2)}+\delta_{(1,-2)}\right)$$
are $W_1^{nd}$-minimal rearrangement couplings.
\end{example}
By an application of the triangle inequality the following corollary of Proposition \ref{thm:dispersion} is immediate:
\begin{corollary}\label{cor:special}
Let $c:\R^2\to \R$ be $L$-Lipschitz-continuous. Then
\begin{align*}
 C(\mu, \nu)=\inf_{\pi\in \Pi(\mu,\nu),\ \pi\ \mathrm{satisfies\  Ass. \ \ref{def:dispersion}}} \left( \int c(x_1,x_2)\,\pi(dx_1,dx_2) +
L\epsilon_{\pi} \right).\end{align*}
\end{corollary}

Proposition \ref{thm:dispersion} and Corollary \ref{cor:special} complement \cite{Jourdain:2018teb}, who give a characterisation of the above $W_1^{nd}$-minimal martingale couplings of the Hoeffding-Frechet coupling $\pi_{HF}$ by use of its characterisation via quantile functions. More specifically  \cite[Theorem 2.11]{Jourdain:2018teb} states 
\begin{align}\label{eq:jm}
\inf_{\pi \in \mathcal{M}(\mu, \nu)} \int |x_1-x_2|\, \pi(dx_1, dx_2) \le 2 W_1(\mu, \nu).
\end{align}
We can recover \eqref{eq:jm}, noting that
\begin{align*}
\inf_{\pi \in \mathcal{M}(\mu, \nu)} \int |x_1-x_2|\, \pi(dx_1, dx_2)&\le \int |x_1-x_2|\,\pi_{HF}(dx_1,dx_2)+\epsilon_{\pi_{HF}}\\
&= \int |x_1-x_2|\,\pi_{HF}(dx_1,dx_2)\\
&\quad+\int \left| \int (x_2-x_1)\,\pi_{HF,x_1}(dx_2)\right|\,\mu(dx_1)\\
&\le 2 \int |x_1-x_2|\,\pi_{HF}(dx_1,dx_2)=2W_1(\mu,\nu),
\end{align*}
where we used Lemma \ref{lem:fh} and Corollary \ref{cor:special} in the first inequality and Jensen's inequality for the second inequality.\\

Let us lastly give the following remark.
\begin{remark}\label{rem:antitone}
While Assumption \ref{def:dispersion} is sufficient for \eqref{eq:dispersion}, it is not necessary. Indeed, at least in the finitely supported case, \eqref{eq:dispersion} also holds for the antitone or decreasing monotone coupling $\pi_{AT}$, which satisfies
\begin{align}\label{eq:antitone}
&\quad\text{There exists a Borel set }\Gamma_{AT}\subseteq \R^2 \text{ such that }\pi_{AT}(\Gamma_{AT})=1\nonumber\\
&\text{and whenever }(x_1,x_2),(y_1,y_2)\in \Gamma_{AT}\text{ and }x_1<y_1\text{ then }x_2\ge y_2. 
\end{align}
\end{remark}
We leave the question of finding a necessary condition for \eqref{eq:dispersion} for future research.\\

\begin{figure}[h!]
\begin{tikzpicture}[line cap=round,line join=round,x=1.5cm,y=1cm]
\draw[-,color=black] (0.5,0) -- (7.5,0) node[right]{$\mu$}; 
\draw[-,color=black] (0.5,3) -- (7.5,3)  node[right]{$\nu$}; 
\draw[->, line width=2pt, color=red] (3,0)  node[below] {\footnotesize $x_1^-$} -- (1,3) node[above] {\footnotesize $x_2^-$};
\draw[->,line width=2pt, color=red] (3,0) -- (3.5,3);
\draw[->,line width=2pt, color=azure] (5,0) node[below] {\footnotesize $x_1^+$} -- (7,3) node[above]  {\footnotesize $x_2^+$};
\draw[->,line width=2pt, color=azure] (5,0) -- (4.5,3);
\end{tikzpicture}
\begin{tikzpicture}[line cap=round,line join=round,x=1.5cm,y=1cm]
\draw[-,color=black] (0.5,0) -- (7.5,0) node[right]{$\mu$}; 
\draw[-,color=black] (0.5,3) -- (7.5,3) node[right]{$\nu$}; 
\draw[->, dashed, line width=2pt, color=red] (3,0)  node[below] {\footnotesize $x_1^-$} -- (1,3) node[above] {\footnotesize $x_2^-$};
\draw[->,line width=2pt, color=red] (3,0) -- (3.5,3);
\draw [->, dashed, line width=2pt, color=red] (3,0) -- (7,3);
\draw[->,dashed, line width=2pt, color=azure] (5,0) node[below] {\footnotesize $x_1^+$} -- (7,3)  node[above]  {\footnotesize $x_2^+$};
\draw[->, line width=2pt, color=azure] (5,0) -- (4.5,3) ;
\draw [->, dashed, line width=2pt, color=azure] (5,0) -- (1,3);
\end{tikzpicture}
\caption{Exchange of masses at $x_2^-$ and $x_2^+$ for the case $x_1^-<x_1^+$.}
\label{fig:1}
\end{figure}

The proofs of the above results are deferred to Sections \ref{sec:thm_dispersion} and \ref{sec:proofs} and rely on the following simple observation: let us assume for the moment that $\pi\in \Pi(\mu,\nu)\setminus \mathcal{M}(\mu, \nu)$ is finitely supported and let us consider the barycentres of the disintegration $(\pi_{x_1})_{x_1\in \R}$ given by $\left(\int (x_2-x_1)\,\pi_{x_1}(dx_2)\right)_{x_1\in \text{supp}(\mu)}$. By the barycentre dispersion assumption \ref{def:dispersion} and convex order of $\mu$ and $\nu$ we can find pairs $x_1^-, x_1^+\in\text{supp}(\mu)$ such that
\begin{align*}
 \int (x_2-x_1^-)\,\pi_{x_1^-}(dx_2)<0,  \qquad \int (x_2-x_1^+)\,\pi_{x_1^+}(dx_2)>0
\end{align*}
and corresponding points $x_{2}^-\in \text{supp}\left(\pi_{x_1^-}\right), x_2^+\in \text{supp}\left(\pi_{x_1^+}\right)$ with $x_1^-<x_1^+$ and $x_2^-<x_1^-<x_1^+<x_2^+$. Assigning a part of the mass at $x_2^-$ and $x_2^+$ to the disintegrations $\pi_{x_1^+}$ and $\pi_{x_1^-}$ respectively then allows to essentially rectify the barycentres of $\pi_{x_1^-}$ and $\pi_{x_1^+}$ piece by piece without changing the marginal constraints (see Figure \ref{fig:1}).

\subsection{Projection on to $\mathcal{M}(\mu, \nu)$: the general case}
It turns out that \eqref{eq:dispersion} is not satisfied in general (see Example \ref{Ex:1}). Instead we obtain the following relaxation of Proposition \ref{thm:dispersion} as a main result:
\begin{theorem}\label{thm:W1}
Let  $\mathfrak{P} \subseteq \mathcal{P}_1(\R)$ be uniformly integrable, i.e. $$\lim_{K \to \infty} \sup_{\mu \in \mathfrak{P}}\int_{\{|x|\ge K\}} |x|\,\mu(dx)=0.$$
 Then for every $\delta>0$ there exists a constant $K=K(\delta,\mathfrak{P})$ such that the following holds: for every measure $\pi\in \Pi(\mu,\nu)$, where $\mu\preceq_c\nu$ and $\nu\in \mathfrak{P}$, we have
\begin{align}\label{eq:rea}
\inf_{\tilde{\pi}\in \mathcal{M}(\mu,\nu)} W^{nd}_1(\pi,\tilde{\pi})\le K \epsilon_{\pi}+\delta.
\end{align} 
\end{theorem}
Similarly to Corollary \ref{cor:special} we obtain:
\begin{corollary}\label{cor:jourdain}
Let $c:\R^2\to \R$ be $L$-Lipschitz-continuous. Then for every $\delta>0$ there exists a constant $K=K(\delta,\nu)$ such that 
\begin{align*}
\inf_{\pi\in \Pi(\mu,\nu)} \int c(x_1,x_2)\,\pi(dx_1,dx_2) +
K L\epsilon_{\pi}
&\le C(\mu, \nu) \\
&\le \inf_{\pi\in \Pi(\mu,\nu)}  \int c(x_1,x_2)\,\pi(dx_1,dx_2) +
KL \epsilon_{\pi} +L\delta.
\end{align*}
\end{corollary}
Consequently $C(\mu,\nu)$ can be approximated by an optimal transport problem with cost function $\tilde{c}(x_1,x_2, \pi_{x_1})=c(x_1,x_2)+K(\delta,\nu)L|\int (y_2-x_1)\, \pi_{x_1}(dy_2)|$ for an $L$-Lipschitz-continuous cost function $c$.  The function  $\tilde{c}$ can be interpreted as a sum of a usual optimal transport cost and a weak optimal transport cost in the spirit of \cite{gozlan2017kantorovich}. The penalisation approach of Corollary \ref{cor:jourdain} is also akin to the numerical approximation results for the MOT problem obtained in \cite{Guo:2017txa}.\\

The dependence of $K$ on $\delta$ and $\mathfrak{P}$ in Theorem \ref{thm:W1} above is crucial, as the following counterexample shows:
\begin{example}\label{Ex:1}
Let us consider $$\mu^n=\nu^n=\frac{1}{n}\sum_{i=1}^n \delta_{i}.$$ Then trivially $\mu^n\preceq_c \nu^n$ for all $n\in \N$, $\mathfrak{P}=\{\nu_n\ : \ n\in \N\}$ is not uniformly integrable and the only martingale coupling $\tilde{\pi}^n\in \mathcal{M}(\mu^n,\nu^n)$ is supported on the diagonal $x_1=x_2$. We take $$\pi^n=\frac{1}{n} \left(\frac{\delta_{(1,1)}+\delta_{(1,2)}}{2}+\frac{\delta_{(n,n-1)}+\delta_{(n,n)}}{2}+\sum_{i=2}^{n-1} 
\frac{\delta_{(i,i-1)}+\delta_{(i,i+1)}}{2}\right),$$
which is ``almost" a martingale coupling. Then 
\begin{align*}
\inf_{\tilde{\pi}\in \mathcal{M}(\mu^n,\nu^n)} W^{nd}_1(\pi^n,\tilde{\pi})=\frac{n-1}{n} \qquad \text{and}\qquad \int \left| \int (x_2-x_1)\,\pi^n_{x_1}(dx_2)\right|\,\mu(dx_1)=\frac{1}{n}.  
\end{align*}
Thus for any $0\le\delta<1$ there exists no $K>0$, which fulfils \eqref{eq:rea}  simultaneously for all $(\pi^n)_{n\in \N}$.
\end{example}

\subsection{Continuity of MOT}

We now turn to our second main result, which establishes continuity of the map $(\mu,\nu)\mapsto C(\mu,\nu)$:
\begin{theorem}\label{thm:approx_simple}
Let $p\ge 1$ and let $(\mu^n)_{n\in \N}$, $(\nu^n)_{n\in \N}$ be two sequences of measures in $\mathcal{P}_p(\R)$ with \mbox{$\mu^n \preceq_c \nu^n$} for all $n \in \N$. Let $\mu,\nu\in \mathcal{P}(\R)$ be such that $\lim_{n\to \infty}W_p(\mu^n,\mu)=0$ and $\lim_{n\to \infty}W_p(\nu^n,\nu)=0$. Furthermore let $c: \R^2\to \R$ be continuous and such that $|c(x_1,x_2)|\le C(1+|x_1|^p+|x_2|^p)$ for some $ C\ge 0$. Then
\begin{align*}
\lim_{n \to \infty} C(\mu^n, \nu^n)=C(\mu, \nu).
\end{align*}
\end{theorem}
This stability result extends the findings of \cite{Juillet:2016ksb} and \cite{Guo:2017txa}.  \cite{Juillet:2016ksb} proves continuity of the left-curtain coupling with respect to its marginals in a Wasserstein-type metric. In particular the results obtained only hold for cost functions satisfying the Spence-Mirrlees condition $c_{xyy}<0$. On the other hand \cite[Prop. 4.7]{Guo:2017txa} assume a Lipschitz-continuous cost function $c$ together with a finite second moment of $\nu$ and exploit a duality result for martingale optimal transport. Our result is more general and only considers the primal formulation of $C(\mu,\nu)$ given in \eqref{eq:mot}. It is akin to a similar stability result in optimal transport with the obvious modifications. The proof of Theorem \ref{thm:approx_simple} extends a natural construction given in \cite[proof of Proposition 4.2, p. 20]{Guo:2017txa}, which essentially couples the marginals $\mu^n, \nu^n$ with the disintegration $(\pi_{x_1})_{x_1\in \R}$. In a second step one then corrects the new coupling to account for the martingale constraint, which is achieved by an application of Theorem \ref{thm:W1}.

\subsection{An independent proof of the monotonicity principle for MOT}
As in classical optimal transport, it is desirable to characterise the sets $\Gamma\subseteq \R^2$, on which optimisers of $C(\mu,\nu)$ live. This has been achieved in the influential work \cite{Beiglbock:2016kt} and is known as a monotonicity principle for martingale optimal transport. To set up notation we recall here the notion of a competitor given in \cite{Beiglbock:2016kt}, which naturally extends the corresponding optimal transport formulation. We recall that $\alpha^1$ denotes push-forward measure of $\alpha$ under the canonical projection to the first coordinate $x=(x_1,x_2)\mapsto x_1$:
\begin{definition}\label{def:monotone}
Let $\alpha\in \mathcal{P}(\R^2)$. We say that $\alpha'\in \mathcal{P}(\R^2)$ is a competitor of $\alpha$, if $\alpha'$ has the same marginals as $\alpha$ and
\begin{align*}
\int y\,\alpha_x(dy)= \int y\,d\alpha'_x(dy)\qquad  \alpha^1\text{-a.s.}
\end{align*}
\end{definition}

The following monotonicity principle was first stated in \cite[Lemma 1.11, p. 49]{Beiglbock:2016kt}, where necessity and a partial sufficiency result was shown.

\begin{theorem}\label{thm:monotone}
Assume that $\mu,\nu\in \mathcal{P}_p(\R)$ satisfy $\mu \preceq_c \nu$ and that $c : \R^2\to \R$ is a continuous cost function such that $|c(x_1,x_2)|\le \tilde{K}(1+|x_1|^p)+|x_2|^p)$ for some $\tilde{K}\ge 0$ and $p> 1$. Then $\pi\in \mathcal{M}(\mu,\nu)$ is an optimiser of $C(\mu,\nu)$ if and only if there exists a Borel set $\Gamma$ with $\pi(\Gamma) = 1$ such that the following holds:\\
if $\alpha$ is a measure on $\R^2$ with $|\text{supp}(\alpha)|<\infty$ and $\text{supp}(\alpha)\subseteq\Gamma$, then we have $$ \int c(x_1,x_2)\, \alpha(dx_1,dx_2) \le  \int c(x_1,x_2)\,\alpha'(dx_1,dx_2)$$
for every competitor $\alpha'$ of $\alpha$.
\end{theorem}
The proof of necessity was later simplified in \cite{Beiglbock:2019ufa} and essentially relies on the idea to select competitors in a measurable way. We give here an independent proof of sufficiency, which uses the stability result stated in Theorem \ref{thm:approx_simple}. The idea is to argue by contraposition: take any martingale measure $\pi\in \mathcal{M}(\mu,\nu)$, any set $\Gamma \subseteq \R^2$ such that $\pi(\Gamma)=1$ and assume $\pi$ is not optimal for $C(\mu,\nu)$. By an approximation result given in Lemma \ref{lemma:approx1} it is possible to find martingale measures $\pi^n$ finitely supported on $\Gamma$ such that $\lim_{n\to \infty} W_p(\pi^n, \pi)=0$. Let us denote the first marginal of $\pi^n$ by $\mu^n$ and the second marginal by $\nu^n$. As $\pi$ is not optimal and as $(\mu,\nu)\mapsto C(\mu,\nu)$ is continuous, there exists a number $n\in \N$ and a competitor $\pi'\in \mathcal{M}(\mu^n,\nu^n)$ with cost $\int c\,d\pi'$ strictly smaller than $\int c\,d\pi^n$, showing that $\Gamma$ is not finitely optimal.\\
In particular this enables us to show sufficiency for continuous functions of polynomial growth similarly to \cite{Griessler:2016wga}, who uses a splitting property for cyclically monotone sets and the decomposition into irreducible components established in \cite{Beiglbock:2016kt}.

\section{Generic approximation results}\label{sec:generic}
Let us now list several approximation results for the nested distance $W_p^{nd}$ and the Wasserstein distance $W_p$, which we will use throughout the paper. As these do not immediately follow from the isometric embedding of the space $(\mathcal{P}_p(\R),W^{nd}_p)$ into a Wasserstein space of nested distributions obtained in \cite{BackhoffVeraguas:2017ww}, we adopt a constructive self-contained approach. The proofs are mainly technical and are thus deferred to the appendix. We also refer to \cite{Alfonsi:2017wl} for an algorithmic approximation of $\mu$, $\nu$ by finitely supported measures, such that the convex order for the approximating measures is retained. Throughout this section we fix some $p\ge 1$.

\begin{lemma}\label{lemma:approx1}\mbox{}
Let $\mu,\nu\in \mathcal{P}_p(\R),$ $\pi \in \Pi(\mu, \nu)$  and $\kappa>0$.  Let $\Gamma\subseteq \R^2$ be a Borel set such that $\pi(\Gamma)=1$. 
\begin{enumerate}[(i)]
\item There exists a measure $\hat{\pi}$, which is finitely supported on $\Gamma$, such that $W_p^{nd}(\pi, \hat{\pi})\le \kappa$. Furthermore
\begin{align}\label{eq:nachtrag}
\int_{\{x_1\ge x\}} (x_2-x_1)\,\hat{\pi}(dx_1,dx_2)\ge \int_{\{x_1\ge x\}} (x_2-x_1)\,\pi(dx_1,dx_2)-\kappa
\end{align}
for all $x\in \mathrm{supp}(\hat{\pi}^1)$.
\item If $\pi\in \mathcal{M}(\mu,\nu)$, then $\hat{\pi}$ can be chosen to be a martingale measure.
\end{enumerate}
\end{lemma}

\begin{lemma}\label{cor:approx}
Let $\mu,\nu\in \mathcal{P}_p(\R),$ $\mu \preceq_c\nu$, $\pi\in \Pi(\mu,\nu)$ and $\kappa>0$. Then there exists a finitely supported measure $\bar{\pi}\in \Pi(\bar{\mu},\bar{\nu})$ such that $\bar{\mu}\preceq_c \bar{\nu}$ and $W_p^{nd}(\pi,\bar{\pi})\le \kappa$.
\end{lemma}

\begin{lemma}\label{lemma:convexity}
Let $\pi^n\in \Pi(\mu^n,\nu^n)$ be a sequence of measures in $\mathcal{P}_1(\R^2)$ and let $(\tilde{\pi}^n)_{n \in \N}$ be another sequence satisfying $\tilde{\pi}^n \in \Pi(\mu^n, \rho^n)$ for all $n\in \N$ and some $(\rho^n)_{n\in \N}$ with $\rho^n \in \mathcal{P}_1(\R)$ for all $n\in \N$. Then
\begin{align*}
W_1^{nd}\left(\frac{1}{n}\sum_{i=1}^n \pi^i,\frac{1}{n}\sum_{i=1}^n \tilde{\pi}^i\right)\le \frac{1}{n}\sum_{i=1}^n \int W_1(\pi^i_{x_1}, \tilde{\pi}^i_{x_1})\,\mu^i(dx_1).
\end{align*}
\end{lemma}


\begin{lemma}\label{lemma:approx2a}
Let $\mu,\nu,\tilde{\mu},\tilde{\nu}$ be elements of $\mathcal{P}_p(\R)$, $\mu\preceq_c\nu$ and let $\pi \in \mathcal{M}(\mu,\nu)$. Then there exists $\tilde{\pi}\in \Pi(\tilde{\mu},\tilde{\nu})$ such that $W_{p}^{p}(\pi, \tilde{\pi}) \leq W_{p}^{p}(\mu, \tilde{\mu})+W_{p}^{p}(\nu, \tilde{\nu})$ and 
\begin{align}\label{eq:gaoyue2}
\int \left|\int (y_2-y_1)\,\tilde{\pi}_{y_1}(dy_2)\right|\,\tilde{\mu}(dy_1)\le W_1(\mu, \tilde{\mu})+W_1(\nu,\tilde{\nu}).
\end{align}
In particular we have $W_p(\pi,\tilde{\pi})\le W_p(\mu,\tilde{\mu})+W_p(\nu,\tilde{\nu}) $.
\end{lemma}


\begin{lemma}\label{lemma:uniform_integrability}
Let $\mu,\nu\in \mathcal{P}_p(\R)$. Let $(\pi^n)_{n\in \N}$ be a sequence of  measures satisfying $\pi^n\in \Pi(\mu^n,\nu^n)$ for all $n\in \N$ and let $(\tilde{\pi}^n)_{n \in \N}$ be another sequence satisfying $\tilde{\pi}^n \in \Pi(\mu^n, \nu^n)$ for all $n\in \N$. Let $\lim_{n \to \infty} W_p(\mu^n, \mu)=\lim_{n \to \infty} W_p(\nu^n, \nu)=0$ and $\lim_{n \to \infty}W_1(\pi^n, \tilde{\pi}^n)=0$ . Then for any continuous function $c:\R \times \R\to \R$ satisfying $|c(x_1, x_2)|\le C(1+|x_1|^p+|x_2|^p)$ for some $C\ge 0$ we have
\begin{align*}
\lim_{n \to \infty} \left(\int c(x_1, x_2)\,\pi^n(dx_1, dx_2)- \int c(x_1, x_2)\, \tilde{\pi}^n(dx_1, dx_2)\right)=0,
\end{align*}
in particular $\lim_{n \to \infty} W_p(\pi^n, \tilde{\pi}^n)=0$.
\end{lemma}

\begin{lemma}\label{lem:aap_1}
For measures $\pi,\tilde{\pi}\in \mathcal{P}_1(\R^2)$ we have
\begin{align*}
\left|\int \left|\int (x_2-x_1)\,\pi_{x_1}(dx_2) \right|\,\pi^1(dx_1)- \int \left| \int (y_2-y_1)\,\tilde{\pi}_{y_1}(dy_2)\right|\,\tilde{\pi}^1(dy_1)\right|\le W_1^{nd}(\pi,\tilde{\pi}).
\end{align*}
Similarly
\begin{align*}
\left|\int \left(\int (x_2-x_1)\,\pi_{x_1}(dx_2) \right)^-\,\pi^1(dx_1)- \int \left( \int (y_2-y_1)\,\tilde{\pi}_{y_1}(dy_2)\right)^-\,\tilde{\pi}^1(dy_1)\right|\le W_1^{nd}(\pi,\tilde{\pi}).
\end{align*}
\end{lemma}

\begin{lemma}\label{lem:aap2}
Let $(\pi^n)_{n\in \N}$ be a sequence of measures in $\mathcal{P}_1(\R^2)$ and let $\pi\in \mathcal{P}_1(\R^2)$ with marginals $\mu\preceq \nu$, such that $\lim_{n\to \infty}W_1^{nd}(\pi^n, \pi)=0$. Assume that
\begin{align*}
\lim_{n\to \infty} \int \left( \int (x_2-x_1)\,\pi^n_{x_1}(dx_2)\right)^- (\pi^n)^1(dx_1)=0.
\end{align*}
Then $\pi\in \mathcal{M}(\mu,\nu)$.
\end{lemma}

\section{Proof of Proposition \ref{thm:dispersion}}\label{sec:thm_dispersion}

\subsection{Proof of Proposition \ref{thm:dispersion} for finitely supported measures}\label{sec:dispersion_finite}
Throughout this section we fix two finitely supported measures $\mu, \nu\in \mathcal{P}(\R)$ and a number $c\ge 0$. To prepare for the general case treated in Section \ref{sec:dispersion_general} we introduce the following generalised barycentre assumption for $c$:

\begin{assumption}[Generalised barycentre dispersion assumption]\label{def:dispersion2}
Let $\pi\in \Pi(\mu,\nu)$. For all $x \in \R$
\begin{align}\label{eq:generaliseddispersion}
\int_{\{x_1 \ge x\}} (x_2-x_1)\, \pi(dx_1, dx_2) \ge -c\mu([x,\infty)).
\end{align}
\end{assumption}

The main aim of the proof of Proposition \ref{thm:dispersion} is to show that under Assumption \ref{def:dispersion}, there exists a martingale measure $\tilde{\pi}\in \mathcal{M}(\mu,\nu)$ such that
\begin{align}\label{eq:aim}
\int W_1(\tilde{\pi}_{x_1}, \pi_{x_1})\,\mu(dx_1)\le \int \left|\int (x_2-x_1)\pi_{x_1}(dx_2)\right| \mu(dx_1)
\end{align}
holds. We remark that, together with Lemma \ref{lem:upperbound}, \eqref{eq:aim} implies Proposition \ref{thm:dispersion}, as the lhs of \eqref{eq:aim} majorizes $W_1^{nd}(\pi, \tilde{\pi})$. In this section we work with a finitely supported measure $\pi\in \Pi(\mu,\nu)$. We show that under Assumption \ref{def:dispersion2} we can find an ``approximate martingale measure" $\tilde{\pi}\in \Pi(\mu,\nu)$, such that
\begin{align*}
\int W_1(\tilde{\pi}_{x_1}, \pi_{x_1})\,\mu(dx_1)\le \int \left|\int (x_2-x_1)\pi_{x_1}(dx_2)+c\right| \mu(dx_1)
\end{align*}
holds. We will see that this implies Proposition \ref{thm:dispersion} for finitely supported $\pi\in \Pi(\mu,\nu)$. We will then extend the result to general measures $\pi$ in Section \ref{sec:dispersion_general}.\\
Let us start by giving the following definition:

\begin{definition}\label{def:j}
For a sequence of measures $(\pi^n)_{n\in \N}$ we set
\begin{align*}
X_{1}^{n,+} &\coloneq \left\{ x_1\in \mathrm{supp}(\mu) \ \bigg|\ \int (x_2-x_1)\,\pi^n_{ x_1}(dx_2)>-c\right\},\\
X_1^{n,0} &\coloneq  \left\{ x_1\in \mathrm{supp}(\mu) \ \bigg|\   \int (x_2-x_1)\,\pi^n_{x_1}(dx_2)= -c\right\},\\
X_{1}^{n,-} &\coloneq \left\{ x_1\in \mathrm{supp}(\mu) \ \bigg|\ \int (x_2-x_1)\,\pi^n_{x_1}(dx_2)<-c\right\}.
\end{align*}
for all $n\in \N$.
Furthermore
\begin{align*}
X_{2}^{n,+}&\coloneq \bigcup_{x_{1}^+\in X_{1}^{n,+}}\mathrm{supp}(\pi^n_{x_{1}^+}), \quad X_{2}^{n,0}\coloneq \bigcup_{x_{1}^0\in X_{1}^{n,0}} \mathrm{supp}(\pi^n_{x_{1}^0})\quad \text{and}\\
 X_{2}^{n,-}&\coloneq \bigcup_{x_{1}^-\in X_{1}^{n,-}}\mathrm{supp}(\pi^n_{ x_1^-}).
\end{align*}
\end{definition}

Now we fix a measure $\pi\in \Pi(\mu,\nu)$ satisfying Assumption \ref{def:dispersion2} for $c\ge0$, set $\pi^{(0)}\coloneq \pi$ and assume that $X_1^{0,-}\neq \emptyset$. The general idea formalised in this section will be to iteratively build measures $\pi^{(j)}\in \Pi(\mu,\nu)$ such that $X_1^{j,-}$ is decreasing to $\emptyset$. This is achieved by switching atoms in the supports of $(\pi_{x_{1}})_{x_1\in \text{supp}(\mu)}$ without changing the marginal constraints. More specifically we will use Algorithm \ref{alg:2} given below, which is written in a slightly elaborate form in order to prepare for the more complicate case of Algorithm \ref{alg:1} discussed in Section \ref{sec:W1}. In the definition of the algorithm, we will use Definition \ref{def:j} for the sequence of measures $(\pi^{(n)})_{n\in \N}$, which are constructed iteratively.

\begin{algorithm}\label{alg:2}
Set $j=0$.
\begin{enumerate}[(i)]
\item  Define $x_1^-(j)\coloneq \max(X_1^{j,-})$ and $x_2^-(j)\coloneq \min(\mathrm{supp}(\pi^{(j)}_{x_1^-(j)}))$. Set $x_1^+(j)\coloneq\max(X_1^{j,+})$ and $x_2^+(j)\coloneq \max(\mathrm{supp}(\pi^{(j)}_{x_1^{+}(j)}))$.
\item Define 
\begin{align}\label{eq:lambdadef}
\begin{split}
\lambda^{(j)}&\coloneq \Bigg[\mu(x_1^-(j))\min\left\{\tilde{\lambda}>0\ \bigg| \ \int (x_2-x_1^-(j))\,\pi^{(j)}_{x_1^-(j)}(dx_2)+\tilde{\lambda}(x_2^+(j)-x_2^-(j))\ge -c\right\}\\
&\qquad\wedge \pi^{(j)}(x_1^-(j),x_2^-(j))\\
&\qquad \wedge \mu(x_1^+(j))\min\left\{\tilde{\lambda}>0\ \bigg| \ \int (x_2-x_1^+(j))\,\pi^{(j)}_{x_1^+(j)}(dx_2)+\tilde{\lambda}(x_2^-(j)-x_2^+(j))\le -c\right\}\\
&\qquad\wedge\pi^{(j)}(x_1^+(j),x_2^+(j))\Bigg]\cdot (x_2^+(j)-x_2^-(j)).
\end{split}
\end{align}
\item Define $\rho^{(j)}\in \mathcal{P}(\R^2\times \R^2)$ via $$\rho^{(j)}(dx_1, dx_2, dy_1, dy_2)\coloneq (\delta_{x_1}(dy_1)\otimes  \mu(dx_1)) \otimes \rho^{(j)}_{(x_1, y_1)}(dx_2,dy_2),$$ where
\begin{align*}
\rho^{(j)}_{(x_1,x_1)}&\coloneq(x_2,x_2)_\ast \pi^{(j)}_{x_1}\quad \text{ for all }x_1\in \mathrm{supp}(\mu)\setminus \{x_1^-(j),x_1^+(j)\}\\
\rho^{(j)}_{(x_1^-(j),x_1^-(j))}&\coloneq(x_2,x_2)_\ast\pi^{(j)}_{x_1^-(j)}\\
&+\frac{\lambda^{(j)}}{\mu(x_1^-(j))(x_2^+(j)-x_2^-(j))}(\delta_{(x_2^-(j),x_2^+(j))}-\delta_{(x_2^-(j),x_2^-(j))})\\
\rho^{(j)}_{(x_1^+(j),x_1^+(j))}&\coloneq(x_2,x_2)_\ast\pi^{(j)}_{x_1^+(j)}\\
&+\frac{\lambda^{(j)}}{\mu(x_1^+(j))(x_2^+(j)-x_2^-(j))}(\delta_{(x_2^+(j),x_2^-(j))}-\delta_{(x_2^+(j),x_2^+(j))}).
\end{align*}
Set $\pi^{(j+1)}(dy_1, dy_2)\coloneq \int \rho^{(j)}(dx_1, dx_2, dy_1, dy_2)$.
\end{enumerate}
Now set $j=j+1$ and iterate \textit{(i)}-\textit{(iii)}. Terminate if $X_1^{j,-}=\emptyset$.
\end{algorithm}

\begin{remark}\label{rem:explained}
We note that the above algorithm formalises the intuition of switching barycentre mass $\lambda^{(j)}$ between the points $(x^{-}_1(j), x_2^-(j))$ and $(x^{+}_1(j), x_2^+(j))$. In the definition of $\lambda^{(j)}$ we pay tribute to the following constraints:
\begin{itemize}
\item After switching masses, $x^{-}_1(j)\in X_1^{j+1,-}\cup X_1^{j+1,0}$ should hold: this explains the first term in \eqref{eq:lambdadef}.
\item We cannot switch more probability mass than available at the point $(x_1^-(j), x_2^-(j))$, which is  exactly $\pi^{(j)}(x_1^-(j),x_2^-(j))$: this explains the second term in in \eqref{eq:lambdadef}.
\item After switching masses, $x^{+}_1(j)\in X_1^{j+1,+}\cup X_1^{j+1,0}$ should hold: this explains the third term in \eqref{eq:lambdadef}.
\item  We cannot switch more probability mass than available at the point $(x_1^+(j), x_2^+(j))$, which is  exactly $\pi^{(j)}(x_1^+(j),x_2^+(j))$: this explains the fourth term in in \eqref{eq:lambdadef}.
\end{itemize}
In particular it is important to realise that one of these constraints will be strictly binding, i.e. after carrying out step $j$ we have one (or more) of the following:
\begin{itemize}
\item We have $x^{-}_1(j)\in X_1^{j+1,0}$.
\item We have $\pi^{(j+1)}(x_1^-(j),x_2^-(j))=0$, which means that we have deleted the left-most point $x_2^-(j)$ of the support of $\pi^{(j)}_{x_1^-(j)}$ (we might have added a new point $x_2^+(j)$ to the support of $\pi^{(j)}_{x_1^-(j)}$. We will see in Lemma \ref{lem:W1_1} that $x_2^+(j)>x_2^-(j)$).
\item We have $x^{+}_1(j)\in X_1^{j+1,0}$.
\item  We have $\pi^{(j+1)}(x_1^+(j),x_2^+(j))=0$,  which means that we have deleted the right-most point of the support of $\pi^{(j)}_{x_1^+(j)}$ (we might have added a new point $x_2^-(j)$ to the support of $\pi^{(j)}_{x_1^+(j)}$. We will see in Lemma \ref{lem:W1_1} that $x_2^-(j)<x_2^+(j)$).
\end{itemize}
\end{remark}

\begin{definition}
We denote the number of steps until termination of Algorithm \ref{alg:2} by $N\in \N\cup \{\infty\}$.
\end{definition}

We note that $\rho^{(j)}$ defines a (bicausal) transport plan between $\pi^{(j)}$ and $\pi^{(j+1)}$, which we will later use to bound $W_1^{nd}(\pi^{(0)}, \pi^{(N)})$.

\begin{remark}\label{rem:rho2}
To simplify notation, we will however mostly work with the measure $\pi^{(j+1)}$ in the proofs below. In particular we will use the fact that $\pi^{(j+1)}$  has first marginal $\mu$ and 
\begin{align*}
\pi^{(j+1)}_{x_1}&=\pi^{(j)}_{x_1}\quad \text{ for all }x_1\in \mathrm{supp}(\mu)\setminus \{x_1^-(j),x_1^+(j)\}\\
\pi^{(j+1)}_{x_1^-(j)}&=\pi^{(j)}_{x_1^-(j)}+\frac{\lambda^{(j)}}{\mu(x_1^-(j))(x_2^+(j)-x_2^-(j))}(\delta_{x_2^+(j)}-\delta_{x_2^-(j)})\\
\pi^{(j+1)}_{x_1^+(j)}&=\pi^{(j)}_{x_1^+(j)}+\frac{\lambda^{(j)}}{\mu(x_1^+(j))(x_2^+(j)-x_2^-(j))}(\delta_{x_2^-(j)}-\delta_{x_2^+(j)}).
\end{align*}
Nevertheless the definition of $\rho^{(j)}$ will be crucial for the estimation of $W_1^{nd}(\pi, \pi^{(j)})$.
\end{remark}

\begin{lemma}\label{lem:W1_1}
In every step $0\le j\le N$ of Algorithm \ref{alg:2} we have $x_1^+(j)> x_1^-(j)$ and  $x_2^+(j)> x_2^-(j)$. Furthermore the measures $\pi^{(j)}$ satisfy the generalised barycentre dispersion assumption \ref{def:dispersion2} for $c$ and $\pi^{(j)}\in \Pi(\mu,\nu)$. 
\end{lemma}

\begin{proof}
By Assumption \ref{def:dispersion2} there exists $x_1>x_1^-(0)$ such that $x_1\in X_1^{0,+}$, in particular $x_1^+(0)>x_1^-(0)$, which also implies $x_2^+(0)>x_1^+(0)-c>x_1^-(0)-c> x_2^-(0)$. We now check that the measure $\pi^{(1)}$ satisfies Assumption \ref{def:dispersion2}. Note that by Algorithm \ref{alg:2}, equation \eqref{eq:generaliseddispersion} trivially holds for $\pi^{(1)}$ and $x>x_1^+(0)$. By definition of $\lambda^{(0)}$ we also have
\begin{align*}
\int (x_2-x_1^+(0))\,\pi^{(1)}_{x_1^+(0)}(dx_2)\ge -c,
\end{align*}
which implies that \eqref{eq:generaliseddispersion} holds for all $x\ge x_1^+(0)$ and by definition of $x_1^-(0)$ then also for all $x> x_1^-(0)$. Next we note that for all $x\le x_1^-(0)$
\begin{align*}
\int_{\{x_1\ge x\}} (x_2-x_1)\,\pi^{(1)}(dx_1,dx_2)&=\int_{\{x_1\ge x\}} (x_2-x_1)\,\pi^{(0)}(dx_1,dx_2)\\
&+\mu(x_1^-(0))\frac{\lambda^{(0)}}{\mu(x_1^-(0))(x_2^+(0)-x_2^-(0))}(x_2^+(0)-x_2^-(0)) \\
&+\mu(x_1^+(0))\frac{\lambda^{(0)}}{\mu(x_1^+(0))(x_2^+(0)-x_2^-(0))}(x_2^-(0)-x_2^+(0)) \\
&=\int_{\{x_1\ge x\}} (x_2-x_1)\,\pi^{(0)}(dx_1,dx_2)\ge -c\mu([x,\infty)),
\end{align*}
so the claim follows. Lastly we show that $\pi^{(1)}\in \Pi(\mu,\nu)$. As noted before, the first marginal of $\pi^{(1)}$ is $\mu$, so we only need to check the second marginal. For this we take a Borel set $A\subseteq \R$ and calculate
\begin{align*}
\int_{\R\times A} \pi^{(1)}(dx_1, dx_2)&=\int_{\R} \int_A\pi^{(1)}_{x_1}(dx_2) \,\mu(dx_1)\\
&=\int_{\R} \int_A\pi^{(0)}_{x_1}(dx_2) \,\mu(dx_1)\\
&+\mu(x_1^-(0))\frac{\lambda^{(0)}}{\mu(x_1^-(0))(x_2^+(0)-x_2^-(0))}(\delta_{x_2^+(0)}(A)-\delta_{x_2^-(0)}(A))\\
&+\mu(x_1^+(0))\frac{\lambda^{(0)}}{\mu(x_1^+(0))(x_2^+(0)-x_2^-(0))}(\delta_{x_2^-(0)}(A)-\delta_{x_2^+(0)}(A))\\
&=\int_{\R} \int_A\pi^{(0)}_{x_1}(dx_2) \,\mu(dx_1)\\
&+\frac{\lambda^{(0)}}{x_2^+(0)-x_2^-(0)}(\delta_{x_2^+(0)}(A)-\delta_{x_2^-(0)}(A)+\delta_{x_2^-(0)}(A)-\delta_{x_2^+(0)}(A))\\
&=\int_{\R} \int_A\pi^{(0)}_{x_1}(dx_2) \,\mu(dx_1)=\int_{\R\times A} \pi^{(0)}(dx_1, dx_2).
\end{align*}
Applying the above arguments inductively concludes the proof.
\end{proof}

Recalling Remark \ref{rem:explained} the following lemma is immediate:

\begin{lemma}\label{lem:W1_2}
Algorithm \ref{alg:2} terminates after at most $N \le |\mathrm{supp}(\mu)|(1+|\mathrm{supp}(\nu)|)$ steps.
\end{lemma}

\begin{proof}
For $j\in \N_0:=\N\cup \{0\}$ and all $x_1\in X_1^{0,+}$ we define the set 
\begin{align*}
I^{(j)}(x_1):=\{ x_2\in \mathrm{supp}(\nu) \ | \ x_2\le \max(\mathrm{supp}(\pi^{(j)}_{x_1}) )\}.
\end{align*}
Similarly for all $x_1\in X_1^{0,-}$ we set
\begin{align*}
I^{(j)}(x_1):=\{ x_2\in \mathrm{supp}(\nu) \ | \ x_2\ge \min(\mathrm{supp}(\pi^{(j)}_{x_1}) )\}.
\end{align*}
Let us explicitly point out here that $\text{supp}(\pi^{(j)}_{x_1})\subseteq I^{(j)}(x_1)$, where the inclusion is typically strict.
By the definition of $\lambda^{(j)}$ in Algorithm \ref{alg:2} and Remark \ref{rem:explained} we note that in every step $j$ at least one of the following three cases occurs:
\begin{enumerate}[(i)]
\item $|X_1^{j,+}|-|X_1^{j+1,+}|=1$ or $|X_1^{j,-}|-|X_1^{j+1,-}|=1$.
\item $|I^{(j)}(x_1^+(j))|-|I^{(j+1)}(x_1^+(j))|\ge 1$.
\item $|I^{(j)}(x_1^-(j))|-|I^{(j+1)}(x_1^-(j))|\ge 1$.
\end{enumerate}
Combining this observation with the fact that again by the definition of $\lambda^{(j)}$ in Algorithm \ref{alg:2} and Remark \ref{rem:explained} we have $X_1^{j,+}\subseteq X_1^{0,+}$ and $X_1^{j,-}\subseteq X_1^{0,-}$ as well as $I^{(j)}(x_1)\subseteq I^{(0)}(x_1)$ for all $x_1\in X_1^{0,+}\cup X_1^{0,-}$ we conclude that the number of steps $N$ is bounded by 
\begin{align*}
|X_1^{0,+}|+|X_1^{0,-}|+\sum_{x_1\in X_1^{0,-}} |I^{(0)}(x_1)| +\sum_{x_1\in X_1^{0,+}} |I^{(0)}(x_1)|&\le |\mathrm{supp}(\mu)|+|\mathrm{supp}(\mu)||\mathrm{supp}(\nu)|\\
&=|\mathrm{supp}(\mu)|(1+|\mathrm{supp}(\nu)|).
\end{align*}
This concludes the proof.
\end{proof}

\begin{proof}[Proof of Proposition \ref{thm:dispersion} for finitely supported $\pi\in \Pi(\mu,\nu)$]
Given Lemmas \ref{lem:W1_1} and \ref{lem:W1_2} all that is left to show is that a slightly generalised version of \eqref{eq:aim} holds for $\pi^{(N)}$, namely
\begin{align*}
\int W_1(\pi^{(N)}_{x_1}, \pi_{x_1})\,\mu(dx_1)\le \int \left|\int (x_2-x_1)\,\pi_{x_1}(dx_2) +c\right| \mu(dx_1).
\end{align*}
Using the triangle inequality and Remark \ref{rem:rho2} we indeed have
\begin{align}\label{eq:triangle1}
\begin{split}
W_1^{nd}(\pi^{(N)}, \pi)&\le \int W_1(\pi^{(N)}_{x_1}, \pi_{x_1})\,\mu(dx_1)\\
&\le \sum_{j=1}^N \int W_1(\pi^{(j)}_{x_1}, \pi^{(j-1)}_{x_1})\, \mu(dx_1)\\
&\le  \sum_{j=1}^N \mu(x_1^-(j-1))\frac{\lambda^{(j-1)}}{\mu(x_1^-(j-1))(x_2^+(j-1)-x_2^-(j-1))}\\
&\qquad\cdot|x_2^+(j-1)-x_2^-(j-1)| \\
&+ \sum_{j=1}^N\mu(x_1^+(j-1))\frac{\lambda^{(j-1)}}{\mu(x_1^+(j-1))(x_2^+(j-1)-x_2^-(j-1))}\\
&\qquad\cdot|x_2^-(j-1)-x_2^+(j-1)| \\\
&= \sum_{j=1}^N 2\lambda^{(j-1)}.
\end{split}
\end{align}
On the other hand, by definition of $\lambda^{(j-1)}$,
\begin{align}\label{eq:triangle2}
\int \left|\int (x_2-x_1)\,\pi^{(j-1)}_{x_1}(dx_2)+c \right| \mu(dx_1)- \int \left|\int (x_2-x_1)\,\pi^{(j)}_{x_1}(dx_2)+c \right| \mu(dx_1)=2\lambda^{(j-1)}.
\end{align}
Combining \eqref{eq:triangle1} and \eqref{eq:triangle2}
\begin{align*}
W_1^{nd}(\pi^{(N)}, \pi)&\le \int W_1(\pi^{(N)}_{x_1}, \pi_{x_1})\,\mu(dx_1)\\
&\le \sum_{j=1}^N \int \left|\int (x_2-x_1)\,\pi^{(j-1)}_{x_1}(dx_2) +c\right| \mu(dx_1)- \int \left|\int (x_2-x_1)\,\pi^{(j)}_{x_1}(dx_2) +c\right| \mu(dx_1)\\ &\le \int \left|\int (x_2-x_1)\,\pi^{(0)}_{x_1}(dx_2)+c \right| \mu(dx_1)=\int \left|\int (x_2-x_1)\,\pi_{x_1}(dx_2) +c\right| \mu(dx_1),
\end{align*}
which shows the claim.\\
Lastly using Algorithm \ref{alg:2} in the special case $c=0$ the claim now follows for finitely supported measures, as $\mu\preceq_c \nu$ and $X_1^{N,-}=\emptyset$ implies $X_1^{N,+}=\emptyset$. Thus $\pi^{(N)}$ is a martingale.
\end{proof}

\subsection{Proof of Proposition \ref{thm:dispersion} for general $\pi\in \Pi(\mu,\nu)$}\label{sec:dispersion_general}
Throughout this section we fix two measures $\mu,\nu \in \mathcal{P}_p(\R)$ such that $\mu\preceq_c \nu$.
We now extend the results from Section \ref{sec:dispersion_finite} to a general coupling $\pi\in \Pi(\mu,\nu)$ satisfying the barycentre dispersion assumption \ref{def:dispersion}. 

\begin{proof}[Proof of Proposition \ref{thm:dispersion} for general $\pi\in \Pi(\mu,\nu)$]
By Lemma \ref{lemma:approx1} applied with $\kappa_n=1/n$ there exists a sequence of finitely supported measures $(\pi^n)_{n \in \N}$ with 
\begin{align*}
\lim_{n \to \infty} W_1^{nd} (\pi^n, \pi)=0
\end{align*}
and 
\begin{align}\label{eq:display}
\int_{\{x_1 \ge x\}} (x_2-x_1) \, \pi^n(dx_1, dx_2)\ge -1/n
\end{align} 
for all $x \in \R$ and for all $n\in \N$. We remark that \eqref{eq:display} follows from the barycentre dispersion assumption \ref{def:dispersion} for $\pi$. Let us denote the marginals of $\pi^n$ by $\mu^n$ and $\nu^n$. In particular $\pi^n$ satisfies Assumption \ref{def:dispersion2} with $c_n=1/n$. Applying Algorithm \ref{alg:2} and using the proof of Proposition \ref{thm:dispersion} for finitely supported measures we can find a sequence of measures $(\pi^n_{mr})_{n\in \N}$ such that for all $n \in \N$ we have $\pi_{mr}^n \in \Pi(\mu^n, \nu^n)$,
\begin{align}\label{eq:fhh}
\int W_1(\pi_{mr,x_1}^n, \pi^n_{x_1})\,\mu^n(dx_1)\le \int \left|\int (x_2-x_1)\,\pi^n_{x_1}(dx_2) +\frac{1}{n}\right|\mu^n(dx_1)
\end{align}
and 
\begin{align}\label{eq:lowerbound}
 \int (x_2-x_1)\, \pi^n_{mr, x_1}(dx_2) \ge -1/n
\end{align}
for all $x_1\in \mathrm{supp}(\mu^n)$. We now want to extend the (only $\mu^n$-a.s. defined) disintegrations $x_1\mapsto \pi_{x_1}^n$ and $x_1\mapsto \pi_{mr,x_1}^n$ to the real line. For this we recall the functions $f^{\kappa}$ defined in \eqref{eq:fkappa} in the proof of Lemma \ref{lemma:approx1}.(i) and set for $\kappa=1/n$
\begin{align}\label{eq:extension}
\pi_{x_1}^n := \pi_{f^{1/n}(x_1)}^n, \qquad \pi^n_{mr,x_1}=\pi_{mr,f^{1/n}(x_1)}^n.
\end{align}
Next we define $\bar{\pi}^n:= \mu\otimes\pi_{x_1}^n$ as well as $\bar{\pi}_{mr}^n:= \mu\otimes\pi_{mr,x_1}^n$ and note that 
\begin{align*}
W_1(\pi_{x_1}^n, \pi^n_{f^{1/n}(x_1)}) =0=W_1(\pi_{mr,x_1}^n, \pi^n_{mr, f^{1/n}(x_1)})
\end{align*}
by \eqref{eq:extension}. But as $\mu^n=f^{1/n}(x_1)_*\mu$ this immediately implies
\begin{align*}
W_1^{nd}(\pi^n, \bar{\pi}^n)&\le \int \left( |x_1-f^{1/n}(x_1)|+W_1(\pi_{x_1}^n, \pi^n_{f^{1/n}(x_1)})\right)\,\mu(dx_1)\\
&= \int |x_1-f^{1/n}(x_1)|\,\mu(dx_1)
\end{align*}
and thus $\lim_{n\to \infty} W_1^{nd}(\pi^n, \bar{\pi}^n)=0$. The same argument shows $\lim_{n\to \infty} W_1^{nd}(\pi^n_{mr}, \bar{\pi}_{mr}^n)=0$. In particular the marginals of $(\bar{\pi}^n)_{n\in \N}$ and $(\bar{\pi}^n_{mr})_{n\in \N}$ still converge to $\mu$ and $\nu$ respectively and we still have from \eqref{eq:fhh} and \eqref{eq:lowerbound} that
\begin{align}\label{eq:fhh1a}
\begin{split}
\int W_1(\bar{\pi}_{mr,x_1}^n, \bar{\pi}^n_{x_1})\,\mu(dx_1)&=\int W_1(\pi_{mr,x_1}^n, \pi^n_{x_1})\,\mu^n(dx_1)\\
&\le \int \left|\int (x_2-x_1)\,\pi^n_{x_1}(dx_2) +\frac{1}{n}\right|\mu^n(dx_1)
\end{split}
\end{align}
as well as  
\begin{align}\label{eq:lowerbound1}
 \int (x_2-f^{1/n}(x_1))\, \bar{\pi}^n_{mr, x_1}(dx_2) \ge -1/n.
\end{align}
We also note that $(\bar{\pi}^n_{mr})^1=\mu$ for all $n\in \N$ and thus 
\begin{align}\label{eq:marginal}
\left(\frac{1}{n} \sum_{i=1}^n \bar{\pi}^i_{mr}\right)_{x_1}=\frac{1}{n} \sum_{i=1}^n \bar{\pi}^i_{mr,x_1}
\end{align}
for all $n\in \N$.\\
The introduction of $\bar{\pi}^n_{mr}$ has a specific purpose: it enables us to apply precompactness results for Young measures, see e.g \cite[Theorem 3.15, p.18]{balder1995lectures}; in particular there exists a disintegration $x_1\mapsto \bar{\pi}_{mr,x_1}$  such that (after taking a subsequence without relabelling) the measures $$\left(\frac{1}{n} \sum_{i=1}^n \bar{\pi}_{mr, x_1}^i\right)_{n\in \N}$$ converge weakly to $\bar{\pi}_{mr,x_1}$ for $\mu$-a.e. $x_1\in \R$. Setting $\pi_{mr} \coloneq \mu \otimes \bar{\pi}_{mr,x_1}$ this implies in particular that $\bar{\pi}_{mr}^n$ converges weakly to $\pi_{mr}$ by \cite[Cor. 3.14]{balder1995lectures}, and thus $\pi_{mr}\in \Pi(\mu,\nu)$. Furthermore
\begin{align*}
\limsup_{n\to \infty} \int\left( \inf_{\gamma^2\in \Pi\left(\frac{1}{n}\sum_{i=1}^n \bar{\pi}^{i}_{mr,x_1}, \pi_{mr,x_1}\right)} \int (|x_2-y_2|\wedge 1)\, \gamma^2(dx_2,dy_2)\right) \mu(dx_1)=0
\end{align*}
by the dominated convergence theorem. As $\lim_{n\to \infty} W_1^{nd}(\pi^n_{mr}, \bar{\pi}_{mr}^n)=0$ it is easy to see that also the first moments of the marginals of $(\frac{1}{n}\sum_{i=1}^n \bar{\pi}_{mr}^i)_{n\in \N}$ converge. Now we conclude by \cite[Lemma 1.4]{backhoff2020all}  that in fact $\lim_{n\to \infty}W^{nd}_1(\frac{1}{n} \sum_{i=1}^n \bar{\pi}_{mr}^i,\pi_{mr})=0$.\\
We now show that $\pi_{mr}$ is actually a martingale measure, i.e. $\pi_{mr}\in \mathcal{M}(\mu,\nu)$. Using the triangle inequality for $(\cdot)^-$ and summing over $i=1,\dots,n$ in \eqref{eq:lowerbound1} we obtain 
\begin{align*}
\left( \int (x_2-x_1)\,\left(\frac{1}{n}\sum_{i=1}^n \bar{\pi}_{mr,x_1}^i\right)(dx_2)\right)^-
&\le \frac{1}{n}\sum_{i=1}^n \left( \int (x_2-x_1)\,\bar{\pi}_{mr,x_1}^i(dx_2)\right)^-\\
&\le \frac{1}{n}\sum_{i=1}^n \Bigg( \left( \int (x_2-f^{1/i}(x_1))\,\bar{\pi}_{mr,x_1}^i(dx_2)\right)^- \\
&\qquad+\int |f^{1/i}(x_1))-x_1|\,\bar{\pi}_{mr,x_1}^i(dx_2)\Bigg)\\
&\le \frac{1}{n}\sum_{i=1}^n\left( \frac{1}{i} + \left|f^{1/i}(x_1))-x_1\right|\right).
\end{align*}
Using \eqref{eq:marginal} we thus conclude that
\begin{align*}
\lim_{n\to \infty} \int \left(  \int (x_2-x_1)\, \left(\frac{1}{n}\sum_{i=1}^n\bar{\pi}^{i}_{mr}\right)_{ x_1}(dx_2) \right)^- \left( \frac{1}{n}\sum_{i=1}^n\bar{\pi}^i_{mr}\right)^{1}(dx_1)=0.
\end{align*}
An application of Lemma \ref{lem:aap2} then shows $\pi_{mr} \in \mathcal{M}(\mu,\nu)$.\\
Lastly we aim to show that 
\begin{align*}
W_1^{nd}\left(\pi_{mr},\pi\right)&\le \int \left| \int (x_2-x_1)\, \pi_{x_1}(dx_2)\right|\,\mu(dx_1).
\end{align*}
Using Lemma \ref{lemma:convexity} we have
\begin{align*}
W_1^{nd}\left(\frac{1}{n}\sum_{i=1}^n \bar{\pi}_{mr}^i,\frac{1}{n}\sum_{i=1}^n \bar{\pi}^i\right)&\le \frac{1}{n}\sum_{i=1}^n \int W_1(\bar{\pi}_{mr,x_1}^i, \bar{\pi}^i_{x_1})\,\mu(dx_1) \\
&\stackrel{\eqref{eq:fhh1a}}{\le} \frac{1}{n}\sum_{i=1}^n \int \left|\int (x_2-x_1)\,\pi^i_{x_1}(dx_2) +\frac{1}{i}\right|\mu^i(dx_1).
\end{align*}
Thus
\begin{align*}
W_1^{nd}\left(\pi_{mr},\pi\right)&\le \limsup_{n\to \infty} W_1^{nd}\left(\pi_{mr}, \frac{1}{n}\sum_{i=1}^n \bar{\pi}_{mr}^i \right)  + \limsup_{n\to \infty}W_1^{nd}\left(\frac{1}{n}\sum_{i=1}^n \bar{\pi}_{mr}^i,\frac{1}{n}\sum_{i=1}^n \bar{\pi}^i\right)\\ &\quad + \limsup_{n\to \infty}W_1^{nd}\left(\frac{1}{n}\sum_{i=1}^n \bar{\pi}^i, \pi\right) \\
&\le \limsup_{n\to \infty}  \frac{1}{n}\sum_{i=1}^n \int \left|\int (x_2-x_1)\,\pi^i_{x_1}(dx_2) +\frac{1}{i}\right|\mu^i(dx_1),
\end{align*}
where we have used
\begin{align*}
\limsup_{n\to \infty} W_1^{nd} \left(\frac{1}{n}\sum_{i=1}^n \bar{\pi}^i, \pi\right)&\le \limsup_{n\to \infty} \frac{1}{n}\sum_{i=1}^n \int W_1(\bar{\pi}_{x_1}^i, \pi_{x_1})\,\mu(dx_1)\\
&=\limsup_{n\to \infty} \frac{1}{n}\sum_{i=1}^n \int W_1(\pi_{f^{1/i}(x_1)}^i, \pi_{x_1})\,\mu(dx_1)=0,
\end{align*}
which holds again by the choice of $\mu^n=(f^{1/n}(x_1))_*\mu$ as in the proof of Lemma \ref{lemma:approx1}.(i). 
As $\lim_{n\to \infty} W_1^{nd}(\pi^n, \pi)=0$ we can apply Lemma \ref{lem:aap_1} to see that the last expression is equal to $$\int \left| \int (x_2-x_1)\, \pi_{x_1}(dx_2)\right|\,\mu(dx_1).$$ This concludes the proof.
\end{proof}

\section{Proof of Theorem \ref{thm:W1}}\label{sec:W1}

Throughout this section we assume $\mu \preceq_c \nu$ and $\mu,\nu\in \mathcal{P}_1(\R)$. Furthermore we make use of the notation introduced in Section \ref{sec:thm_dispersion} for the case $c=0$, i.e. we write
\begin{align*}
X_{1}^{n,+} &= \left\{ x_1\in \text{supp}(\mu) \ \bigg|\ \int (x_2-x_1)\,\pi^n_{ x_1}(dx_2)>0\right\},\\
X_1^{n,0} &=  \left\{ x_1\in \text{supp}(\mu) \ \bigg|\ \int (x_2-x_1)\,\pi^n_{x_1}(dx_2)=0\right\},\\
X_{1}^{n,-} &= \left\{ x_1\in \text{supp}(\mu) \ \bigg|\ \int (x_2-x_1)\,\pi^n_{x_1}(dx_2)<0\right\}
\end{align*}
for a sequence of measures $(\pi^n)_{n\in \N}$.
We also recall
\begin{align*}
X_{2}^{n,+}&= \bigcup_{x_{1}^+\in X_{1}^{n,+}}\mathrm{supp}(\pi^n_{x_{1}^+}), \quad X_{2}^{n,0}= \bigcup_{x_{1}^0\in X_{1}^{n,0}} \mathrm{supp}(\pi^n_{x_{1}^0})\quad \text{and}\\
 X_{2}^{n,-}&= \bigcup_{x_{1}^-\in X_{1}^{n,-}}\mathrm{supp}(\pi^n_{ x_1^-}).
\end{align*}

\subsection{Proof of Theorem \ref{thm:W1} for finitely supported $\pi\in \Pi(\mu,\nu)$ with common compact support}\label{sec:w1_finite}
We prove Theorem \ref{thm:W1} via several lemmas. We first argue for finitely supported $\pi\in \Pi(\mu, \nu)$ and write $\pi^{(0)}=\pi$. To motivate the construction in this section, let us first consider a particular case:
\begin{lemma}\label{lem:aux0}
Assume that $\pi=\pi^{(0)}\in \Pi(\mu, \nu)$ is finitely supported and that $X_1^{0,0}=\emptyset$. Then there exist pairs $(x_{1}^-, x_{1}^+)\in X_{1}^{0,-}\times X_{1}^{0,+}$ and $(x_{2}^-,x_{2}^+)\in \text{supp}(\pi^{(0)}_{x_{1}^-})\times \text{supp}(\pi^{(0)}_{x_{1}^+})$ such that $x_{2}^-<x_{2}^+.$
\end{lemma}

\begin{proof}
Let us assume towards a contradiction that the claim does not hold. We first note that for all $x_{1}^+\in X_{1}^{0,+}$ there is  $x_{2}^+\in \text{supp}(\pi^{(0)}_{x_{1}^+})$ such that $x_{2}^+> x_{1}^+$ and correspondingly for all $x_{1}^-\in X_{1}^{0,-}$ there is  $x_{2}^-\in \text{supp}(\pi^{(0)}_{x_{1}^-})$ such that $x_{2}^-< x_{1}^-$. This implies that 
\begin{align}\label{eq 1}
\max\{x_{1}^+ \ : \ x_{1}^+\in X_{1}^{0,+}\}&<\max\{x_{2}^+ \ : \ x_{2}^+\in X_{2}^{0,+}\} \le \min \{x_{2}^- \ : \ x_{2}^-\in X_{2}^{0,-}\}\\
&< \min \{x_{1}^- \ : \ x_{1}^-\in X_{1}^{0,-}\}.\nonumber
\end{align}
We conclude that 
$$\{x_1\in \mathrm{supp}(\mu)\ : \ x_1\le \max( X_{2}^{0,+})\}=X_1^{0,+},$$
$$\{x_2\in \mathrm{supp}(\nu)\ : \ x_2\le \max(X_{2}^{0,+})\}=X_2^{0,+}$$
and $X_2^{0,+}\cap X_2^{0,-}\subseteq \{\max(X_{2}^{0,+})\}$.
Furthermore, by definition of $X_1^{0,+}$ we have
\begin{align}\label{eq:idiot}
\int_{X_1^{0,+}\times \R} (x_2-x_1) \pi^{(0)}(dx_1,dx_2)= \int_{X_1^{0,+}} \int  (x_2-x_1) \pi^{(0)}_{x_1}(dx_2)\mu(dx_1)>0.
\end{align}
Also from \eqref{eq 1} and the defining property of disintegrations \eqref{eq:idiot2} we conclude for a generic Borel measurable function $f:\R\to \R$ that 
\begin{align}\label{eq:disint}
\begin{split}
\int_{X_2^{0,+}} f(x_2)\,\nu(dx_2)&=\int \int_{X_2^{0,+}} f(x_2)\, \pi^{(0)}(dx_1, dx_2)=\int \int_{X_2^{0,+}} f(x_2)\, \pi_{x_1}^{(0)}(dx_2)\,\mu(dx_1)\\
&=\int_{X_1^{0,+}} \int_{X_2^{0,+}} f(x_2)\,\pi_{x_1}^{(0)}(dx_2)\,\mu(dx_1)\\
&+\int_{(X_1^{0,+})^c} \int_{X_2^{0,+}} f(x_2)\, \pi_{x_1}^{(0)}(dx_2)\,\mu(dx_1)\\
&= \int_{X_1^{0,+}} \int f(x_2)\,\pi_{x_1}^{(0)}(dx_2)\,\mu(dx_1)\\
&+\int_{X_1^{0,-}} f(\max(X_{2}^{0,+}))\pi_{x_1}^{(0)}(\max(X_{2}^{0,+}))\,\mu(dx_1),
\end{split}
\end{align}
noting that $(X_1^{0,+})^c\cap \text{supp}(\mu)=X_1^{0,-}$ and recalling that $X_2^{0,+}\cap X_2^{0,-}\subseteq \{\max(X_{2}^{0,+})\}$.
Let us define $g(x)\coloneq(\max (X_{2}^{0,+})-x)^+$. Then by the above
\begin{align*}
\int g(x_1)\,\mu(dx_1)&=\int_{X_1^{0,+}} (\max(X_{2}^{0,+})-x_1)^+\,\mu(dx_1)\\
&=\int_{X_1^{0,+}} (\max(X_{2}^{0,+})-x_1)\,\mu(dx_1)\\
&\stackrel{\eqref{eq:idiot}}{>} \int_{X_1^{0,+}} \int (\max(X_{2}^{0,+})-x_2)\,\pi^{(0)}_{x_1}(dx_2)\mu(dx_1)\\
&\stackrel{\eqref{eq:disint}}{=} \int_{X_2^{0,+}} (\max(X_{2}^{0,+})-x_2)\,\nu(dx_2)\stackrel{\eqref{eq 1}}{=}\int g(x_2)\,\nu(dx_2).
\end{align*}
Noting that $g$ is convex, this contradicts $\mu \preceq_c \nu$ and shows the claim.
\end{proof}

We are now ready for the general case:
\begin{lemma}\label{lem:aux1}
Let $j\in \N_0=\N\cup \{0\}$, assume that $\pi^{(j)}\in \Pi(\mu, \nu)\setminus \mathcal{M}(\mu,\nu)$ is finitely supported and there exist no pairs $(x_{1}^-, x_{1}^+)\in X_{1}^{j,-}\times X_{1}^{j,+}$ and $\left(x_{2}^-,x_{2}^+\right)\in \text{supp}(\pi^{(j)}_{x_{1}^-})\times \text{supp}(\pi^{(j)}_{x_{1}^+})$ such that $x_{2}^-<x_{2}^+$. Set $x_2^+(j)\coloneq \max(X_2^{j,+})\le\min(X_2^{j,-})=:x_2^-(j)$. Then there exist vectors $$T_1^j\coloneq (x_1^+(j),x_1^{0,1}(j), \dots, x_1^{0,m_j}(j),x_1^{-}(j))$$  and $$T_2^j\coloneq \left(x_2^+(j),x_2^{0,1,-}(j),x_2^{0,1,+}(j), \dots, x_2^{0,m_j,+}(j), x_2^-(j)\right)$$ with $1\le m_j \le |\text{supp}(\mu)|$ such that the following holds: 
\begin{itemize}
\item $x_1^{\pm}(j)\in X_1^{j,\pm}$ and $x_2^{\pm}(j)\in \mathrm{supp}(\pi^{(j)}_{x_1^{\pm}(j)})$,
\item $x_1^{0,i}(j)\in X_1^{j, 0}$  for $i=1, \dots, m_j,$
\item  $x_2^{0,i,-}(j)=\min(\mathrm{supp}(\pi^{(j)}_{x_1^{0,i}(j)}))$ and $x_2^{0,i,+}(j)=\max(\mathrm{supp}(\pi^{(j)}_{x_1^{0,i}(j)}))$ for $i=1, \dots, m_j,$
\item the following holds:
\begin{align}\label{eq:oooorrdeeerr!}
x_2^{0,1,-}(j)<x_2^{+}(j)
\le x_2^{0,2,-}(j)<x_2^{0,1,+}(j)\le \cdots \le x_2^-(j)<x_2^{0,m_j,+}(j),
\end{align}
see Figure \ref{fig:2_def} for an illustration.
\end{itemize}
\end{lemma}

\begin{figure}[h!]
\begin{tikzpicture}[line cap=round,line join=round,x=0.9cm,y=0.9cm]
\draw[-,color=black] (-0.5,0) -- (12.5,0) node[right]{$\mu$}; 
\draw[-,color=black] (-0.5,3) -- (12.5,3)  node[right]{$\nu$}; 
\draw[->, line width=2pt, color=darkspringgreen] (3,0)  -- (1,3) node[above] {\footnotesize $x_2^{0,1,-}$};
\draw[->,line width=2pt, color=azure] (0,0) node[below] {\footnotesize $x_1^+$} -- (2,3)node[above] {\footnotesize $x_2^{+}$};
\draw[->,line width=2pt, color=darkspringgreen] (3,0) node[below] {\footnotesize $x_1^{0,1}$} -- (5,3) node[above] {\footnotesize $x_2^{0,1,+}$};
\draw[->,line width=2pt, color=darkspringgreen] (6,0) -- (4,3)node[above] {\footnotesize $x_2^{0,2,-}$};
\draw[->,line width=2pt, color=darkspringgreen] (6,0) node[below] {\footnotesize $x_1^{0,2}$}-- (8,3) node[above]  {\footnotesize $x_2^{0,2,+}$};
\draw[->,line width=2pt, color=darkspringgreen] (9,0) node[below] {\footnotesize $x_1^{0,3}$} -- (11,3) node[above]  {\footnotesize $x_2^{0,3,+}$};
\draw[->,line width=2pt, color=darkspringgreen] (9,0) -- (7,3) node[above]  {\footnotesize $x_2^{0,3,-}$};
\draw[->,line width=2pt, color=red] (12,0) node[below] {\footnotesize $x_1^{-}$}-- (10,3) node[above]  {\footnotesize $x_2^{-}$};
\end{tikzpicture}
\caption{$T_1^j$ and $T_2^j$ for $m_j=3$.}
\label{fig:2_def}
\end{figure}

\begin{proof}
Note that $x_2^+(j)\le x_2^-(j)$ follows from the assumption. We now prove the claim inductively. Thus we assume towards a contradiction, that there exists no $x_1^0\in X_1^{j,0}$ and $(x_2^{0,-}, x_2^{0,+})$ with $x_2^{0,-}, x_2^{0,+}\in \mathrm{supp}(\pi^{(j)}_{x_1^0})$ and $x_2^{0,-}<x_2^{+}(j)<x_2^{0,+}$. In other words, for each $x_1^0\in X_1^{j,0}$ we have
\begin{align}\label{eq 2}
\mathrm{supp}(\pi^{(j)}_{x_1^0} ) \subseteq (-\infty, x_2^+(j)] \quad  \text{or} \quad \mathrm{supp}(\pi^{(j)}_{x_1^0} ) \subseteq [x_2^+(j),\infty).
\end{align}
Furthermore, recall that $x_1^+(j)<x_2^+(j)$ and \eqref{eq 1}, which ensures $(-\infty, x_2^+(j)]\cap X_1^-=\emptyset$ and thus implies $$\int_{(-\infty, x_2^{+}(j)]}\int  (x_2-x_1)\,\pi^{(j)}(dx_1,dx_2)>0.$$
Set $g(x)\coloneq(x_2^+(j)-x)^+$. Then, as in the proof of Lemma \ref{lem:aux0},
\begin{align*}
\int g(x_1)\,\mu(dx_1)&=\int_{(-\infty, x_2^{+}(j)]}(x_2^+(j)-x_1)^+\,\mu(dx_1)\\
&=\int_{(-\infty, x_2^{+}(j)]}(x_2^+(j)-x_1)\,\mu(dx_1)\\
&> \int_{(-\infty, x_2^{+}(j)]} \int (x_2^+(j)-x_2)\,\pi^{(j)}_{x_1}(dx_2)\mu(dx_1)\\
&\stackrel{\eqref{eq 2}, \eqref{eq 1}}{=} \int_{(-\infty, x_2^{+}(j)]}(x_2^+(j)-x_2)\,\nu(dx_2)\stackrel{\eqref{eq 1}}{=}\int g(x_2)\,\nu(dx_2).
\end{align*}
The above contradicts $\mu\preceq_c \nu$ and thus shows existence of $x_1^{0,1}(j)$. Let us choose $x_2^{0,1,-}(j)\coloneq\min(\mathrm{supp}(\pi^{(j)}_{x_1^{0,1}(j)}))$ and $x_2^{0,1,+}(j)\coloneq \max(\mathrm{supp}(\pi^{(j)}_{x_1^{0,1}(j)}))$.  Now we iterate the argument until $x_2^{0,k,+}(j)>x_2^-(j)$ for some $k\in \N$ (note that
 $\text{supp}(\nu)$ is finite). This shows existence of vectors $x_1^{0,1}(j), \dots, x_1^{0,k}(j)$ such that 
\begin{align} \label{eq:property}
x_2^{0,i,-}(j)<x_2^{0,i-1,+}(j)<x_2^{0,i,+}(j).
\end{align} 
Note that if $x_2^{0,i,-}(j)<x_2^{0,i-2,+}(j)$ for some $i=2,\dots, k$ (where we set $x_2^{0,0,+}(j)\coloneq x_2^+(j)$), then $x_1^{0,i-1}(j)$ can be deleted from $(x_1^{0,1}(j), \dots, x_1^{0,k}(j))$ without changing \eqref{eq:property}. In conclusion we can assume that $x_2^{0,i,-}(j)\ge x_2^{0,i-2,+}(j)$ for all $i=2, \dots, m_j$. This shows \eqref{eq:oooorrdeeerr!} and concludes the proof.
\end{proof}

\begin{definition}\label{def:exchange}
We call the tuples $T_1^j$ and $T_2^j$ constructed in Lemma \ref{lem:aux1} \textnormal{exchange tuples for $\pi^{(j)}$}, if 
$$|T_1^j|=\min \left\{|T_1^j|\ | \ \exists~ T_2^j \mathrm{\ such\ that\ }(T_1^j,T_2^j) \mathrm{\ are\ as\ in\ Lemma\ }\ref{lem:aux1}\right\}.$$

\end{definition}
For notational convenience we make the following additional conventions for the rest of this section: we set $x_1^{0,0}(j)\coloneq x_1^+(j)$ and $x_1^{0,m_j+1}(j)\coloneq x_1^-(j)$ as well as $x_2^{0,0,+}(j)\coloneq x_2^+(j)$ and $x_2^{0,m_j+1,-}(j)\coloneq x_2^-(j)$ for all $j=1, \dots, N$. This is particularly useful in counting arguments, where we do not need to stress the special role of $x_1^+(j)$ or $x_1^-(j)$ respectively.\\


Given Lemmas \ref{lem:aux0} and \ref{lem:aux1}, we now apply the following algorithm, which should be compared to Algorithm \ref{alg:2}:
\begin{algorithm}\label{alg:1}
Set $j=0$.
\begin{enumerate}[(i)]
\item If there exists some 4-tuple $(x_1^-,x_1^+,x_2^-,x_2^+)$ with $(x_{1}^-, x_{1}^+)\in X_{1}^{j,-}\times X_{1}^{j,+}$, $x_2^-<x_2^+$ and $\left(x_{2}^-,x_{2}^+\right)\in \text{supp}(\pi^{(j)}_{x_{1}^-})\times \text{supp}(\pi^{(j)}_{x_{1}^+})$, then set $$(x_1^-(j), x_1^+(j), x_2^-(j), x_2^+(j))\coloneq \left(x_1^-,x_1^+,\min(\mathrm{supp}(\pi^{(j)}_{x_1^-})),\max(\mathrm{supp}(\pi^{(j)}_{x_1^+}))\right)$$ and carry out steps \textit{(ii)}-\textit{(iii)} of Algorithm \ref{alg:2}. Set $m_j=0$.
\item If there exist no 4-tuples $(x_1^-,x_1^+,x_2^-,x_2^+)$ with $(x_{1}^-, x_{1}^+)\in X_{1}^{j,-}\times X_{1}^{j,+}$,  $x_2^-<x_2^+$ and $\left(x_{2}^-,x_{2}^+\right)\in \text{supp}(\pi^{(j)}_{x_{1}^-})\times \text{supp}(\pi^{(j)}_{x_{1}^+})$, then use Lemma \ref{lem:aux1} to choose exchange tuples for $\pi^{(j)}$ denoted by  $T_1^j$ and $T_2^j$. Set
\begin{align*}
\lambda^{(j,+)}&\coloneq  \mu(x_1^{+}(j))\min\left\{\tilde{\lambda}>0\ \bigg| \ \int (x_2-x_1^{+}(j))\,\pi^{(j)}_{x_1^{+}(j)}(dx_2)+\tilde{\lambda}(x_2^{0,1,-}(j)-x_2^{+}(j))\le 0\right\}
\end{align*}
and
\begin{align*}
\lambda^{(j,-)}&\coloneq \mu(x_1^{-}(j))\min\left\{\tilde{\lambda}>0\ \bigg| \ \int (x_2-x_1^{-}(j))\,\pi^{(j)}_{x_1^{-}(j)}(dx_2)+\tilde{\lambda}(x_2^{0,m_j,+}(j)-x_2^{-}(j))\ge 0\right\}
\end{align*}
Furthermore for $i=1, \dots, m_j+1$ set
\begin{align*}
\lambda^{(j,i)}&\coloneq \pi^{(j)}(x_1^{0,i}(j),x_2^{0,i,-}(j))\wedge\pi^{(j)}(x_1^{0,i-1}(j),x_2^{0,i-1,+}(j)).
\end{align*}
and 
\begin{align*}
\lambda^{(j)}&=\lambda^{(j,+)} (x_2^{+}(j)-x_2^{0,1,-}(j)) \wedge \min_{i=1,\dots, m_j+1} \lambda^{(j,i)} (x_2^{0,i-1,+}(j)-x_2^{0,i,-}(j))\\
&\wedge \lambda^{(j,-)}(x_2^{0,m_j,+}(j)-x_2^{-}(j))>0.
\end{align*}
Now define $\rho^{(j)}\in \mathcal{P}(\R^2\times \R^2)$  via $$\rho^{(j)}(dx_1, dx_2, dy_1, dy_2)\coloneq (\delta_{x_1}(dy_1)\otimes  \mu(dx_1)) \otimes \rho^{(j)}_{(x_1, y_1)}(dx_2,dy_2),$$ where 
\begin{align*}
\rho^{(j)}_{(x_1,x_1)}&\coloneq(x_2, x_2)_\ast \pi^{(j)}_{x_1}\quad \text{ for all }x_1\in \mathrm{supp}(\mu)\setminus T_1^j\\
\rho^{(j)}_{(x_1^{+}(j), x_1^{+}(j))}&\coloneq (x_2, x_2)_\ast \pi^{(j)}_{x_1^{+}(j)}+\frac{\lambda^{(j)}}{\mu(x_1^{+}(j))(x_2^{+}(j)-x_2^{0,1,-}(j))}\\
&\qquad\cdot (\delta_{(x_2^{+}(j),x_2^{0,1,-}(j))}-\delta_{(x_2^{+}(j),x_2^{+}(j))}),\\
\rho^{(j)}_{(x_1^{0,i}(j),x_1^{0,i}(j))}&\coloneq (x_2, x_2)_\ast \pi^{(j)}_{x_1^{0,i}(j)}+\frac{\lambda^{(j)}}{\mu(x_1^{0,i}(j))(x_2^{0,i-1,+}(j)-x_2^{0,i,-}(j))}\\
&\qquad\cdot (\delta_{(x_2^{0,i,-}(j),x_2^{0,i-1,+}(j))}-\delta_{(x_2^{0,i,-}(j),x_2^{0,i,-}(j))})\\
&\qquad+\frac{\lambda^{(j)}}{\mu(x_1^{0,i}(j))(x_2^{0,i,+}(j)-x_2^{0,i+1,-}(j))}\\
&\qquad\cdot (\delta_{(x_2^{0,i,+}(j),x_2^{0,i+1,-}(j))}-\delta_{(x_2^{0,i,+}(j),x_2^{0,i,+}(j))}) \qquad \text{for } i=1,\dots, m_j\\
\rho^{(j)}_{(x_1^{-}(j),x_1^{-}(j))}&\coloneq (x_2, x_2)_\ast \pi^{(j)}_{x_1^{-}(j)}+\frac{\lambda^{(j)}}{\mu(x_1^{-}(j))(x_2^{0,m_j,+}(j)-x_2^{-}(j))}\\
&\qquad\cdot (\delta_{(x_2^{-}(j),x_2^{0,m_j,+}(j))}-\delta_{(x_2^{-}(j),x_2^{-}(j))}).
\end{align*}
Set $\pi^{(j+1)}(dy_1, dy_2)\coloneq \int \rho^{(j)}(dx_1, dx_2, dy_1, dy_2)$.
\end{enumerate}
Set $j=j+1$. Now iterate \textit{(i)-(ii)} and terminate if $\pi^{(j)}\in \mathcal{M}(\mu,\nu)$ for some $j\in \N$.  In that case, set $\pi_{mr}\coloneq \pi^{(j)}$.
\end{algorithm}

\begin{remark}\label{rem:explained2}
We now make the same remarks as for Algorithm \ref{alg:2} to emphasize the similarities. Indeed note that the above algorithm formalises the intuition of switching barycentre mass $\lambda^{(j)}$ between the points $(x^{+}_1(j), x_2^+(j))$  and $(x^{-}_1(j), x_2^-(j))$ through the intermediate points $$\left(x_1^{0,1}(j), x_2^{0,1,-}(j)\right), \dots, \left(x_1^{0, m_j}(j), x_2^{0, m_j,+}(j)\right).$$ In the definition of $\lambda^{(j)}$ we thus have to take care of the following constraints:
\begin{itemize}
\item After switching masses $x^{+}_1(j)\in X_1^{j+1,+}\cup X_1^{j+1,0}$ should hold, which explains the choice of $\lambda^{(j,+)}$.
\item After switching masses $x^{-}_1(j)\in X_1^{j+1,-}\cup X_1^{j+1,0}$ should hold, which explains the choice of $\lambda^{(j,-)}$.
\item For $i=1, \dots, m_j+1$, we cannot switch more probability mass than available at each of the points $(x_1^{0,i}(j), x_2^{0,i,-}(j))$ and $(x_1^{0,i-1}(j), x_2^{0,i-1,+}(j))$, which is exactly equal to $\pi^{(j)}(x_1^{0,i}(j),x_2^{0,i,-}(j))$ and $\pi^{(j)}(x_1^{0,i-1}(j),x_2^{0,i-1,+}(j))$ respectively. This explains the choice of $\lambda^{(j,i)}$.
\end{itemize}
Again, one of these constraints will be strict, i.e. after carrying out step $j$ we will either have deleted one left-most or right-most point of the support of  $\pi^{(j)}_{x_1^{0,i}(j)}$, or $x^{+}_1(j)\in X_1^{j+1,0}$ or  $x^{-}_1(j)\in X_1^{j+1,0}$ (we might have added new points to the corresponding supports, which are strictly larger or smaller than the ones we deleted).
\end{remark}

\begin{definition}
We denote the number of steps until termination of Algorithm \ref{alg:1} by $N\in \N\cup \{\infty\}$.
\end{definition}

\begin{remark}\label{rem:rho}
As in Section \ref{sec:thm_dispersion} we will mostly work with the definition of $\pi^{(j)}$ directly in order to shorten notation. Nevertheless, to make arguments in the proof of Lemma \ref{lem:aux4} and Section \ref{sec:non_compact} precise, we will sometimes use $\rho^{(j)}$ directly. In particular we will again use that for $j=0, \dots, N-1$ we have
\begin{align*}
W_1^{nd}(\pi^{(j)}, \pi^{(j+1)})\le \int \int |x_2-y_2|\,\rho_{(x_1,x_1)}^{(j)}(dx_2, dy_2)\,\mu(dx_1).
\end{align*}
\end{remark}

We are now ready to describe the evolution of the law of the canonical process $(x_1, x_2)$ throughout the steps of the algorithm. To formalise this, we write $z^j=(x_1^j, x_2^j)\in \R^2$ and define the disintegrations $\rho^{(j)}_{z^j}(dz^{j+1})$ with respect to the first two coordinates correspondingly.

\begin{definition}\label{definition:rho}
If $N<\infty$ then we define the joint law of $\rho^{(j)}$ over all steps $j=0, \dots, N-1$ as
\begin{align*}
\rho(dz^0,dz^1, dz^2,\dots,dz^N)\coloneq \rho_{z^{N-1}}^{(N-1)}(dz^N)\dots \rho^{(1)}_{z^1}(dz^2)\rho^{(0)}(dz^0,dz^1),
\end{align*}
where $z^0, \dots, z^{N}\in\R^2$.
Furthermore, for an index set $I\subseteq\{0,1, \dots, N\}$ we define $\rho^{I}$ as the push-forward measure of the projection $(z^0, \dots, z^N)\mapsto (z^i)_{i\in I}$ under $\rho$. 
\end{definition}

\begin{figure}[h!]
\begin{tikzpicture}[line cap=round,line join=round,x=0.9cm,y=0.9cm]
\draw[-,color=black] (-0.5,0) -- (12.5,0) node[right]{$\mu$}; 
\draw[-,color=black] (-0.5,3) -- (12.5,3)  node[right]{$\nu$}; 
\draw[->, line width=2pt, color=darkspringgreen] (3,0)  -- (1,3) node[above] {\footnotesize $x_2^{0,1,-}$};
\draw[->,line width=2pt, color=azure] (0,0) node[below] {\footnotesize $x_1^+$} -- (2,3)node[above] {\footnotesize $x_2^{+}$};
\draw[->,line width=2pt, color=darkspringgreen] (3,0) node[below] {\footnotesize $x_1^{0,1}$} -- (5,3) node[above] {\footnotesize $x_2^{0,1,+}$};
\draw[->,line width=2pt, color=darkspringgreen] (6,0) -- (4,3)node[above] {\footnotesize $x_2^{0,2,-}$};
\draw[->,line width=2pt, color=darkspringgreen] (6,0) node[below] {\footnotesize $x_1^{0,2}$}-- (8,3) node[above]  {\footnotesize $x_2^{0,2,+}$};
\draw[->,line width=2pt, color=darkspringgreen] (9,0) node[below] {\footnotesize $x_1^{0,3}$} -- (11,3) node[above]  {\footnotesize $x_2^{0,3,+}$};
\draw[->,line width=2pt, color=darkspringgreen] (9,0) -- (7,3) node[above]  {\footnotesize $x_2^{0,3,-}$};
\draw[->,line width=2pt, color=red] (12,0) node[below] {\footnotesize $x_1^{-}$}-- (10,3) node[above]  {\footnotesize $x_2^{-}$};
\end{tikzpicture}
\begin{tikzpicture}[line cap=round,line join=round,x=0.9cm,y=0.9cm]
\draw[-,color=black] (-0.5,0) -- (12.5,0) node[right]{$\mu$}; 
\draw[-,color=black] (-0.5,3) -- (12.5,3)  node[right]{$\nu$}; 
\draw[->,dashed, line width=2pt, color=darkspringgreen] (3,0)  -- (1,3) node[above] {\footnotesize $x_2^{0,1,-}$};
\draw[->, line width=2pt, color=darkspringgreen] (3,0)  -- (2,3);
\draw[->,dashed,line width=2pt, color=azure] (0,0) node[below] {\footnotesize $x_1^+$} -- (2,3)node[above] {\footnotesize $x_2^{+}$};
\draw[->,line width=2pt, color=azure] (0,0)  -- (1,3);
\draw[->,dashed, line width=2pt, color=darkspringgreen] (3,0) node[below] {\footnotesize $x_1^{0,1}$} -- (5,3) node[above] {\footnotesize $x_2^{0,1,+}$};
\draw[->, line width=2pt, color=darkspringgreen] (3,0)  -- (4,3);
\draw[->,dashed, line width=2pt, color=darkspringgreen] (6,0) -- (4,3)node[above] {\footnotesize $x_2^{0,2,-}$};
\draw[->, line width=2pt, color=darkspringgreen] (6,0) -- (5,3);
\draw[->,dashed, line width=2pt, color=darkspringgreen] (6,0) node[below] {\footnotesize $x_1^{0,2}$}-- (8,3) node[above]  {\footnotesize $x_2^{0,2,+}$};
\draw[->, line width=2pt, color=darkspringgreen] (6,0) -- (7,3);
\draw[->,dashed, line width=2pt, color=darkspringgreen] (9,0) node[below] {\footnotesize $x_1^{0,3}$} -- (11,3) node[above]  {\footnotesize $x_2^{0,3,+}$};
\draw[->, line width=2pt, color=darkspringgreen] (9,0) -- (10,3);
\draw[->,dashed, line width=2pt, color=darkspringgreen] (9,0) -- (7,3) node[above]  {\footnotesize $x_2^{0,3,-}$};
\draw[->, line width=2pt, color=darkspringgreen] (9,0) -- (8,3);
\draw[->,dashed,line width=2pt, color=red] (12,0) node[below] {\footnotesize $x_1^{-}$}-- (10,3) node[above]  {\footnotesize $x_2^{-}$};
\draw[->, line width=2pt, color=red] (12,0) -- (11,3);
\end{tikzpicture}
\caption{Exchange of masses at $x_2^{0,1,-}<x_2^+ \le x_2^{0,2,-} < x_2^{0,1,+} \le x_2^{0,3,-}<x_2^{0,2,+} \le x_2^- < x_2^{0,3,+}$ for the case $m_j=3$.}
\label{fig:2a}
\end{figure}

Applying Algorithm \ref{alg:1} recursively, we ``rectify" the barycentres of both $\pi^{(j)}_{x_1^+}$ and $\pi^{(j)}_{x_1^-}$, i.e. we shift both of them by the same amount to the left and right respectively. If we are in case \textit{(ii)}, the barycentres of the disintegrations at points $(x_1^{0,1}(j), \dots, x_1^{0,m_j}(j))$ still remain zero (see Figure \ref{fig:2a}). We formally prove this in the following lemmas, which list important properties of Algorithm \ref{alg:1}. 
\begin{lemma}\label{lem:aux}
The following properties hold for Algorithm \ref{alg:1}:
\begin{enumerate}[(i)]
\item If $x_1 \in X_1^{\tilde{j},0}$, then $x_1\in X_{1}^{j,0}$ for all $j \ge \tilde{j}$.
\item For any $0\le j \le N$ we have $\pi^{(j)}\in \Pi(\mu,\nu)$.
\item For any $0\le j \le N-1$ we have
\begin{align*}
\int \left|\int (x_2-x_1)\,\pi^{(j)}_{x_1}(dx_2) \right|\,\mu(dx_1)-\int \left|\int (x_2-x_1)\,\pi^{(j+1)}_{x_1}(dx_2) \right|\,\mu(dx_1)=2\lambda^{(j)}.
\end{align*}
and 
\begin{align*}
W_1^{nd}(\pi^{(j)}, \pi^{(j+1)})\le\int W_1(\pi^{(j)}_{x_1}, \pi^{(j+1)}_{x_1})\,\mu(dx_1)\le 2(m_j+1)\lambda^{(j)}.
\end{align*}
\end{enumerate}
\end{lemma}

\begin{proof}
Recalling the observations in Section \ref{sec:dispersion_finite} (especially Lemma \ref{lem:W1_1} and the proof of Prop. \ref{thm:W1} for finitely supported $\pi$ in the case $c=0$), we only have to prove the claims for steps $j$, in which case (ii) in Algorithm \ref{alg:1} is applied.\\
(i) For any element $x_1^{0,i}(j)\in T_1^j\cap X_1^{j,0}$, $i=1, \dots, m_j$ we have by Algorithm \ref{alg:1} 
\begin{align*}
\int (x_2 -x_1^{0,i}(j))\,\pi^{(j+1)}_{x_1^{0,i}}(j)(dx_2)&=\int (x_2 -x_1^{0,i}(j))\,\pi^{(j)}_{x_1^{0,i}}(j)(dx_2)\\
&+\frac{\lambda^{(j)}}{\mu(x_1^{0,i}(j))(x_2^{0,i-1,+}(j)-x_2^{0,i,-}(j))}(x_2^{0,i-1,+}(j)-x_2^{0,i,-}(j))\\
&+\frac{\lambda^{(j)}}{\mu(x_1^{0,i}(j))(x_2^{0,i,+}(j)-x_2^{0,i+1,-}(j))}(x_2^{0,i+1,-}(j)-x_2^{0,i,+}(j))\\
&=\int (x_2 -x_1^{(0,i)}(j))\,\pi^{(j)}_{x_1^{(0,i)}}(j)(dx_2)\\
&\quad +\frac{\lambda^{(j)}}{\mu(x_1^{0,i}(j))}-\frac{\lambda^{(j)}}{\mu(x_1^{0,i}(j))}=0.
\end{align*}
(ii) We note that it is sufficient to check $\pi^{(j+1)}(\R\times \{x_2\})=\pi^{(j)}(\R\times \{x_2\})$ for all $x_2\in T_2^j$. To see this let us consider $x_2^{0,i,+}(j)\in T_2^j$ and calculate for $j=1, \dots, m_j$
\begin{align*}
\pi^{(j+1)}(\R\times \{x_2^{0,i,+}(j)\})&=\pi^{(j)}(\R\times \{x_2^{0,i,+}(j)\})\\
&+\mu(x_1^{0,i+1}(j))\frac{\lambda^{(j)}}{\mu(x_1^{0,i+1}(j))(x_2^{0,i,+}(j)-x_2^{0,i+1,-})} \\
&-\mu(x_1^{0,i}(j))\frac{\lambda^{(j)}}{\mu(x_1^{0,i}(j))(x_2^{0,i,+}(j)-x_2^{0,i+1,-})} \\
&=\pi^{(j)}(\R\times \{x_2^{0,i,+}(j)\}).
\end{align*}
The cases $x_2^+, x_2^-, x_2^{0,i,-}\in T_2^j$ work analogously.\\
(iii) The first claim follows from (i) and the observation that 
\begin{align*}
&\int (x_2-x_1^+(j))\,\pi_{x_1^+(j)}^{(j)}(dx_2) -\int (x_2-x_1^+(j))\,\pi_{x_1^+(j)}^{(j+1)}(dx_2)\\
&=  \frac{\lambda^{(j)}}{\mu(x_1^{+}(j))(x_2^{+}(j)-x_2^{0,1,-}(j))}(x_2^{+}(j)-x_2^{0,1,-}(j))=\frac{\lambda^{(j)}}{\mu(x_1^+)}
\end{align*}
and 
\begin{align*}
&\int (x_2-x_1^-(j))\,\pi_{x_1^-(j)}^{(j+1)}(dx_2) -\int (x_2-x_1^-(j))\,\pi_{x_1^-(j)}^{(j)}(dx_2)\\
&=  \frac{\lambda^{(j)}}{\mu(x_1^{-}(j))(x_2^{0,m_j,+}(j)-x_2^{-}(j))}(x_2^{0,m_j,+}(j)-x_2^{-}(j))=\frac{\lambda^{(j)}}{\mu(x_1^-)}.
\end{align*}
We now show the second claim. Similarly to (i) we conclude that for all $x_1^{0,i}(j)\in T_1^j$
\begin{align*}
W_1\left(\pi^{(j)}_{x_1^{0,i}(j)}, \pi_{x_1^{0,i}(j)}^{(j+1)}\right)\le 2 \frac{\lambda^{(j)}}{\mu(x_1^{0,i}(j))}.
\end{align*}
Furthermore for $x_1^+(j)\in T_1^j$
\begin{align*}
W_1\left(\pi^{(j)}_{x_1^+(j)}, \pi_{x_1^+(j)}^{(j+1)}\right)\le  \frac{\lambda^{(j)}}{\mu(x_1^{+}(j))(x_2^{+}(j)-x_2^{0,1,-}(j))}|x_2^{0,1,-}(j)-x_2^{+}(j)|=\frac{\lambda^{(j)}}{\mu(x_1^+(j))}
\end{align*}
and similarly for $x_1^-(j)\in T_1^j$
\begin{align*}
W_1\left(\pi^{(j)}_{x_1^-(j)}, \pi_{x_1^-(j)}^{(j+1)}\right)\le \frac{\lambda^{(j)}}{\mu(x_1^{-}(j))(x_2^{0,m_j,+}(j)-x_2^{-}(j))}|x_2^{0,m_j,+}(j)-x_2^{-}(j)|=\frac{\lambda^{(j)}}{\mu(x_1^-(j))}.
\end{align*}
Writing 
\begin{align*}
W_1^{nd}(\pi^{(j)}, \pi^{(j+1)})&\le \int W_1(\pi^{(j)}_{x_1}, \pi^{(j+1)}_{x_1})\, \mu(dx_1)\\
& \le \mu(x_1^+(j))~ W_1(\pi^{(j)}_{x_1^+(j)}, \pi_{x_1^+(j)}^{(j+1)})+ \sum_{i=1}^{m_j} \mu(x_1^{0,i}(j))~W_1(\pi^{(j)}_{x_1^{0,i}(j)}, \pi^{(j+1)}_{x_1^{0,i}(j)})\\
&+ \mu(x_1^-(j))~ W_1(\pi^{(j)}_{x_1^-(j)}, \pi_{x_1^-(j)}^{(j+1)}) \le 2(m_j+1)\lambda^{(j)}
\end{align*}
concludes the proof.
\end{proof}

\begin{lemma}\label{lem:aux3}
Let us take a finitely supported measure $\pi\in \Pi(\mu,\nu)$, set $\pi^{(0)}=\pi$ and assume that  there exist no pairs $(x_{1}^-, x_{1}^+)\in X_{1}^{0,-}\times X_{1}^{0,+}$ and $\left(x_{2}^-,x_{2}^+\right)\in \text{supp}(\pi^{(0)}_{x_{1}^-})\times \text{supp}(\pi^{(0)}_{x_{1}^+})$ such that $x_{2}^-<x_{2}^+$. If we apply Algorithm \ref{alg:1} to $\pi^{(0)}=\pi$, then the following hold:
\begin{enumerate}[(i)]
\item For any steps $j\ge \tilde{j}$ with $j,\tilde{j}\in \{0, \dots, N\}$ and any $x_1\in X_1^{\tilde{j},0}$ we have $$\min(\mathrm{supp}(\pi_{x_1}^{(\tilde{j})}))\le \min(\mathrm{supp}(\pi_{x_1}^{(j)}))\le  \max(\mathrm{supp}(\pi_{x_1}^{(j)})) \le \max(\mathrm{supp}(\pi_{x_1}^{(\tilde{j})})).$$
\item For any steps $j\ge \tilde{j}$ with $j,\tilde{j}\in \{0, \dots, N\}$ and any $x_1\in X_1^{\tilde{j},+}$ we have
\begin{align*}
\max(\mathrm{supp}(\pi^{(j)}_{x_1}))\le \max(\mathrm{supp}(\pi^{(\tilde{j})}_{x_1})).
\end{align*}
Similarly for any $x_1\in X_1^{\tilde{j},-}$ we have
\begin{align*}
\min(\mathrm{supp}(\pi^{(j)}_{x_1}))\ge \min(\mathrm{supp}(\pi^{(\tilde{j})}_{x_1})).
\end{align*}
In particular we have
\begin{align}\label{eq:oha3}
  \max(X_2^{j, +})\le \max(X_2^{\tilde{j}, +})\le \min(X_2^{\tilde{j},-})\le  \min(X_2^{j, -}).
\end{align}
Thus, in every step $0\le j  \le N$ there are no pairs $(x_{1}^-, x_{1}^+)\in X_{1}^{j,-}\times X_{1}^{j,+}$ and $\left(x_{2}^-,x_{2}^+\right)\in \mathrm{supp}(\pi_{x_{1}^-}^{(j)})\times \mathrm{supp}(\pi_{x_{1}^+}^{(j)})$ such that $x_{2}^-<x_{2}^+$. 
\item For any $j > \tilde{j}$ and $x_1^0 \in X_1^{j,0}\setminus X_1^{\tilde{j},0}$ we have $$\mathrm{conv}(\mathrm{supp}(\pi_{x_1^0}^{(j)})) \cap (\max( X_2^{\tilde{j}, +}), \min(X_2^{\tilde{j}, -}))=\emptyset.$$
\item  For any $j > \tilde{j}$ and $x_1^0\in X_1^{j,0}$ the following holds: if $$\mathrm{conv}(\mathrm{supp}(\pi_{x_1^0}^{(j)})) \cap (\max( X_2^{\tilde{j}, +}), \min(X_2^{\tilde{j}, -}))\neq \emptyset$$ then $x_1^0 \in X_1^{\tilde{j},0}$.
\item If $j > \tilde{j}$ the $m_j\ge m_{\tilde{j}}$.
\end{enumerate}
\end{lemma}
\begin{proof}
(i) this follows immediately from the definition of $x_2^{0,i,-}(j)$ and $x_2^{0,i,+}(j)$ and \eqref{eq:oooorrdeeerr!} in Lemma \ref{lem:aux1} as well as the construction of $\pi^{(j+1)}$ in Algorithm \ref{alg:1}.\\
(ii) to show the first and second assertion we use the definition of $\lambda^{(j)}$ in Algorithm \ref{alg:1}, Remark \ref{rem:explained2} and Lemma \ref{lem:aux1}, which state that we always shift the right-most element of $X_2^{j,+}$ and the left-most point of $X_2^{j,-}$. Combining this with (i) in case $x_1\in X_1^{\hat{j},0}$ for some $\tilde{j}\le \hat{j}\le j$ concludes the proof. The third assertion follows directly from the first two. The last assertion is true for $j=0$ by assumption. It then follows for all $j\in \{1,\dots, N\}$ by plugging in $\tilde{j}=0$ into \eqref{eq:oha3}, which gives
\begin{align}\label{eq:oha}
  \max(X_2^{j, +})\le  \max(X_2^{0,+})\le \min(X_2^{0,-})\le  \min(X_2^{j, -}).
\end{align}
(iii)   by Lemma \ref{lem:aux} $X_1^{\tilde{j},0}\subseteq X_1^{j,0}$ holds. As $x_1^0\in X_1^{j,0}\setminus X_1^{\tilde{j},0}$ we have $x_1^0\in X_1^{\tilde{j},+}\cup X_1^{\tilde{j},-}$ and thus $\text{conv}(\mathrm{supp}(\pi_{x_1^0}^{(\tilde{j})})) \cap (\max( X_2^{\tilde{j}, +}), \min(X_2^{\tilde{j}, -}))=\emptyset$ by definition. For concreteness assume $x_1^0\in X_1^{\tilde{j},+}$, so that $\max( \mathrm{supp}(\pi_{x_1^0}^{(\tilde{j})}))\le\max( X_2^{\tilde{j},+})$. We now apply (ii) to see that $$\max( \mathrm{supp}(\pi_{x_1^0}^{(j)}))\le\max( \mathrm{supp}(\pi_{x_1^0}^{(\tilde{j})})) \le\max( X_2^{\tilde{j},+}).$$ Thus $\text{conv}(\mathrm{supp}(\pi_{x_1^0}^{(j)})) \cap (\max( X_2^{\tilde{j}, +}), \min(X_2^{\tilde{j}, -}))=\emptyset$ if $x_1^0\in X_1^{\tilde{j},+}$. The case $x_1^0\in X_1^{\tilde{j},-}$ follows from similar arguments. \\
(iv) This is the contraposition of (iii).\\
(v) We argue by contradiction: assume that $m_j < m_{\tilde{j}}$, i.e. there exist exchange tuples $T_1^j, T_2^j$ such that \eqref{eq:oooorrdeeerr!} holds. 
Set $$i_+:=\min\{i\in \{1,\dots, m_j\}: \ \max(\text{supp}(\pi_{x_1^{0,i}(j)}^{(j)}))>x_2^{+}(\tilde{j})\}$$ and 
$$ i_-:= \min\{i\in \{1,\dots, m_j\}: \ \max(\text{supp}(\pi_{x_1^{0,i}(j)}^{(j)}))>x_2^{-}(\tilde{j})\},$$ which are both well-defined by \eqref{eq:oha3}, and $\hat{m}_{\tilde{j}}:=i_--i_++1\le m_j<m_{\tilde{j}}.$ Consider $$\hat{T}_1^{\tilde{j}}:=(x_1^+(\tilde{j}), x_1^{0, i_+}(j), \dots, x_1^{0,i_-}(j), x_1^-(\tilde{j}))$$ and $$\hat{T}_2^{\tilde{j}}\coloneq \left(x_2^+(\tilde{j}),\min(\text{supp}(\pi^{(\tilde{j})}_{x_1^{0,i_+}(j)})),\max(\text{supp}(\pi^{(\tilde{j})}_{x_1^{0,i_+}(j)})), \dots, \max(\text{supp}(\pi^{(\tilde{j})}_{x_1^{0,i_-}(j)})), x_2^-(\tilde{j})\right).$$ By definition of $i_+, i_-$ and (iv) we have $x_1^{0, i_+}(j), \dots, x_1^{0,i_-}(j) \in X_1^{0,\tilde{j}}$. By (i) we conclude that $$\text{conv}(\text{supp}(\pi^{(\tilde{j})}_{x_1}))\cap \text{conv}(\text{supp}((\pi^{(\tilde{j})}_{\bar{x}_1}))\neq \emptyset$$ for neighbouring elements $x_1, \bar{x}_1$ in $\hat{T}_1^{\tilde{j}}$, so that suitable subsets of $\hat{T}_1^{\tilde{j}}$ and $\hat{T}_2^{\tilde{j}}$ satisfy \eqref{eq:oooorrdeeerr!}. In conclusion $m_{\tilde{j}}\le \hat{m}_{\tilde{j}}$ by Definition \ref{def:exchange}, a contradiction to $\hat{m}_{\tilde{j}}\le m_j < m_{\tilde{j}}$.
\end{proof}

Having established some basic properties of Algorithm \ref{alg:1} we now show that $N<\infty$. More concretely we have the following lemma, which should be compared to Lemma \ref{lem:W1_2}:
\begin{lemma}
Algorithm \ref{alg:1} terminates after at most $N \le |\mathrm{supp}(\mu)|(1+|\mathrm{supp}(\nu)|)$ steps.
\end{lemma}

\begin{proof}
The proof is analogous to the proof of Lemma \ref{lem:W1_2}.
Indeed, for $j\in \N_0$ define the set 
\begin{align*}
I^{(j)}(x_1):=\{ x_2\in \mathrm{supp}(\nu) \ | \ x_2\le \max(\mathrm{supp}(\pi^{(j)}_{x_1}) )\}
\end{align*}
for all $x_1\in X_1^{j,+}$,
\begin{align*}
I^{(j)}(x_1):=\{ x_2\in \mathrm{supp}(\nu) \ | \ \min(\mathrm{supp}(\pi^{(j)}_{x_1})\le x_2\le \max(\mathrm{supp}(\pi^{(j)}_{x_1}) )\}
\end{align*}
for all $x_1\in X_1^{j,0}$ and
\begin{align*}
I^{(j)}(x_1):=\{ x_2\in \mathrm{supp}(\nu) \ | \ x_2\ge \min(\mathrm{supp}(\pi^{(j)}_{x_1}) )\}
\end{align*}\\
 for all $x_1\in X_1^{j,-}$. Let us again remark that we have $\text{supp}(\pi_{x_1}^{(j)})\subseteq I^{(j)}(x_1)$, where the inclusion is typically strict.
 By the definition of $\lambda^{(j)}$ in Algorithm \ref{alg:1} and Remark \ref{rem:explained2} we note that in every step $j$ at least one of the following four cases occurs:
\begin{enumerate}[(i)]
\item $|X_1^{j,+}|-|X_1^{j+1,+}|=1$ or $|X_1^{j,-}|-|X_1^{j+1,-}|=1$.
\item $|I^{(j)}(x_1^+(j))|-|I^{(j+1)}(x_1^+(j))|\ge 1$.
\item $|I^{(j)}(x_1)|-|I^{(j+1)}(x_1)|\ge 1$  for some $x_1\in X_1^{j,0}$.
\item $|I^{(j)}(x_1^-(j))|-|I^{(j+1)}(x_1^-(j))|\ge 1$.
\end{enumerate}
Again, Remark \ref{rem:explained2} states that  $X_1^{j,+}\subseteq X_1^{0,+}$ and $X_1^{j,-}\subseteq X_1^{0,-}$. Combining this with Lemma \ref{lem:aux3}.(i),(ii) we conclude that $I^{(j)}(x_1)\subseteq I^{(0)}(x_1)$ for all $x_1\in \mathrm{supp}(\mu)$.
Thus the number of steps $N$ is bounded by 
\begin{align*}
&|X_1^{0,+}|+|X_1^{0,-}|+\sum_{x_1\in X_1^{0,-}} |I^{(0)}(x_1)| +\sum_{x_1\in X_1^{0,0}} |I^{(0)}(x_1)| +\sum_{x_1\in X_1^{0,+}} |I^{(0)}(x_1)|\\
&\le |\mathrm{supp}(\mu)|+|\mathrm{supp}(\mu)||\mathrm{supp}(\nu)|=|\mathrm{supp}(\mu)|(1+|\mathrm{supp}(\nu)|).
\end{align*}
This concludes the proof.
\end{proof}

We now argue that \eqref{eq:rea} holds for finitely supported measures $\pi \in \Pi(\mu,\nu)$.  By the triangle inequality
\begin{align}\label{eq:triangle}
W_1^{nd}(\pi^{(0)}, \pi^{(N)})\le \int W_1(\pi^{(0)}_{x_1}, \pi^{(N)}_{x_1})\,\mu(dx_1) \le  \sum_{j=0}^{N-1}\int W_1(\pi^{(j)}_{x_1}, \pi^{(j+1)}_{x_1})\,\mu(dx_1).
\end{align}
Thus it is sufficient to consider $\int W_1(\pi^{(j)}_{x_1}, \pi^{(j+1)}_{x_1})\,\mu(dx_1)$ individually for each $j\in \{1, \dots, N-1\}$. Furthermore, if in step $j$ of Algorithm \ref{alg:1} case (i) is carried out, then by Lemma \ref{lem:aux}.(iii) we have
\begin{align}\label{eq:nochmher}
\begin{split}
\int W_1(\pi^{(j)}_{x_1}, \pi^{(j+1)}_{x_1})\,\mu(dx_1)&\le 2\lambda^{(j)}\\
&\qquad=\int \left| \int (x_2-x_1)\,\pi^{(j)}(dx_2)\right|\,\mu(dx_1)\\
&\qquad\qquad-\int \left| \int (x_2-x_1)\,\pi^{(j+1)}(dx_2)\right|\,\mu(dx_1).
\end{split}
\end{align}
Let us now denote by $N_0\le N$ the first step, in which case (ii) in Algorithm \ref{alg:1} is carried out. In light of Lemma \ref{lem:aux3}.(ii), case (ii) in Algorithm \ref{alg:1} is then carried out for all $j\ge N_0$. As we can handle case (i) according to Section \ref{sec:thm_dispersion} and in particular \eqref{eq:nochmher} holds, we can simply start our analysis in step $j=N_0$. To avoid writing $N_0$ everywhere, we make the following standing assumption for the remainder of this section:

\begin{assumption}\label{ass:standing}
Let $\pi\in \Pi(\mu,\nu)$ be such that there exist no pairs $(x_{1}^-, x_{1}^+)\in X_{1}^-\times X_{1}^+$ and $\left(x_{2}^-,x_{2}^+\right)\in \text{supp}(\pi_{x_{1}^-})\times \text{supp}(\pi_{x_{1}^+})$ such that $x_{2}^-<x_{2}^+$.
\end{assumption}

Recalling Lemma \ref{lem:aux3}.(ii), Assumption \ref{ass:standing} is then satisfied for $\pi^{(j)}\in \Pi(\mu,\nu)$ for all $j\in \{0, \dots, N-1\}$.\\

In order get some intuition for the general result we now treat the case $N=1$.
\begin{lemma}\label{lem:aux5}
Assume that $N=1$. Then $\lambda^{(0)}$ is of order $1/m_0^2$.
\end{lemma}
\begin{proof}
Note that by \eqref{eq:oooorrdeeerr!} the intervals $\left[x_2^{0,i+1,-}(0), x_2^{0,i,+}(0)\right)$, $i=0,1,2,\dots, m_0$ are disjoint (see Figure \ref{fig:2_def}). Furthermore, by definition of $\lambda^{(0)}$
\begin{align}\label{eq:one}
\mu\left(x_1^{0,i+1}(0)\right)\ge \pi^{(0)}\left(\left(x_1^{0,i+1}(0), x_2^{0,i+1,-}(0)\right)\right)
\ge \frac{\lambda^{(0)}}{x_2^{0,i,+}(0)-x_2^{0,i+1,-}(0)}, \quad i=0,1,2,\dots, m_0
\end{align}
has to hold. Summing \eqref{eq:one} over $i=0,1,2,\dots, m_0$ this implies
\begin{align}\label{eq:bound}
1\ge  \sum_{i=0}^{m_0} \mu\left(x_1^{0,i+1}(0)\right)\ge \sum_{i=0}^{m_0} \frac{\lambda^{(0)}}{x_2^{0,i,+}(0)-x_2^{0,i+1,-}(0)}\ge \frac{\lambda^{(0)}(m_0+1)^2}{2\tilde{K}},
\end{align}
where the last inequality follows from the arithmetic-harmonic mean inequality as $$\left[x_2^{0,i+1,-}(0), x_2^{0,i,+}(0)\right),\quad i=0,1,2,\dots, m_0$$ are disjoint and $\text{supp}(\nu)\subseteq [-\tilde{K},\tilde{K}]$ for some $\tilde{K}>0$. This shows the desired growth for $\lambda^{(0)}$.
\end{proof}

The general case follows from more involved arguments. In particular we will identify points in the support of $\mu$ in each step $j$ of Algorithm \ref{alg:1}, where barycentre mass is only ever shifted in one direction (namely downwards in our case). These yield an upper bound on $\sum_{j=0}^{N-1} \lambda^{(j)}$ by a similar argument as in Lemma \ref{lem:aux5} above:
\begin{lemma}\label{lem:aux4}
Fix $\delta>0$ and assume that $\text{supp}(\nu)\subseteq [-\tilde{K},\tilde{K}]$. Let $\delta>0$. Then there exists a constant $K=K(\delta, \tilde{K})$ such that 
\begin{align}
\inf_{\tilde{\pi}\in \mathcal{M}(\mu,\nu)} W^{nd}_1(\pi,\tilde{\pi})\le K \epsilon_{\pi}+\delta.
\end{align} 
\end{lemma}
\begin{proof}
Let us fix $m\ge 5$ and define $J(m)\coloneq \{j\in \{1,\dots, N-1\} \ : \ m_j\ge m\}$ as well as $\tilde{j}\coloneq \min(J(m))$. Next we define the disjoint intervals
\begin{align}\label{eq:resc}
A_1&\coloneq \left[x_2^{0,1,-}(\tilde{j}), x_2^{0,5,-}(\tilde{j})\right),\nonumber\\
A_2&\coloneq \left[x_2^{0,5,-}(\tilde{j}), x_2^{0,9,-}(\tilde{j})\right),\nonumber\\
\vdots&\qquad \vdots\nonumber\\
A_{\tilde{m}} &\coloneq \left[x_2^{0,4(\tilde{m}-1)+1,-}(\tilde{j}), x_2^{0,4\tilde{m}+1,-}(\tilde{j})\right),
\end{align}
where $\tilde{m}\coloneq\lfloor m_{\tilde{j}}/4\rfloor$. Lemma \ref{lem:aux2} states, that for every $j=\tilde{j},\dots, N-1$ and every $k=1,\dots, \tilde{m}$ there exists $i_{j,k}\in \{1, \dots, m_j\}$ such that $$[x_2^{0,i_{j,k}, -}(j), x_2^{0,i_{j,k},+}(j))\subseteq A_{k}.$$ We make the convention that $[x_2^{0,i_{j,k}, -}(j), x_2^{0,i_{j,k},+}(j))$ is the left-most such interval, i.e. 
$$i_{j,k}\coloneq \min \left\{ i\in \{1, \dots, m_j\}\ | \ [x_2^{0,i, -}(j), x_2^{0,i,+}(j))\subseteq A_k \right\}.$$
We denote the corresponding left-neighbouring points $x_1^{0,i_{j,k}-1}(j)$ in $T_1^j$ by $\overleftarrow{x}_1^{j,k}$. Fix now $x_1\in \mathrm{supp}(\mu)$. \\
Assume for a second that $x_1=\overleftarrow{x}_1^{j,k}=\overleftarrow{x}_1^{\hat{j},\hat{k}}$ for some $ \hat{j}>j\ge \tilde{j}$ and some $1\le k, \hat{k}\le \tilde{m}$. By Lemma \ref{lem:rev} we obtain $k=\hat{k}$.
In particular the sets 
\begin{align}\label{eq:disjoint}
\left\{x_1\in \mathrm{supp}(\mu)\ : \ x_1= \overleftarrow{x}_1^{j,k}\text{ for some }j \in J(m)\right\}_{k=1, \dots, \tilde{m}}
\end{align}
are disjoint and trivially
\begin{align*}
\bigcup_{k=1}^{\tilde{m}} \left\{x_1\in \mathrm{supp}(\mu)\ : \ x_1= \overleftarrow{x}_1^{j,k} \text{ for some }j \in J(m)\right\} \subseteq \mathrm{supp}(\mu).
\end{align*}

Assume now that for $\hat{j}> j\ge\tilde{j}$ we have $x_1=\overleftarrow{x}_1^{j,k}=\overleftarrow{x}_1^{\hat{j},k}$, $x_1=x_1^{0,i}(j)=x_1^{0,\hat{i}}(\hat{j})$ and $$\left((x_1, x_2^{0,i,+}(j)),(x_1,x_2^{0,i+1,-}(j)), (x_1, x_2^{0,\hat{i},+}(\hat{j}))\right)\in \text{supp}(\rho^{\{j,j+1,\hat{j}\}}).$$ Lemma \ref{lem:monotone} states that then $
x_2^{0,i+1,-}(j)\ge x_2^{0,\hat{i},+}(\hat{j})$.\\
Next let us define $J(m,k,x_1)\coloneq \{j\in J(m)\ :\  \overleftarrow{x}_1^{j,k}=x_1\}$. The above conclusion from Lemma \ref{lem:monotone} together with the definition of $\lambda^{(j)}$ in Algorithm \ref{alg:1} and the definition of $\rho$ in Definition \ref{definition:rho} implies that
\begin{align*}
\sum_{j\in J(m,k,x_1)} \lambda^{(j)}&\stackrel{\text{Def. of }\lambda^{(j)}}{=}\sum_{j\in J(m,k,x_1)} \mu(x_1) \int_{A_k\times A_k} (x_2-y_2)\,\rho^{(j)}_{(x_1,x_1)}(dx_2,dy_2)\\
&\stackrel{\text{Def. of }\rho}{=} \mu(x_1) \int \sum_{j\in J(m,k,x_1)} (z_2^{j}-z_2^{j+1})\mathds{1}_{\{z_2^{j}, z_2^{j+1}\in A_k\}}\,  \rho_{(x_1, \dots, x_1)}(dz^0_2, \dots, dz^N_2)\\
&\le \mu(x_1) |A_k|.
\end{align*}
Note that for the equality above we have just pulled the sum inside the integral and used the definition of $\rho$.  For the inequality we used that $\rho_{(x_1, \dots, x_1})$-a.s. we have $z_2^j\ge z_2^{j+1}$ by the definition $ \overleftarrow{x}_1^{j,k}$ and $z_2^{j+1}\ge z_2^{\hat{j}}$ by Lemma \ref{lem:monotone} for all $j\le \hat{j}$, so that all terms appearing in the sum are ordered and contained in $A_k$.
Summing over $x_1\in \mathrm{supp}(\mu)$ and $k=1, \dots, \tilde{m}$ this implies
\begin{align*}
 \sum_{k=1}^{\tilde{m}} \sum_{j\in J(m)} \frac{\lambda^{(j)}}{|A_k|} &\stackrel{\eqref{eq:disjoint}}{=}  \sum_{k=1}^{\tilde{m} } \sum_{x_1\in \mathrm{supp}(\mu)}\sum_{j\in J(m,k, x_1)} \frac{\lambda^{(j)}}{|A_k|}\\
&= \sum_{x_1\in \mathrm{supp}(\mu)} \sum_{k=1}^{\tilde{m} } \sum_{j\in J(m,k, x_1)} \frac{\lambda^{(j)}}{|A_k|}\\
 & \le \sum_{x_1\in \mathrm{supp}(\mu)}\mu(x_1).
\end{align*}
In particular
\begin{align*}
1=\sum_{x_1\in \text{supp}(\mu)} \mu(x_1)&\ge\sum_{j\in J(m)} \sum_{k=1}^{\tilde{m}} \frac{\lambda^{(j)}}{|A_k|}\\
&=\sum_{j\in J(m)} \sum_{k=1}^{\tilde{m}} \frac{\lambda^{(j)}}{x_2^{0,4k+1,-}-x_2^{0,4(k-1)+1,-}}\\
&\ge \sum_{j\in J(m)}  \frac{\lambda^{(j)}\,\tilde{m}^2}{2\tilde{K}} \ge \sum_{j\in J(m)}  \frac{\lambda^{(j)}\,(m-3)^2}{32\tilde{K}}
\end{align*}
noting that $\tilde{m} \ge (m-3)/4$.
This implies  $$\sum_{j\in J(m)}\lambda^{(j)} \le \frac{32\tilde{K}}{(m-3)^2}$$ and thus
\begin{align}\label{eq:comeon}
\sum_{j\in J(m)}2(m_j+1)\,\lambda^{(j)}= 2m \sum_{j\in J(m)} \lambda^{(j)}+ 2\sum_{\hat{m}\ge m}\sum_{j\in J(\hat{m})}\lambda^{(j)}&\le \frac{64 \tilde{K}m}{(m-3)^2}+\sum_{\hat{m}\ge m}\frac{64\tilde{K}}{(\hat{m}-3)^{2}}.
\end{align}
Given $\delta>0$ there exists $m>0$ such that the sum on the right hand side of \eqref{eq:comeon} is less than $\delta$. We take the smallest such $m$ and define $K(\delta, \tilde{K})\coloneq m$.
Using Lemma \ref{lem:aux}.(iii) and the triangle inequality we conclude that 
\begin{align*}
W_1^{nd}(\pi^{(0)}, \pi^{(N)})&\le \int W_1(\pi^{(0)}_{x_1}, \pi^{(N)}_{x_1})\,\mu(dx_1)\\
& \le  \sum_{j=0}^{N-1}\int W_1(\pi^{(j)}_{x_1}, \pi^{(j+1)}_{x_1})\,\mu(dx_1)\\
&\le  2 \sum_{j=0}^{N-1} (m_j+1)\lambda^{(j)}\\
&=  2\sum_{j\in J(K(\delta,\tilde{K}))} (m_j+1) \lambda^{(j)} +2\sum_{j\notin J(K(\delta,\tilde{K}))} (m_j+1) \lambda^{(j)} \\
&\le \delta + 2K(\delta, \tilde{K})  \Bigg(\sum_{j\notin J(K(\delta,\tilde{K}))} \int \left|\int (x_2-x_1)\,\pi^{(j)}_{x_1}(dx_2) \right|\,\mu(dx_1)\\
&-\int \left|\int (x_2-x_1)\,\pi^{(j+1)}_{x_1}(dx_2) \right|\,\mu(dx_1)\Bigg)\\
&\le  \delta + 2K(\delta, \tilde{K}) \Bigg( \sum_{j=0}^{N-1} \int \left|\int (x_2-x_1)\,\pi^{(j)}_{x_1}(dx_2) \right|\,\mu(dx_1)\\
&\quad-\int \left|\int (x_2-x_1)\,\pi^{(j+1)}_{x_1}(dx_2) \right|\,\mu(dx_1)\Bigg)\\
&=\delta+2K(\delta,\tilde{K})\epsilon_{\pi}.
\end{align*}
In particular $K=2K(\delta,\tilde{K})$ depends on $\delta$ and the support of $\nu$ (via $\tilde{K}$) only.
\end{proof}

We have used the following lemmas, which we state in the notation of Lemma \ref{lem:aux4}:

\begin{lemma}\label{lem:aux2}
Let $\mu \in \Pi(\mu,\nu)$ be finitely supported, fix $m\in \N$ and recall $J(m)= \{j\in \{1,\dots, N\} \ : \ m_j\ge m\}$ and $\tilde{j}\coloneq \min(J(m))$. Then for every $j= \tilde{j},\dots,N-1$ and every $k=1, \dots, \tilde{m}=\lfloor m_{\tilde{j}}/4\rfloor$, there exists $i_{j,k}\in \{1, \dots, m_j\}$ such that $[x_2^{0,i_{j,k}, -}(j), x_2^{0,i_{j,k},+}(j))\subseteq A_{k}$.
\end{lemma} 

\begin{proof}
We will prove the lemma by contradiction. In particular we will make use of the properties of Algorithm \ref{alg:1} established in Lemmas \ref{lem:aux} and \ref{lem:aux3} above. Let us fix $j \ge \tilde{j}$. To simplify notation we consider the interval $A_1=[x_2^{0,1,-}(\tilde{j}), x_2^{0,5,-}(\tilde{j}))$; the arguments for general $k$ are the same. We note that by equation \eqref{eq:oha} we have 
\begin{align}\label{eq:oha2}
x_2^+(j)\le \max(X_2^{\tilde{j}, +}) \le  \min(X_2^{\tilde{j}, -}) \le  x_2^-(j).
\end{align}
We can assume for notational simplicity that $x_2^+(j)\ge x_2^{0,1,-}(\tilde{j})$: otherwise we consider the smallest $x_2$ contained in the set $\{x_2^+(j), x_2^{0,1,+}(j), x_2^{0,2,+}(j), \dots, x_2^{0,m_j,+}(j)\}$ satisfying $x_2\ge x_2^{0,1,-}(\tilde{j})$ and shift the argument to start from this element instead of $x_2^+(j)$ accordingly. Such a smallest $x_2$ exists by Lemma \ref{lem:aux1} and \eqref{eq:oha2}.\\ 
We now consider the two cases 
\begin{enumerate}
\item [(a)] $x_2^{0,1, -}(j) \ge x_2^{0,1,-}(\tilde{j})$ and
\item [(b)] $x_2^{0,1, -}(j)<x_2^{0,1,-}(\tilde{j})$
\end{enumerate}
separately. Cases (a) and (b) correspond to Figures \ref{fig:2} and \ref{fig:3}, where we have drawn the first elements of $T_2^j$ in blue and $T_2^{\tilde{j}}$ in green. Let us first consider case (a) and note that by Lemma \ref{lem:aux3}.(v) and the definition of $J(m)$ we have $m_{\tilde{j}} \le m_j$ and $m_{\tilde{j}}$ was chosen minimally according to Definition \ref{def:exchange}.
\begin{figure}
\begin{tikzpicture}[line cap=round,line join=round,x=1cm,y=1cm]
\draw[-,color=black] (-1.5,0) -- (9,0) node[right]{$\mu$}; 
\draw[-,color=black] (-1.5,3) -- (9,3)  node[right]{$\nu$}; 
\draw[->,line width=2pt, color=darkspringgreen] (0,0) node[below] {\footnotesize $x_1^+$} -- (1,3)node[above] {\footnotesize $x_2^{+}$};
\draw[->, line width=2pt, color=darkspringgreen] (1.5,0)  -- (0.25,3) node[above] {\footnotesize $x_2^{0,1,-}$};
\draw[->,dashed, line width=2pt, color=azure] (-1,0) node[below] {\footnotesize $x_1^+(j)$} -- (1,3);
\draw[->,dashed, line width=2pt, color=azure] (2.25,0) node[below] {\footnotesize $x_2^{0,1}(j)$} -- (0.5,3);
\draw[->,dashed, line width=2pt, color=red] (2.25,0)  -- (4,3);
\draw[->,dashed, line width=2pt, color=azure] (2.25,0)  -- (3.75,3);
\draw[->,line width=2pt, color=darkspringgreen] (1.5,0) node[below] {\footnotesize $x_1^{0,1}$} -- (2.75,3) node[above] {\footnotesize $x_2^{0,1,+}$};
\draw[->,line width=2pt, color=darkspringgreen] (3.25,0) -- (2,3)node[above] {\footnotesize $x_2^{0,2,-}$};
\draw[->,line width=2pt, color=darkspringgreen] (3.25,0) node[below] {\footnotesize $x_1^{0,2}$}-- (4.5,3) node[above]  {\footnotesize $x_2^{0,2,+}$};
\draw[->,line width=2pt, color=darkspringgreen] (5,0) node[below] {\footnotesize $x_1^{0,3}$} -- (6.25,3) node[above]  {\footnotesize $x_2^{0,3,+}$};
\draw[->,line width=2pt, color=darkspringgreen] (5,0) -- (3.75,3) node[above]  {\footnotesize $x_2^{0,3,-}$};
\draw[->,line width=2pt, color=darkspringgreen] (6.75,0) node[below] {\footnotesize $x_1^{0,4}$} -- (5.5,3) node[above]  {\footnotesize $x_2^{0,4,+}$};
\draw[->,line width=2pt, color=darkspringgreen] (6.75,0) -- (8,3) node[above]  {\footnotesize $x_2^{0,4,-}$};
\draw[->,line width=2pt, color=darkspringgreen] (8.5,0) node[below] {\footnotesize $x_1^{0,5}$}-- (7.25,3) node[above]  {\footnotesize $x_2^{0,5,-}$};
\end{tikzpicture}
\caption{Case $(a)$: The first elements of $T_2^{\tilde{j}}$ (green) and of $T_2^{j}$ (blue dotted). The red dotted arrows show values which lead to a contradiction to Definition \ref{def:exchange}. }
\label{fig:2}
\end{figure}
We conclude that $x_2^{0,1,+}(j) \le x_2^{0,3,-}(\tilde{j})$: indeed as $x_2^+(\tilde{j})\ge x_2^+(j)>x_2^{0,1,-}(j)$ by \eqref{eq:oha2}, we otherwise have $[x_2^{+}(\tilde{j}), x_2^{0,3,-}(\tilde{j})] \subset [x_2^{0,1,-}(j), x_2^{0,1,+}(j)]$. By Lemma \ref{lem:aux3}.(iv) we have $x_1^{0,1}(j) \in X_1^{0,\tilde{j}}$ and by Lemma \ref{lem:aux3}.(i) there exist $\tilde{x}_2^-, \tilde{x}_2^+ \in \text{supp}(\pi^{(\tilde{j})}_{x_1^{0,1}(j)})$ such that $\tilde{x}_2^-\le x_2^{0,1,-}(j)$ and $\tilde{x}_2^+ \ge x_2^{0,1,+}(j)$. This leads to a contradiction to minimality of $m_{\tilde{j}}$ at step $\tilde{j}$, as we could replace $x_1^{0,1}(\tilde{j})$ and $x_1^{0,2}(\tilde{j})$ by $x_1^{0,1}(j)$, thus reducing $|T_1^{\tilde{j}}|$ from $m_{\tilde{j}}+2$ to $m_{\tilde{j}}+1$. We thus conclude $[x_2^{0,1-}(j), x_2^{0,1,+}(j)) \subseteq A_1$.\\
The case $(b)$ works similarly, compare Figure \ref{fig:3}.
\begin{figure}
\begin{tikzpicture}[line cap=round,line join=round,x=1cm,y=1cm]
\draw[->,line width=2pt, color=darkspringgreen] (0,0) node[below] {\footnotesize $x_1^+$} -- (1,3)node[above] {\footnotesize $x_2^{+}$};
\draw[-,color=black] (-1.5,0) -- (9,0) node[right]{$\mu$}; 
\draw[-,color=black] (-1.5,3) -- (9,3)  node[right]{$\nu$}; 
\draw[->, line width=2pt, color=darkspringgreen] (1.5,0)  -- (0.25,3) node[above] {\footnotesize $x_2^{0,1,-}$};
\draw[->,dashed, line width=2pt, color=azure] (-1,0) node[below] {\footnotesize $x_1^+(j)$} -- (1,3);
\draw[->,dashed, line width=2pt, color=azure] (2.25,0) -- (0,3);
\draw[->,dashed, line width=2pt, color=red] (2.25,0)  -- (4,3);
\draw[->,dashed, line width=2pt, color=azure] (2.25,0) node[below] {\footnotesize $x_1^{0,1}(j)$} -- (3.75,3);
\draw[->,dashed, line width=2pt, color=red] (4.5,0)  -- (0.5,3);
\draw[->,dashed, line width=2pt, color=azure] (4.5,0)  -- (1.1,3);
\draw[->,dashed, line width=2pt, color=azure] (4.5,0) node[below] {\footnotesize $x_1^{0,2}(j)$} -- (7,3);
\draw[->,dashed, line width=2pt, color=red] (4.5,0)  -- (7.5,3);
\draw[->,line width=2pt, color=darkspringgreen] (1.5,0) node[below] {\footnotesize $x_1^{0,1}$} -- (2.75,3) node[above] {\footnotesize $x_2^{0,1,+}$};
\draw[->,line width=2pt, color=darkspringgreen] (3.25,0) -- (2,3)node[above] {\footnotesize $x_2^{0,2,-}$};
\draw[->,line width=2pt, color=darkspringgreen] (3.25,0) node[below] {\footnotesize $x_1^{0,2}$}-- (4.5,3) node[above]  {\footnotesize $x_2^{0,2,+}$};
\draw[->,line width=2pt, color=darkspringgreen] (5,0)  -- (6.25,3) node[above]  {\footnotesize $x_2^{0,3,+}$};
\draw[->,line width=2pt, color=darkspringgreen] (5,0) -- (3.75,3) node[above]  {\footnotesize $x_2^{0,3,-}$};
\draw[->,line width=2pt, color=darkspringgreen] (6.75,0) node[below] {\footnotesize $x_1^{0,4}$} -- (5.5,3) node[above]  {\footnotesize $x_2^{0,4,+}$};
\draw[->,line width=2pt, color=darkspringgreen] (6.75,0) -- (8,3) node[above]  {\footnotesize $x_2^{0,4,-}$};
\draw[->,line width=2pt, color=darkspringgreen] (8.5,0) node[below] {\footnotesize $x_1^{0,5}$}-- (7.25,3) node[above]  {\footnotesize $x_2^{0,5,-}$};
\end{tikzpicture}
\caption{Case $(b)$: The first elements of $T_2^{\tilde{j}}$ (green) and of $T_2^{j}$ (blue dotted). The red dotted arrows show values which lead to a contradiction to Definition \ref{def:exchange}. }
\label{fig:3}
\end{figure}
We conclude $x_2^{0,1,+}(j) \le x_2^{0,3,-}(\tilde{j})$ (otherwise $x_1^{0,1}(\tilde{j})$ and $x_1^{0,2}(\tilde{j})$ could be replaced by $x_1^{0,1}(j)$, which is an element of $X_1^{0, \tilde{j}}$ by Lemma \ref{lem:aux3}.(iv) as in case (a), and this again contradicts minimality) and then $x_2^{0,2,+}(j) \le x_2^{0,5,-}(\tilde{j})$ (otherwise $x_1^{0,3}(\tilde{j})$ and $x_1^{0,4}(\tilde{j})$ could be replaced by $x_1^{0,2}(j)$, which is an element of $X_1^{0, \tilde{j}}$ by Lemma \ref{lem:aux3}.(iv), and as $x_2^{0,2,-}(j) \ge x_2^{+}(j)\ge x_2^{0,1,-}(\tilde{j})$ by \eqref{eq:oooorrdeeerr!} and the convention $x_2^+(j)\ge x_2^{0,1,-}(\tilde{j})$, this concludes the proof.
\end{proof}

\begin{lemma}\label{lem:rev}
If $x_1=\overleftarrow{x}_1^{j,k}=\overleftarrow{x}_1^{\hat{j},\hat{k}}$ for some $ \hat{j}>j\ge \tilde{j}$ and some $1\le k, \hat{k}\le \tilde{m}$, then $k=\hat{k}$.
\end{lemma}
\begin{proof}
We argue by contradiction and thus assume $k\neq \hat{k}$. For concreteness we assume $\hat{k}<k$; the case $\hat{k}>k$ follows by a very similar reasoning. We claim that 
\begin{align}\label{eq:rev_claim}
A_{\hat{k}}= \left[x_2^{0,4(\hat{k}-1)+1,-}(\tilde{j}), x_2^{0,4\hat{k}+1,-}(\tilde{j})\right) \subseteq [\min(\mathrm{supp}(\pi_{x_1}^{(j)})), \max(\mathrm{supp}(\pi_{x_1}^{(j)})).
\end{align}
To show this we first note that $x_2^{0,i_{j,k}-1,+}(j)\in A_k$ by definition of $\overleftarrow{x}_1^{j,k}$. Furthermore, as $k>1$, $$\text{conv}(\text{supp}(\pi^{(j)}_{x_1}))\cap (\max(X_2^{\tilde{j},+}, \min (X_2^{\tilde{j},-}))\neq \emptyset,$$ which implies $x_1=\overleftarrow{x}_1^{j,k}\in X_1^{\tilde{j},0}$ by Lemma \ref{lem:aux3}.(iv). 
As $x_1=\overleftarrow{x}_1^{j,k}=\overleftarrow{x}_1^{\hat{j},\hat{k}}$, we have by Lemma \ref{lem:aux3}.(i) $$\min(\text{supp}(\pi^{(j)}_{x_1}))\le \min(\text{supp}(\pi^{(\hat{j})}_{x_1}))=x_2^{0,i_{\hat{j},\hat{k}}-1,-}(\hat{j})<\min (A_{\hat{k}})$$ and  $$\max(\text{supp}(\pi^{(j)}_{x_1}))= x_2^{0, i_{j,k}-1,+}(j)>\min(A_k)>\max(A_{\hat{k}}),$$ where the last inequality follows from the definition of $A_k$. This shows \eqref{eq:rev_claim}.\\
We claim that this contradicts minimality of $m_{\tilde{j}}$, by a very similar construction to the proof of Lemma \ref{lem:aux3}.(v). Indeed, define 
\begin{align*}
\hat{T}^{\tilde{j}}_1 &:=(x_1^+(\tilde{j}), x_1^{0,1}(\tilde{j})\dots,x_1^{0,4(\hat{k}-1)}(\tilde{j}), \overleftarrow{x}_1^{\hat{j},\hat{k}}, x_1^{4\hat{k}+1}(\tilde{j}), \dots, x_1^-(\tilde{j}))\\
 \hat{T}^{\tilde{j}}_2 &:=(x_2^+(\tilde{j}), x_2^{0,1,-}(\tilde{j})\dots, x_2^{0,4(\hat{k}-1),+}(\tilde{j}), \min(\text{supp}(\pi_{x_1}^{(\tilde{j})})), \max(\text{supp}(\pi_{x_1}^{(\tilde{j})})), x_2^{0,4\hat{k}+1,-}(\tilde{j}), \dots, x_2^-(\tilde{j})).
\end{align*}
By \eqref{eq:rev_claim}, suitable subtuples of $\hat{T}^{\tilde{j}}_1, \hat{T}^{\tilde{j}}_2$ satisfy \eqref{eq:oooorrdeeerr!}. In particular, by minimality, $m_{\tilde{j}}\le m_{\tilde{j}}-3$, a contradiction.
\end{proof}

\begin{lemma}\label{lem:monotone}
Assume that for $\hat{j}> j\ge\tilde{j}$ we have $x_1=\overleftarrow{x}_1^{j,k}=\overleftarrow{x}_1^{\hat{j},k}$, $x_1=x_1^{0,i}(j)=x_1^{0,\hat{i}}(\hat{j})$ and 
\begin{align}\label{eq:support}
\left((x_1, x_2^{0,i,+}(j)),(x_1,x_2^{0,i+1,-}(j)), (x_1, x_2^{0,\hat{i},+}(\hat{j}))\right)\in \mathrm{supp}(\rho^{\{j,j+1,\hat{j}\}}).
\end{align}
Then 
$x_2^{0,i+1,-}(j)\ge x_2^{0,\hat{i},+}(\hat{j})$.
\end{lemma}
\begin{proof}
If there is no $j<\bar{j}<\hat{j}$ such that $x_1\in T_1^{\bar{j}}, x_2^{0,i+1,-}(j)\in T_2^{\bar{j}}$, then no mass is shifted at $(x_1, x_2^{0,i+1,-}(j))$ in between $j$ and $\hat{j}$, so we have $x_2^{0,i+1,-}(j)= x_2^{0,\hat{i},+}(\hat{j})$ by the assumption \eqref{eq:support}. In this case the conclusion of the lemma is satisfied.\\
We thus consider the case that $x_1\in T_1^{\bar{j}},x_2^{0,i+1,-}(j)\in T_2^{\bar{j}}$ for some $j<\bar{j}<\hat{j}$. We also assume that $\bar{j}$ denotes the smallest such number $\bar{j}$. There are two possibilities: either $x_2^{0,i+1,-}(j)=x_2^{0,\bar{i},-}(\bar{j})$ or  $x_2^{0,i+1,-}(j)=x_2^{0,\bar{i},+}(\bar{j})$ for some $1\le \bar{i}\le m_{\bar{j}}$.\\
 Let us first consider the case $x_2^{0,i+1,-}(j)=x_2^{0,\bar{i},-}(\bar{j})$. As $x_2^{0,i+1,-}(j)\in A_k$, we immediately have that $x_2^{0,\bar{i},-}(\bar{j})=\min(\mathrm{supp}(\pi^{(\bar{j})}_{x_1}))\ge \min(A_k)$, which implies by Lemma \ref{lem:aux3}.(i) that $$\min(\mathrm{supp}(\pi^{(\hat{j})}_{x_1}))\ge \min(\mathrm{supp}(\pi^{(\bar{j})}_{x_1}))\ge \min(A_k).$$ This contradicts the definition of $\overleftarrow{x}_1^{\hat{j},k}$.\\
Let us now consider the case that $x_2^{0,i+1,-}(j)=x_2^{0,\bar{i},+}(\bar{j})$. By definition we have $x_2^{0,\bar{i},+}(\bar{j})=\max(\mathrm{supp}(\pi^{(\bar{j})}_{x_1}))$. Using again  Lemma \ref{lem:aux3}.(i), (ii) this implies $$x_2^{0,\hat{i},+}(\hat{j})=\max(\mathrm{supp}(\pi^{(\hat{j})}_{x_1}))\le \max(\mathrm{supp}(\pi^{(\bar{j})}_{x_1}))=x_2^{0,\bar{i},+}(\bar{j})=x_2^{0,i+1,-}(j),$$
which concludes the proof.
\end{proof}

\subsection{Proof of Theorem \ref{thm:W1} for finitely supported $\pi \in \Pi(\mu, \nu)$ with $\nu\in \mathfrak{P}$}\label{sec:non_compact}
Let  $\mathfrak{P} \subseteq \mathcal{P}_1(\R)$ be uniformly integrable. We have already proved Theorem \ref{thm:W1} for all finitely supported probability measures $\pi\in \Pi(\mu,\nu)$, where the support of $\nu$ is contained in a common compact set $[-\tilde{K}, \tilde{K}]$. We can now extend inequality \eqref{eq:rea} to all finitely supported measures $\pi\in \Pi(\mu,\nu)$, for which $\nu\in \mathfrak{P}$.\\
For this let us recall Definition \ref{definition:rho}, specifically
\begin{align*}
\rho(dz^0,dz^1, dz^2,\dots,dz^N)= \rho_{z^{N-1}}^{(N-1)}(dz^N)\dots \rho^{(1)}_{z^1}(dx^2)\rho^{(0)}(dz^0,dz^1)
\end{align*}
and
\begin{align*}
\rho^{\{0,N\}}(dx, dy) = \int \rho(dx,dz^1, \dots, ,dz^{N-1},dy).
\end{align*}

We are now ready to extend the proof of Theorem \ref{thm:W1} to measures $\pi\in \Pi(\mu,\nu)$ where $\nu\in \mathfrak{P}$.

\begin{proof}[Proof of Theorem \ref{thm:W1} for finitely supported measures with $\nu\in \mathfrak{P}$]
Applying Algorithm \ref{alg:1} as in Section \ref{sec:w1_finite} we obtain a martingale measure $\pi_m\coloneq \pi^{(N)}\in \mathcal{M}(\mu,\nu)$ and a coupling $\rho^{\{0,N\}}\in \Pi(\pi,\pi_m)$. We now show that \eqref{eq:rea} holds for $\pi_m$: indeed, as $\mathfrak{P}$ is uniformly integrable there exists $\tilde{K}=\tilde{K}(\mathfrak{P})>0$ such that $\int_{[-\tilde{K},\tilde{K}]^c} |x_2|\,\nu(dx_2)\le \delta/8$. Next we observe that by the triangle inequality
\begin{align}\label{eq:non-compact1}
|x_2-y_2|\mathds{1}_{\{x_2\notin [-\tilde{K}, \tilde{K}]\}\cup\{ y_2\notin [-\tilde{K}, \tilde{K}]\} } \le 2(|x_2|\mathds{1}_{\{x_2\notin [-\tilde{K}, \tilde{K}] \}}+|y_2|\mathds{1}_{\{y_2\notin [-\tilde{K}, \tilde{K}]\} })
\end{align}
holds for all $x_2, y_2\in \R$, so it holds in particular for all $(x_2,y_2)$ such that $(x_2,y_2)\in \text{supp}(\rho^{\{0,N\}}_{(x_1,x_1)})$ for some $x_1\in \text{supp}(\mu)$. 
We next define for all $0\le j\le N$
\begin{align*}
\tilde{m}_j:=\max \Big\{m \in \N :&\ \exists i \in \{0,\dots, m_j-1\} \\
&\text{ s.t. } [x_2^{0,i+k,-}(j), x_2^{0,i+k,+}(j)]\subseteq [-\tilde{K},\tilde{K}] \text{ for all }k=1, \dots, m  \Big\},
\end{align*}
where $x_2^{0,i+k,-}(j),x_2^{0,i+k,+}(j)$ denote elements of $T_2^j$ as defined in Lemma \ref{lem:aux1}. Let us also make the following important observation, which follow immediately from Definition \ref{def:exchange}:
\begin{enumerate}[(a)]
\item for every $0\le j\le N$ there are at most four distinct $x_1\in T_1^j$ for which simultaneously $\mathrm{supp}(\pi^{(j)}_{x_1})\cap [-\tilde{K},\tilde{K}]\neq \emptyset$ and $\mathrm{supp}(\pi^{(j)}_{x_1})\cap [-\tilde{K},\tilde{K}]^c\neq \emptyset$ (i.e. at most two on each side of the interval).
\end{enumerate}
Let us define
\begin{align*}
B_{\tilde{K}}^j \coloneq \{x_1 \in \mathrm{supp}(\mu) \ | \ \mathrm{supp}(\pi^{(j)}_{x_1})\cap [-\tilde{K},\tilde{K}]\neq \emptyset\}
\end{align*}
for all $0\le j\le N-1.$ We claim the following: if $(x_2,y_2)\in \text{supp}(\rho_{(x_1,x_1)}^{\{0,N\}})$ for some $x_1\in \text{supp}(\mu)$ and $x_2,y_2\in [-\tilde{K},\tilde{K}]$, then $x_1\in B_{\tilde{K}}^j$ for all $j\in \{0,\dots,N\}$. Indeed this can be checked using again Remark \ref{rem:explained2} and Lemma \ref{lem:aux3}.(i)-(ii).\\
We now write
\begin{align*}
W_1^{nd}(\pi, \pi_{m}&)\le \int\int |x_2-y_2|\,\rho^{\{0,N\}}_{(x_1,x_1)}(dx_2,dy_2)\,\mu(dx_1)\\
&\le \int\int_{{\{x_2\notin [-\tilde{K}, \tilde{K}]\} \cup \{y_2\notin [-\tilde{K}, \tilde{K}]\}}} |x_2-y_2|\,\rho^{\{0,N\}}_{(x_1,x_1)}(dx_2,dy_2)\,\mu(dx_1)\\
&+\int\int_{\{x_2\in [-\tilde{K}, \tilde{K}], y_2\in [-\tilde{K}, \tilde{K}]\}} |x_2-y_2|\,\rho^{\{0,N\}}_{(x_1,x_1)}(dx_2,dy_2)\,\mu(dx_1)\\
&\le 2\left(\int_{[-\tilde{K}, \tilde{K}]^c} |x_2| \, \nu(dx_2)+ \int_{[-\tilde{K}, \tilde{K}]^c} |y_2|\, \nu(dy_2)\right)\\
&+\int_{x_1\in B_{\tilde{K}}^0\cap B_{\tilde{K}}^N} \int_{[-\tilde{K},\tilde{K}]^2} |x_2-y_2|\,\rho^{\{0,N\}}_{(x_1,x_1)}(dx_2,dy_2)\,\mu(dx_1).
\end{align*}
Next, applying the triangle inequality to $|x_2-y_2|$ we have 
\begin{align*}
&\int_{x_1\in B_{\tilde{K}}^0\cap B_{\tilde{K}}^N} \int_{[-\tilde{K},\tilde{K}]^2} |x_2-y_2|\,\rho^{\{0,N\}}_{(x_1,x_1)}(dx_2,dy_2)\,\mu(dx_1)\\
&\le \sum_{j=0}^{N-1} \int_{\{x_1\in B_{\tilde{K}}^j\}}\int |x_2-y_2|\,\rho^{(j)}_{(x_1,x_1)}(dx_2,dy_2)\,\mu(dx_1)\\
&\stackrel{(a)}{\le} \sum_{j=0}^{N-1} 2(\tilde{m_j}+4)\lambda^{(j)}
\end{align*}
by the definitions of $\rho^{\{0,N\}}$ and $\tilde{m}_j $ combined  with a computation similar to the proof of (iii) of Lemma \ref{lem:aux}. Now we follow exactly the same arguments as in the proof of Lemma \ref{lem:aux4} with $\delta$ replaced by $\delta/2$ and $m_j$ replaced by $\tilde{m}_j$ to obtain
\begin{align*}
\sum_{j=0}^{N-1} 2(\tilde{m_j}+4)\lambda^{(j)}&\le \delta/2+\sum_{j=0}^{N-1} (K(\delta/2,\tilde{K})+4)\Bigg(\left|\int (x_2-x_1)\,\pi^{(j)}_{x_1}(dx_2) \right|\,\mu(dx_1) \\
&\quad -\int \left|\int (x_2-x_1)\,\pi^{(j+1)}_{x_1}(dx_2) \right|\,\mu(dx_1)\Bigg)\nonumber\\
&\le \delta/2+(K(\delta/2,\tilde{K})+4)\epsilon_{\pi}.
\end{align*}
Combining the estimates above we finally obtain
\begin{align*}
W_1^{nd}(\pi, \pi_{m}&)\le \int\int |x_2-y_2|\,\rho^{\{0,N\}}_{(x_1,x_1)}(dx_2,dy_2)\,\mu(dx_1)\\
&\le 2\left(\int_{[-\tilde{K}, \tilde{K}]^c} |x_2| \, \nu(dx_2)+ \int_{[-\tilde{K}, \tilde{K}]^c} |y_2|\, \nu(dy_2)\right)\\
&+\int_{\{x_1\in B_{\tilde{K}}^0\cap B_{\tilde{K}}^N\}} \int_{[-\tilde{K}, \tilde{K}]^2} |x_2-y_2|\,\rho^{\{0,N\}}_{(x_1,x_1)}(dx_2,dy_2)\,\mu(dx_1)\\
&\le \delta/2+\delta/2+(K(\delta/2, \tilde{K})+4)\epsilon_{\pi}\\
&\le \delta+(K(\delta/2, \tilde{K})+4)\epsilon_{\pi}.
\end{align*}
The claim follows by setting $K(\delta, \mathfrak{P}):= K(\delta/2, \tilde{K})+4$.
\end{proof}

\subsection{Proof of Theorem \ref{thm:W1} for general $\pi\in \Pi(\mu,\nu)$}
Lastly we give the proof of Theorem \ref{thm:W1} for general $\pi\in \Pi(\mu,\nu)$. This follows the arguments given in the proof of Proposition \ref{thm:dispersion} in Section \ref{sec:dispersion_general} very closely.

\begin{proof}[Proof of Theorem \ref{thm:W1} for general $\pi\in \Pi(\mu,\nu)$]
Fix $\delta>0$. By Lemma \ref{cor:approx} for $\kappa=1/n$ there exists a sequence of finitely supported measures $(\pi^n)_{n\in \N}$ such that $\pi^n\in \Pi(\mu^n,\nu^n)$, where $\mu^n\preceq_c\nu^n$ and 
\begin{align*}
W_1^{nd}\left(\pi^n,\pi\right)\le 1/n
\end{align*}
for all $n\in \N$.
In particular we have $W_1(\nu^n,\nu)\le 1/n$ for all $n\in \N$ and thus it follows from uniform integrability of $\mathfrak{P}$ and standard approximation by Lipschitz indicators that we can find $\tilde{K}>0$ such that for all $n$ large enough we have $\sup_{\nu \in \mathfrak{P}}\int_{[-\tilde{K},\tilde{K}]^c} |x_2|\,\nu^n(dx_2)\le \delta/8$. By the proof of Theorem \ref{thm:W1} for finitely supported measures we can thus find a sequence of measures $(\pi^n_m)_{n\in \N}$ such that $\pi^n_m \in \mathcal{M}(\mu^n, \nu^n)$ and
\begin{align}\label{eq:fhh2}
W_1^{nd}(\pi^n_m,\pi^n)&\le \int W_1(\pi^n_{m,x_1}, \pi^n_{x_1})\,\mu^n(dx_1)\le \delta + K(\delta, \mathfrak{P})\epsilon_{\pi^n}
\end{align}
for all $n \in \N$. We now apply the same arguments as in the proof of Proposition \ref{thm:dispersion} for general $\pi\in \Pi(\mu,\nu)$ with $\pi^n_{mr}$ replaced by $\pi^n_m$. This proof can be found in Section \ref{sec:dispersion_general}, so we only paraphrase the remaining steps here: extending the disintegrations in the same way we can still define $\bar{\pi}^n:= \mu\otimes\pi_{x_1}^n$ as well as $\bar{\pi}_{m}^n:= \mu\otimes\pi_{m,x_1}^n$ and note that $\lim_{n\to \infty} W_1^{nd}(\pi^n, \bar{\pi}^n)=0$ as well as $\lim_{n\to \infty} W_1^{nd}(\pi^n_{m}, \bar{\pi}_{m}^n)=0$. Next, 
\begin{align}\label{eq:fhh1}
\begin{split}
W_1^{nd}(\bar{\pi}_{m}^n,\bar{\pi}^n)&\le \int W_1(\bar{\pi}_{m,x_1}^n, \bar{\pi}^n_{x_1})\,\mu(dx_1)=\int W_1(\pi_{m,x_1}^n, \pi^n_{x_1})\,\mu^n(dx_1)\\
&\le \delta + K(\delta, \mathfrak{P})\epsilon_{\pi^n}
\end{split}
\end{align}
and
\begin{align}\label{eq:lowerbound2}
\int (x_2-f^{1/n}(x_1))\,\pi^n_{m,x_1}(dx_2)=0
\end{align}
hold. We now use the same precompactness results as in the proof of Proposition \ref{thm:dispersion} to obtain a disintegration $x_1\mapsto \bar{\pi}_{m,x_1}$  such that (after taking a subsequence without relabelling) the measures $$\left(\frac{1}{n} \sum_{i=1}^n \bar{\pi}_{m, x_1}^i\right)_{n\in \N}$$ converge weakly to $\bar{\pi}_{m,x_1}$ for $\mu$-a.e. $x_1\in \R$. We set $\pi_{m} \coloneq \mu \otimes \bar{\pi}_{m,x_1}$ and again conclude that $\lim_{n\to \infty}W^{nd}_p(\frac{1}{n} \sum_{i=1}^n \bar{\pi}_{m}^i,\pi_{m})=0$.\\
The proof that $\pi_{m}$ is actually a martingale measure is also analogous to the one given in the proof of Proposition \ref{thm:dispersion}: indeed we again have
\begin{align*}
\int \left(  \int (x_2-x_1)\, \left(\frac{1}{n}\sum_{i=1}^n\pi^{i}_{m}\right)_{ x_1}(dx_2) \right)^- \left( \frac{1}{n}\sum_{i=1}^n\pi^i_{m}\right)^{1}(dx_1)\le \epsilon.
\end{align*}
As $\epsilon>0$ was arbitrary, an application of Lemma \ref{lem:aap2} then shows $\pi_{mr} \in \mathcal{M}(\mu,\nu)$.\\
The arguments to show that
\begin{align*}
W_1^{nd}\left(\pi_{mr},\pi\right)&\le K(\delta,\mathfrak{P})\epsilon_\pi+\delta.
\end{align*}
remain unchanged. This concludes the proof.
\end{proof}

\section{Proofs of remaining results in Section \ref{sec:main}}\label{sec:proofs}
\begin{proof}[Proof of Lemma \ref{lem:upperbound}]
We observe that for an arbitrary $\tilde{\pi}\in \mathcal{M}(\mu,\nu)$
\begin{align*}
W_1^{nd}(\pi,\tilde{\pi})&=\inf_{\gamma^1\in \Pi(\mu,\mu)} \Bigg(\int|x_1-y_1|\,\gamma^1(dx_1,y_1) \\\
&\qquad\qquad+\int \inf_{\gamma^2\in \Pi(\pi_{x_1},\tilde{\pi}_{y_1})}\int |x_2-y_2|\,\gamma^2(dx_2, dy_2)\, \gamma^1(dx_1,dy_1)\Bigg) \nonumber\\
&\ge  \inf_{\gamma^1\in \Pi(\mu,\mu)}\Bigg(\int|x_1-y_1|\, \gamma^1(dx_1,dy_1)\\
&\qquad\qquad+ \int \inf_{\gamma^2\in \Pi(\pi_{x_1},\tilde{\pi}_{y_1})}\left| \int (x_2-y_2)\,\gamma^2(dx_2, dy_2)\right| \gamma^1(dx_1,dy_1)\Bigg) \nonumber\\
&=\inf_{\gamma^1\in \Pi(\mu,\mu)}\Bigg(\int|x_1-y_1|\, \gamma^1(dx_1,dy_1)\\
&\qquad\qquad+ \int \inf_{\gamma^2\in \Pi(\pi_{x_1},\tilde{\pi}_{y_1})}\left| \int (x_2-y_1)\,\gamma^2(dx_2, dy_2)\right| \gamma^1(dx_1,dy_1)\Bigg) \nonumber\\
&=\inf_{\gamma^1\in \Pi(\mu,\mu)}\Bigg(\int|x_1-y_1|\, \gamma^1(dx_1,dy_1)+ \int \inf_{\gamma^2\in \Pi(\pi_{x_1},\tilde{\pi}_{y_1})}\Bigg| \int x_2\,\gamma^2(dx_2, dy_2)\\
&\qquad\qquad -x_1+x_1-y_1\Bigg| \gamma^1(dx_1,dy_1)\Bigg) \nonumber\\
&\ge \int \left| \int (x_2-x_1)\,\pi_{x_1}(dx_2)\right|\mu(dx_1)=\epsilon_{\pi}\nonumber
\end{align*}
holds by an application of Jensen's inequality and reverse triangle inequality. This shows the claim.
\end{proof}

\begin{proof}[Proof of Lemma \ref{lem:fh}]
Assume towards a contradiction that there exists $x_0\in \R$ such that $$ \int_{\{x_1 \ge x_0\}} (x_2-x_1) \, \pi_{HF}(dx_1, dx_2)<0.$$ Noting that $$x \mapsto \int_{\{x_1 \ge x\}} (x_2-x_1)\, \pi_{HF}(dx_1, dx_2)$$ is left-continuous and non-increasing on $\{x\in \R \ : \ \inf(\text{supp}(\pi_{HF,x})) \ge x\}$ we conclude that $$\mu\left(\left\{x \in \R \ : \  \int_{\{x_1 \ge x\}} (x_2-x_1) \, \pi_{HF}(dx_1, dx_2)<0,  \inf(\text{supp}(\pi_{HF,x}))<x\right\}\right)>0.$$ Consequently we can choose
$$ x^\ast \in \left\{x \in \Gamma^1 \ : \  \int_{\{x_1 \ge x\}} (x_2-x_1) \, \pi_{HF}(dx_1, dx_2)<0,  \inf(\text{supp}(\pi_{HF,x}))<x, \pi_{HF,x}(\Gamma_{x})=1\right\} \neq \emptyset,$$
where $\Gamma^1$ is the projection of $\Gamma$ onto the first coordinate and $\Gamma_{x_1}$ is the $x_1$-section of $\Gamma$.
We set $\hat{x}=\inf(\text{supp}(\pi_{HF,x^\ast}))$. For the convex function $x \mapsto (x-\hat{x})^+$ we then have 
\begin{align*}
\int (x_1-\hat{x})^+ \, \mu(dx_1)&\ge \int_{\{x_1\ge x^\ast\}} (x_1-\hat{x})  \, \mu(dx_1) > \int_{\{x_1\ge x^\ast\}} (x_2-\hat{x})\, \pi_{HF}(dx_1,dx_2)\\
&=\int_{\{x_1\ge x^\ast\}\cap \Gamma} (x_2-\hat{x})\, \pi_{HF}(dx_1,dx_2)= \int_{\{x_2\ge \hat{x}\}\cap \Gamma} (x_2-\hat{x})\, \pi_{HF}(dx_1,dx_2)\\
&=\int_{\{x_2\ge \hat{x}\}} (x_2-\hat{x})\, \nu(dx_2)=\int (x_2-\hat{x})^+ \, \nu(dx_2),
\end{align*}
where we used the definition of $x^\ast$ for the first and second inequality and \eqref{eq:fh_montonone} for the equality in the second line. This is a contradiction to $\mu \preceq_c \nu$ and thus proves the claim.
\end{proof}

\begin{proof}[Proof of Corollary \ref{cor:special}]
As every $\pi \in \mathcal{M}(\mu, \nu)$ satisfies $\int (x_2-x_1)\, \pi_{x_1}(dx_2)=0$ $\mu$-a.s., it clearly fulfils the barycentre dispersion assumption \ref{def:dispersion} and we have
\begin{align*}
&\inf_{\pi\in \Pi(\mu,\nu),\ \pi\ \mathrm{satisfies\ Ass.\  \ref{def:dispersion}}} \left( \int c(x_1,x_2)\,\pi(dx_1,dx_2) +
L\int \left|\int (x_2-x_1)\,\pi_{x_1}(dx_2)\right|\,\mu(dx_1) \right) \\
&\le C(\mu, \nu).
\end{align*}
Now take any $\pi\in \Pi(\mu, \nu)$ satisfying the barycentre dispersion assumption \ref{def:dispersion} and any $\tilde{\pi} \in \mathcal{M}(\mu, \nu)$. Then 
\begin{align}\label{eq:proof}
\int c(x_1, x_2)\, \tilde{\pi}(dx_1, dx_2) &\le \int c(x_1, x_2)\,\pi(dx_1, dx_2)+ \Big( \int c(x_1, x_2)\, \tilde{\pi}(dx_1, dx_2)\\
&- \int c(x_1, x_2)\, \pi(dx_1, dx_2)\Big)\nonumber \\
&\le \int c(x_1, x_2)\,\pi(dx_1, dx_2)+ LW_1^{nd}(\pi, \tilde{\pi})
\end{align}
as $c$ is $L$-Lipschitz-continuous. Taking the infimum over $\tilde{\pi} \in \mathcal{M}(\mu, \nu)$ in \eqref{eq:proof} and using Proposition \ref{thm:dispersion} we conclude that 
\begin{align*}
C(\mu, \nu) &\le \int c(x_1, x_2)\, \pi(dx_1, dx_2)+ L\inf_{\tilde{\pi} \in \mathcal{M}(\mu, \nu)} W^{nd}_1(\pi, \tilde{\pi}) \\
&= \int c(x_1, x_2)\, \pi(dx_1, dx_2)+ L\int \left|\int (x_2-x_1)\,\pi_{x_1}(dx_2)\right|\,\mu(dx_1) .
\end{align*}
Taking the infimum over $\pi \in \Pi(\mu, \nu)$ satisfying Assumption \ref{def:dispersion} concludes the proof.
\end{proof}

\begin{proof}[Proof of Remark \ref{rem:antitone}]
Let us first assume that $\pi_{AT}\in \Pi(\mu,\nu)$ is finitely supported. We note that the barycentre dispersion assumption is not satisfied in general for $\pi_{AT}$, but we may apply Algorithm \ref{alg:1}. Particular care has to be taken if there exists $x_1^\ast \in \text{supp}(\mu)$ such that $x_1^\ast \in \text{conv}(\text{supp}(\pi_{x_1^\ast}))$. By \eqref{eq:antitone}, there exists at most one such $x_1^\ast$. We remark that this in particular includes the case that $X_1^{0,0}=\{x_1^\ast\}$. In the case that $x_1^\ast\notin X_1^{0,0}$ we now first apply Algorithm \ref{alg:2} to $x_1^\ast$ until $x_1^\ast \in X_1^{j,0}$ for some $j\in \N$, in such a way that $(\min(X_2^{j,0}),\max(X_2^{j,0}))\cap \text{supp}(\nu)=\bigcup_{x_1\in X_1^{j,0}}\text{supp}(\pi^{(j)}_{x_1})$. This can always be achieved by exchanging mass in the direct (left or right) neighbourhood of $\pi_{x_1^\ast}$, as \eqref{eq:antitone} holds for $\pi$. For the rest of the iterations we now leave $\{\pi^{(j)}_{x_1} \ : \ x_1\in X_1^{j,0}\}$ unchanged. Formally this can be achieved by following \cite[proof of Lemma 2.8]{Beiglbock:2016kt}: let us define the sub-probability measure $\pi^\ast$ via $\pi^{\ast}(A):=\pi^{(j)}(A)-\pi^{(j)}(A\cap (\{X_1^{j,0}\}\times \R))$ for all Borel sets $A$. We call its marginals $\mu^\ast$ and $\nu^\ast$. It remains to check that these are still in convex order: take any convex function $\varphi:\R\to \R$. Then $\varphi$ is dominated by a convex function $\psi$ which is linear on $(\min(X_2^{j,0}),\max(X_2^{j,0}))$ and agrees with $\varphi$ on $\R\setminus (\min(X_2^{j,0}),\max(X_2^{j,0}))$. Then $\int_{X_1^{j,0}} \int \psi(x_2) \pi^{(j)}_{x_1}(dx_2)\,\mu(dx_1)= \int_{X_1^{j,0}} \psi(x_1)\, \mu(dx_1)$ and
\begin{align*}
\int \varphi(x_1)\,\mu^\ast(dx_1)& \le\int \psi(x_1)\,\mu^\ast(dx_1)\\
& = \int \psi(x_1) \,\mu(dx_1) -\int_{X_1^{j,0}} \psi(x_1)\,\mu(dx_1) \\
&\le  \int \psi(x_2) \,\nu(dx_2)-\int_{X_1^{j,0}} \psi(x_1)\,\mu(dx_1) \\
&=  \int \psi(x_2)\,\nu(dx_2) -\int_{X_1^{j,0}} \int \psi(x_2)\,\pi^{(j)}_{x_1}(dx_2)\,\mu(dx_1) \\
&= \int \psi(x_2)\,\nu^\ast(dx_2)=\int \varphi(x_2)\,\nu^\ast(dx_2),
\end{align*}
thus $\mu^\ast \preceq_c \nu^\ast$ follows.
We now apply Algorithm \ref{alg:1} to $\pi^{\ast}$ to obtain a sub-probability measure $\pi^{(\ast, N)}$ with marginals $\mu^\ast$ and $\nu^\ast$ satisfying the martingale condition $\int (x_2-x_1)\,\pi^{(\ast, N)}_{x_1}=0$ for all $x_1\in \text{supp}(\mu)\setminus X_1^{j,0}$. We  denote the (bicausal) coupling between $\pi$ and $\pi^{(N)}$ defined via $\pi^{(N)}(A):=\pi^{(\ast,N)}(A)+\pi^{(j)}(A\cap (\{X_1^{j,0}\}\times \R))$ for all Borel sets $A$ by $\rho^{\{0,N\}}$. By construction of $\pi^{(j)}$ and the properties of Algorithm \ref{alg:1}, in particular item (i) of Lemma \ref{lem:aux3}, we note that for $(x, y) \in \text{supp}(\rho^{\{0,N\}})$ we have
\begin{align*}
&x_2 \ge y_2 \text{ if } x_1\in X_1^{0,+},\\
&x_2 \le y_2 \text{ if } x_1\in X_1^{0,-},\\
&x_2 = y_2 \text{ if } x_1\in X_1^{0,0}.
\end{align*}
Thus in particular
\begin{align*}
\begin{split}
W_1^{nd}(\pi^{(0)}, \pi^{(N)})&\le \int \int  |x_2-y_2|\, \rho^{\{0,N\}}_{(x_1, x_1)}(dx_2, dy_2) \,\mu(dx_1) \\
&=\int_{x_1\in X_1^{0,+}} \int (x_2-y_2)\, \rho^{\{0,N\}}_{(x_1, x_1)}(dx_2, dy_2) \,\mu(dx_1) \\
&\qquad+\int_{x_1\in X_1^{0,-}} \int (y_2-x_2)\, \rho^{\{0,N\}}_{(x_1, x_1)}(dx_2, dy_2) \,\mu(dx_1) \\
&= \int \left| \int (x_2-x_1) \,\pi_{x_1}(dx_2)\right|\,\mu(dx_1).
\end{split}
\end{align*}
 This finishes the proof for the finitely supported case.
\end{proof}

\begin{proof}[Proof of Corollary \ref{cor:jourdain}]
Fix $\delta>0$ and $\nu\in \mathcal{P}_1(\R)$. As $\mathfrak{P}=\{\nu\}$ is uniformly integrable, we can apply Theorem \ref{thm:W1} to obtain a constant $K(\delta,\nu)$ such that we have
\begin{align}\label{eq:jou}
\inf_{\tilde{\pi}\in \mathcal{M}(\mu,\nu)} W_1^{nd}(\pi,\tilde{\pi})\le K(\delta,\nu)\epsilon_{\pi}+\delta
\end{align}
for all $\pi\in \Pi(\mu,\nu)$ where $\mu \in \mathcal{P}(\R)$ with $\mu\preceq_c\nu$. As $\epsilon_{\pi}=0$ for all $\pi \in \mathcal{M}(\mu,\nu)$ the first inequality in Corollary \ref{cor:jourdain} is trivial. Now take any $\pi\in \Pi(\mu, \nu)$ and any $\tilde{\pi} \in \mathcal{M}(\mu, \nu)$. Then as in the proof of Corollary \ref{cor:special}
\begin{align}\label{eq:proof2}
\int c(x_1, x_2)\, \tilde{\pi}(dx_1, dx_2) &\le \int c(x_1, x_2)\,\pi(dx_1, dx_2)+ \Big( \int c(x_1, x_2)\, \tilde{\pi}(dx_1, dx_2)\\
&- \int c(x_1, x_2)\, \pi(dx_1, dx_2)\Big)\nonumber \\
&\le \int c(x_1, x_2)\,\pi(dx_1, dx_2)+ LW_1^{nd}(\pi, \tilde{\pi})
\end{align}
as $c$ is $L$-Lipschitz-continuous. Taking the infimum over $\tilde{\pi} \in \mathcal{M}(\mu, \nu)$ in \eqref{eq:proof2} and using \eqref{eq:jou} we conclude that 
\begin{align*}
C(\mu, \nu) &\le \int c(x_1, x_2)\, \pi(dx_1, dx_2)+ L\inf_{\tilde{\pi} \in \mathcal{M}(\mu, \nu)} W^{nd}_1(\pi, \tilde{\pi}) \\
&= \int c(x_1, x_2)\, \pi(dx_1, dx_2)+ LK(\delta,\nu)\epsilon_{\pi}+L\delta.
\end{align*}
Taking the infimum over $\pi \in \Pi(\mu, \nu)$ concludes the proof.
\end{proof}

\begin{proof}[Proof of Theorem \ref{thm:approx_simple}]
For all $n\in \N$ we take $\pi^n \in \mathcal{M}(\mu^n,\nu^n)$ such that 
\begin{align*}
\inf_{\pi \in \mathcal{M}(\mu^n,\nu^n)} &\left( \int c(x_1,x_2) \, \pi(dx)\right)\ge  \int  c(x_1,x_2) \, \pi^n(dx)-1/n
\end{align*}
and note that (possibly after taking a subsequence) there exists $\tilde{\pi} \in \mathcal{M}(\mu,\nu)$ such that $\lim_{n\to \infty}W_p(\pi^n,\tilde{\pi})=0$. Then
\begin{align*}
\liminf_{n\to \infty}\inf_{\pi \in \mathcal{M}(\mu^n,\nu^n)} \int c(x_1,x_2)\,\pi(dx_1,dx_2) &\ge\liminf_{n\to \infty}\left( \int c(x_1,x_2)\,\pi^n(dx_1,dx_2) -1/n\right) \\
&=\int c(x_1,x_2)\,\tilde{\pi}(dx_1,dx_2)\\
&\ge  \inf_{\pi \in \mathcal{M}(\mu,\nu)}\int c(x_1,x_2)\,\pi(dx_1,dx_2).
\end{align*}
For the converse inequality we note that for all $n\in \N$ there exists $\pi^n \in \mathcal{M}(\mu,\nu)$ such that
\begin{align*}
&\inf_{\pi\in \mathcal{M}(\mu,\nu)} \int  c(x_1,x_2) \, \pi(dx_1,dx_2)\ge \int  c(x_1,x_2) \,\pi^n(dx_1,dx_2) -1/n.
\end{align*}
We now apply Lemma \ref{lemma:approx2a} to conclude that for every $n\in \N$ there exists a coupling $\tilde{\pi}^n\in \Pi(\mu^n,\nu^n)$ such that $W_p(\pi^n, \tilde{\pi}^n)\le W_p(\mu, \mu^n)+W_p(\nu, \nu^n)$ and
\begin{align*}
\int \left|\int (x_2-x_1)\, \tilde{\pi}^n_{x_1}(dx_2) \right|\,\mu^n(dx_1)\le W_p(\mu, \mu^n)+W_p(\nu, \nu^n).
\end{align*}
Note that as $\lim_{n\to \infty} W_p(\nu_n,\nu)=0$ the sequence $\{\nu_n\}_{n\in \N}$ is in particular uniformly integrable. By Theorem \ref{thm:W1} there exists a sequence $\pi^n_{m}\in \mathcal{M}(\mu^n,\nu^n)$ such that $\lim_{n\to \infty}W^{nd}_1(\tilde{\pi}^n, \pi^n_{m})=0$. Recall that the function $c:\R\times \R\to \R$ satisfies $|c(x_1,x_2)|\le C(1+|x_1|^p)+|x_2|^p)$. Then using Lemma \ref{lemma:uniform_integrability}
\begin{align*}
\inf_{\pi\in \mathcal{M}(\mu,\nu)}  \int  c(x_1,x_2) \, \pi(dx_1,dx_2)&\ge \limsup_{n \to \infty} \int c(x_1,x_2) \, \pi^n(dx_1,dx_2) \\
&\ge \limsup_{n \to \infty}  \int  c(x_1,x_2) \, \tilde{\pi}^n(dx_1,dx_2)\\
&\ge \limsup_{n \to \infty}  \ \int c(x_1,x_2) \,\pi_{m}^n(dx_1,dx_2)\\
&\ge \limsup_{n \to \infty} \inf_{\pi\in \mathcal{M}(\mu^n,\nu^n)}\int c(x_1,x_2)\, \pi(dx_1,dx_2).
\end{align*}
This concludes the proof.
\end{proof}

\begin{proof}[Proof of Theorem \ref{thm:monotone}]
We only show sufficiency here. For a proof of necessity see e.g. \cite{Beiglbock:2019ufa}.
Let us assume that $\pi\in \mathcal{M}(\mu,\nu)$ is not an optimiser of \eqref{eq:mot}. We denote 
\begin{align*}
\delta&\coloneq
 \int c(x_1,x_2)\, \pi(dx_1,dx_2)-\inf_{\hat{\pi}\in \mathcal{M}(\mu,\nu)} \int c(x_1,x_2) \, \hat{\pi}(dx_1,dx_2)>0.
\end{align*}
Let $\Gamma\subseteq \R^2$ be a Borel set such that $\pi(\Gamma)=1$. By Lemma \ref{lemma:approx1} there exists a sequence of measures $(\pi^n)_{n\in \N}$, such that for each $n\in \N$ $\pi^n$ is finitely supported on $\Gamma$, $\pi^n \in \mathcal{M}(\mu^n,\nu^n)$ for some sequences of measures $(\mu^n)_{n\in \N}$ and $(\nu^n)_{n\in \N}$ and $\lim_{n\to \infty} W^{nd}_p(\pi^n,\pi)=0$. Clearly  $\mu^n \preceq_c \nu^n$ and $$\lim_{n\to \infty} W_p(\mu^n,\mu)=\lim_{n\to \infty} W_p(\nu^n,\nu)=0.$$ By Theorem \ref{thm:approx_simple} we also have
\begin{align*}
\lim_{n\to \infty} \inf_{\pi \in \mathcal{M}(\mu^n,\nu^n)} \int c(x,y)\,\pi(dx,dy)=\inf_{\pi \in \mathcal{M}(\mu,\nu)} \int c(x,y)\,\pi(dx,dy),
\end{align*}
in particular there exists $n_0\in \N$ such that for all $n\ge n_0$
\begin{align*}
 \inf_{\pi \in \mathcal{M}(\mu^n,\nu^n)} \int c(x_1,x_2)\,\pi(dx_1,dx_2)\le \int c(x_1,x_2) \,\pi^n(dx_1,dx_2)-2\delta/3.
\end{align*}
There exists a measure $\pi'\in \mathcal{M}(\mu^n,\nu^n)$ such that 
\begin{align*}
\int c(x_1,x_2)\,\pi'(dx_1,dx_2)- \inf_{\pi \in \mathcal{M}(\mu^n,\nu^n)} \int c(x_1,x_2)\,\pi(dx_1,dx_2)\le \delta/3.
\end{align*}
In particular $\pi'$ is a competitor of $\pi^n$ and
\begin{align*}
\int c(x_1,x_2)\, \pi'(dx_1,dx_2) \le \int c(x_1,x_2)\, \pi^n(dx_1,dx_2)-\delta/3,
\end{align*} 
showing that $\Gamma$ is not finitely optimal.
\end{proof}

\begin{appendix}
\section{Proofs of approximation results}

Let us first recall the following result:

\begin{lemma}[General Tchakaloff's theorem, cf. {\cite[Corollary 2]{bayer2006proof}}]\label{lem:tchakaloff}
Let $\mu\in \mathcal{P}(\R)$, $m\in \N$ and $f:\R\to \R^m$ such that $\int |f(x)|\,\mu(dx)<\infty$. Then there exists a probability measure $\tilde{\mu}\in \mathcal{P}(\R)$ with finite support such that $\mathrm{supp}(\tilde{\mu})\subseteq \mathrm{supp}(\mu)$ and $\int f(x)\,\tilde{\mu}(dx)=\int f(x)\,\mu(dx)$.
\end{lemma}

\begin{proof}[Proof of Lemma \ref{lemma:approx1}]
Throughout the proof we make the convention that $x/0:=0$ for any $x\in \R$.
We prove (i) via two discretisations: first we approximate the  marginal $\mu$ and consecutively we approximate the disintegration $(\pi_{x_1})_{x_1\in \R}$.\\\\\
For the first approximation, note that by \cite[Theorem 1.1.8]{strvar}  the property $\pi(\Gamma)=1$ implies $$\mu\left(\left\{x_1\in \R \ : \ \pi_{x_1}(\Gamma_{x_1}\}=1\right\}\right)=1,$$ where $\Gamma_{x_1}$ denotes the $x_1$-section of $\Gamma$. Without loss of generality we thus assume that $\pi_{x_1}(\Gamma_{x_1})=1$ for all $x_1\in \Gamma^{1}$, where we recall that $\Gamma^{1}$ denotes the projection of $\Gamma$ to the first coordinate. We can further assume without loss of generality that $\Gamma^1$ is not finite and fix some $0\neq a_0\in \Gamma^1$ satisfying $ \int |x_2|^p \,\pi_{a_0}(dx_2)<\infty$ (see \cite[Cor. 1.1.7]{strvar}). Choose $c_p\ge 1$ such that $(x+y)^p\le c_p(x^p+y^p)$ for all $x,y \in \R$.
As $\int |x_1|^p\,\mu(dx_1)<\infty$, $\int |x_2|^p\,\nu(dx_2)<\infty$ and $x_1\mapsto \pi_{x_1}$ is Borel, an application of Lusin's theorem (see \cite[Theorem 7.1.12]{Bogachev:2007gk}) to the measure $\zeta$ defined via $$\zeta(A)\coloneq\frac{1}{3}\left(\frac{\int_{A} |x_1|^p\,\mu(dx_1)}{\int |x_1|^p\,\mu(dx_1)}+\frac{\int_A\int |x_2|^p \, \pi_{x_1}(dx_2)\,\mu(dx_1)}{\int |x_2|^p\,\nu(dx_2)}+\mu(A)\right)$$ for every Borel set $A\subseteq \R$, there exists a compact set $K_1\subseteq \Gamma^1$ such that $a_0\in K_1$,
\begin{align*}
&\int_{K_1^c} |x_1|^p  \mu(dx_1)\le \kappa^p/(6c_p),\quad
\int_{K_1^c}\int |x_2|^p \, \pi_{x_1}(dx_2)\mu(dx_1)\le \kappa^p/(6c_p),\\
&\mu(K_1^c)\le \frac{ \kappa^p/(6c_p)}{|a_0|^p\vee \int |x_2|^p \, \pi_{a_0}(dx_2)}
\end{align*}  and $x_1\mapsto \pi_{x_1}$ is continuous in $W_p$ on $K_1$. As $K_1$ is compact, $x_1\mapsto \pi_{x_1}$ is uniformly continuous on $K_1$. Thus there exists $\delta >0$ such that $W_p(\pi_{x_1}, \pi_{y_1})\le \kappa/6$ for all $x_1, y_1\in K_1$ with $|x_1-y_1|\le \delta$ and a finite set $K_{1,\kappa}=\{a_0,a_1, \dots, a_N\}$ of $K_1$ such that $a_1< \dots< a_N$ and $$\min_{a\in K_{1,\kappa},\ a\le y_1}(y_1-a)\le \kappa/6\wedge \delta$$ for all $y_1\in K_1$,  where $a_1$ and $a_N$ are the left and right end-points of $K_1$. Let $(P_i)_{i=0}^{N}$ be a disjoint partition of $K_1$ given by
\begin{align*}
P_i=\left\{ y_1\in K_1:\ a_i= \mbox{argmin}_{a\in K_{1,\kappa},\ a\le y_1} (y_1-a) \right\}, \qquad i=0, \dots, N.
\end{align*}
We now define 
\begin{align}\label{eq:fkappa}
f^{\kappa}(x_1)=a_0\mathds{1}_{K_1^c}(x_1)+\sum_{i=0}^{N} a_i\mathds{1}_{P_i}(x_1)
\end{align}
and set $\tilde{\pi}=(f^{\kappa}_\ast \pi^1) \otimes \pi_{x_1}$, where we recall that $(f_\ast \pi^1)$ denotes the push-forward of $\pi^1$ by the function $f^\kappa$. To estimate $W_1^{nd}(\pi, \tilde{\pi})$ we set $\gamma^1=(x, f^\kappa(x))_\ast \pi^1$ and note that $\gamma^1\in \Pi(\pi^1, \tilde{\pi}^1)$. We define  $\gamma^2\in \Pi(\pi_{x_1}, \tilde{\pi}_{y_1})$ as a coupling which attains $W_p(\pi_{x_1},\tilde{\pi}_{y_1})$. Then by the triangle inequality and noting that $\pi^1(\Gamma^1)=1$
\begin{align}\label{eq:long1}
W_p^{nd}(\pi,\tilde{\pi})&\le \left(\int |x_1-y_1|^p + \int |x_2-y_2|^p\,\gamma^2(dx_2,dy_2)\gamma^1(dx_1,dy_1) \right)^{1/p}\nonumber\\
&\le \bigg(c_p \int_{K_1^c} \left(|x_1|^p+|f^\kappa(x_1)|^p +\int |x_2|^p +|y_2|^p\,\gamma^2(dx_2,dy_2)\right)\,\pi^1(dx_1)\nonumber \\
&\quad+ \int_{K_1} \left(|x_1-f^\kappa(x_1)|^p +\int |x_2-y_2|^p\,\gamma^2(dx_2,dy_2)\right)\,\pi^1(dx_1) \bigg)^{1/p}\nonumber\\
&\le \bigg(c_p \int_{K_1^c} |x_1|^p\, \mu(dx_1)+c_p |a_0|^p\,\mu(K_1^c)+c_p \int_{K_1^c}\int |x_2|^p\,\pi_{x_1}(dx_2)\mu(dx_1)\nonumber \\
&\quad +c_p\mu(K_1^c)\int |x_2|^p\, \pi_{a_0}(dx_2)+(\kappa/6)^p +(\kappa/6)^p\bigg)^{1/p}\le \kappa.
\end{align}
This concludes the first approximation step.\\

For the second approximation step we first fix $x_1\in K_{1,\kappa}\cup \{a_0\}$. We now approximate the probability measure $\pi_{x_1}$: as $\int |x_2|^p\, \nu(dx_2)<\infty$ there exists a finite set $K_{2,\kappa}(x_1)=\{b_1(x_1), b_2(x_1), \dots, b_{N(x_1)}(x_1)\}\subseteq \R$ such that $b_1(x_1)\le b_2(x_1)\le \dots \le b_{N(x_1)}(x_1)$ and $$\inf_{x_2\in K_{2,\kappa}(x_1)}|x_2-y_2|\le \kappa/3$$ for all $y_2\in [b_{1(x_1)},b_{N(x_1)}]$ and $\int_{[b_1(x_1),b_N(x_1)]^c} |x_2|^p\,  \pi_{x_1}(dx_2)\le \kappa^p/(3c_p)$. 
Let us set $b_0(x_1)\coloneq -\infty,\  b_{N(x_1)+1}\coloneq \infty$ with the convention that $[-\infty, b_1(x_1))\coloneq (-\infty, b_1(x_1))$ . By Tchakaloff's theorem as stated in Lemma \ref{lem:tchakaloff} there exist finitely supported measures $\{\hat{\pi}^{i,x_1}\ : \ i=1, \dots, N(x_1)+1\}$ such that
\begin{align*}
\text{supp}\left(\hat{\pi}^{i,x_1}\right)&\subseteq \Gamma_{x_1}\cap [b_{i-1}(x_1),b_i(x_1)),\\
\hat{\pi}^{i, x_1}([b_{i-1}(x_1),b_i(x_1))&=\pi_{x_1}([b_{i-1}(x_1),b_i(x_1)),\\
\int x_2\, \hat{\pi}^{i,x_1}(dx_2)&=\int_{[b_{i-1}(x_1),b_i(x_1))} x_2\, \pi_{x_1}(dx_2),\\
\int |x_2|^p\, \hat{\pi}^{i,x_1}(dx_2)&=\int_{[b_{i-1}(x_1),b_i(x_1))} |x_2|^p \, \pi_{x_1}(dx_2)\quad\text{for all }i=1,\dots, N(x_1)+1.
\end{align*} 
We set
\begin{align*}
\hat{\pi}(dx_1,dx_2)&=\tilde{\pi}^1(dx_1)\Bigg(\sum_{i=1}^{N(x_1)+1} \hat{\pi}^{i,x_1}(dx_2)\Bigg),
\end{align*}
which yields in particular
\begin{align} \label{eq:martingale}
\begin{split}
\int x_2 \, \hat{\pi}_{x_1}(dx_2)&=\sum_{i=1}^{N(x_1)+1} \int x_2\, \hat{\pi}^{i,x_1}(dx_2)\\
&=\sum_{i=1}^{N(x_1)+1} \int_{[b_{i-1}(x_1),b_i(x_1))} x_2\, \pi_{x_1}(dx_2)=\int x_2\,\pi_{x_1}(dx_2)
\end{split}
\end{align}
and
\begin{align}\label{eq:long2}
W^{nd}_p(\tilde{\pi},\hat{\pi})&\le \left(\int \left( \inf_{\gamma^2\in \Pi(\tilde{\pi}_{x_1}, \hat{\pi}_{x_1})} \int |x_2-y_2|^p\,\gamma^2(dx_2,dy_2) \right) \tilde{\pi}^{1}(dx_1)\right)^{1/p}\nonumber\\
&\le \Bigg(\int \bigg(c_p \int_{(-\infty,b_1(x_1))} \left|x_2\right|^p\,\pi_{x_1}(dx_2)+c_p \int_{(-\infty,b_1(x_1))} \left|x_2\right|^p\,\hat{\pi}^{i,x_1}(dx_2)\nonumber \\
&\qquad+\sum_{i=2}^{N(x_1)}\int_{[b_{i-1}(x_1),b_i(x_1))} \left|x_2-y_2\right|^p\,\frac{(\hat{\pi}^{i,x_1}\times\pi_{x_1})(dx_2,dy_2)}{\pi_{x_{1}}\left(\left[b_{i-1}\left(x_{1}\right), b_{i}\left(x_{1}\right)\right)\right)}\nonumber\\
&\qquad +c_p \int_{[b_N(x_1),\infty)} \left|x_2\right|^p\,\pi_{x_1}(dx_2)+c_p \int_{[b_N(x_1),\infty)} \left|x_2\right|^p\,\hat{\pi}_{\tilde{b}_{N}(x_1)}(dx_2) \bigg)\, \tilde{\pi}^{1}(dx_1)\Bigg)^{1/p}\nonumber\\
&\le \kappa.
\end{align}
This concludes the second approximation step. In particular the estimates above imply that $W_p^{nd}(\pi, \hat{\pi})\le W_p^{nd}(\pi, \tilde{\pi})+W_p^{nd}(\tilde{\pi}, \hat{\pi})\le 2\kappa$ and $\hat{\pi}$ is a probability measure finitely supported on $\Gamma$. \\

We now show \eqref{eq:nachtrag} for $\hat{\pi}$. 
First, we again note that $\tilde{\pi}^1=\hat{\pi}^1$ by definition and by the above construction 
$$\int\left(x_{2}-x_{1}\right) \hat{\pi}_{x_{1}}\left(d x_{2}\right)= \int\left(x_{2}-x_{1}\right) \tilde{\pi}_{x_{1}}\left(d x_{2}\right).$$ Thus we obtain for $x\in \mathrm{supp}((\hat{\pi})^1)$
\begin{align*}
& \int_{\{x_1\ge x\}} (x_2-x_1)\,\hat{\pi}(dx_1,dx_2)-\int_{\{x_1\ge x\}} (x_2-x_1)\,\tilde{\pi}(dx_1,dx_2)\\
&= \int_{\{x_1 \ge x\}} \int (x_2-x_1)\, \hat{\pi}_{x_1}(dx_2)\,\tilde{\pi}^1(dx_1)-\int_{\{x_1 \ge x\}}\int (x_2-x_1)\, \tilde{\pi}_{x_1}(dx_2)\,\tilde{\pi}^1(dx_1)\\
&=\int_{\{x_1 \ge x\}}  \left(\int  x_2\, \hat{\pi}_{x_1}(dx_2)-\int x_2\, \tilde{\pi}_{x_1}(dx_2)\right)\,\tilde{\pi}^1(dx_1)=0.
\end{align*}
Similarly using \eqref{eq:long1} we obtain (recalling the definition of $f^\kappa$ in \eqref{eq:fkappa})
\begin{align*}
& \int_{\{x_1\ge x\}} (x_2-x_1)\,\tilde{\pi}(dx_1,dx_2)-\int_{\{x_1\ge x\}} (x_2-x_1)\,\pi(dx_1,dx_2)\\
&\ge -\int_{K_1^c} \left(|x_1|+|f^{\kappa}(x_1)|+\int |x_2|+|y_2|\,\gamma^2(dx_2,dy_2)\right)\,\pi^1(dx_1)\\
&\quad-\int_{\{x_1 \ge x\}\cap K_1} \left(|x_1-f^{\kappa}(x_1)| +\int |x_2-y_2|\, \gamma^2(dx_2,dy_2)\right)\,\pi^1(dx_1)\ge -\kappa.
\end{align*}
Thus
\begin{align*}
& \int_{\{x_1\ge x\}} (x_2-x_1)\,\hat{\pi}(dx_1,dx_2)-\int_{\{x_1\ge x\}} (x_2-x_1)\,\pi(dx_1,dx_2)\\
&= 
 \int_{\{x_1\ge x\}} (x_2-x_1)\,\hat{\pi}(dx_1,dx_2)-\int_{\{x_1\ge x\}} (x_2-x_1)\,\tilde{\pi}(dx_1,dx_2)\\
&\quad + \int_{\{x_1\ge x\}} (x_2-x_1)\,\tilde{\pi}(dx_1,dx_2)-\int_{\{x_1\ge x\}} (x_2-x_1)\,\pi(dx_1,dx_2)\ge -\kappa.
\end{align*}
This concludes the proof of (i).\\

For (ii) we note that $$\mu\left(\left\{x_1\in \R \ : \ \int (x_2-x_1)\,\pi_{x_1}(dx_2)=0\right\}\right)=1$$ for $\pi\in \mathcal{M}(\mu,\nu)$. We can thus proceed as in (i) on $\Gamma^1\cap \left\{x_1\in \R \ : \ \int (x_2-x_1)\,\pi_{x_1}(dx_2)=0\right\}$, noting that \eqref{eq:martingale} holds. This concludes the proof.
\end{proof}

\begin{proof}[Proof of Lemma \ref{cor:approx}]
Let us adopt the same notation and conventions as in the proof of Lemma \ref{lemma:approx1}. Let us  note that $\mu \preceq_c \nu$ implies $\mathcal{M}(\mu,\nu)\neq \emptyset$ and let us fix a martingale measure $\dot{\pi}\in \mathcal{M}(\mu,\nu)$.  Similarly to the proof of Lemma \ref{lemma:approx1} we also fix $0\neq a_0\in \Gamma^1$ satisfying $ \int |x_2|^p \,\pi_{a_0}(dx_2)\vee  \int |x_2|^p \,\dot{\pi}_{a_0}(dx_2)<\infty.$ Then, applying Lusin's theorem to the measure $\zeta$ defined via 
\begin{align*}
\zeta(A)\coloneq\frac{1}{4}\Bigg(\mu(A)+\frac{\int_{A} |x_1|^p\,\mu(dx_1)}{\int |x_1|^p\,\mu(dx_1)}&+\frac{\int_A\int |x_2|^p \, \pi_{x_1}(dx_2)\,\mu(dx_1)}{\int |x_2|^p\,\nu(dx_2)}\\
&\qquad\qquad\qquad+ \frac{\int_A\int |x_2|^p \, \dot{\pi}_{x_1}(dx_2)\,\mu(dx_1)}{\int |x_2|^p\,\nu(dx_2)}\Bigg)
\end{align*} for every Borel set $A\subseteq \R$,
we can find a compact set $K_1$ such that $a_0\in K_1$,
\begin{align*}
\int_{K_1^c} |x_1|^p  \mu(dx_1)&\le \kappa/(6c_p),\quad \int_{K_1^c}\int |x_2|^p \, \pi_{x_1}(dx_2)\mu(dx_1)\le \kappa/(6c_p),\\
\int_{K_1^c}\int |x_2|^p \, \dot{\pi}_{x_1}(dx_2)\mu(dx_1)&\le \kappa/(6c_p),\quad
\mu(K_1^c)\le \frac{ \kappa/(6c_p)}{|a_0|\vee \int |x_2|^p \, \pi_{a_0}(dx_2)\vee \int |x_2|^p \, \dot{\pi}_{a_0}(dx_2)}
\end{align*} 
and both $x_1\mapsto \pi_{x_1}$ and $x_1\mapsto \dot{\pi}_{x_1}$ are continuous in $W_p$ on $K_1$. Now we proceed exactly as in the proof of Lemma \ref{lemma:approx1} for $\Gamma=\R^2$: we conclude from the above  that there exists $\delta >0$ such that $$W_p(\pi_{x_1}, \pi_{y_1})\vee W_p(\dot{\pi}_{x_1}, \dot{\pi}_{y_1})\le \kappa/6$$ for all $x_1, y_1\in K_1$ with $|x_1-y_1|\le \delta$ and a finite set $K_{1,\kappa}=\{a_0,a_1, \dots, a_N\}$ of $K_1$ 
such that $a_1\le \dots\le a_N$ and $$\min_{a\in K_{1,\kappa},\ a\le y_1}(y_1-a)\le \kappa/6\wedge \delta$$ for all $y_1\in K_1$,  where $a_1$ and $a_N$ are the left and right end-points of $K_1$. Let $(P_i)_{i=0}^{N}$ be a disjoint partition of $K_1$ given by
\begin{align*}
P_i=\left\{ y_1\in K_1:\ a_i= \mbox{argmin}_{a\in K_{1,\kappa},\ a\le y_1} (y_1-a) \right\}, \qquad i=0, \dots, N.
\end{align*}
We now define 
\begin{align}\label{eq:fkappa}
f^{\kappa}(x_1)=a_0\mathds{1}_{K_1^c}(x_1)+\sum_{i=0}^{N} a_i\mathds{1}_{P_i}(x_1).
\end{align}
Applying now the second approximation step in the proof of Lemma \ref{lemma:approx1} for both $\pi$ and $\dot{\pi}$ individually and using the same estimates as in the proof of Lemma \ref{lemma:approx1} we can thus find finitely supported measures, which we call $\hat{\pi} \in \Pi(\bar{\mu}, \hat{\nu})$ and $\pi'\in \mathcal{M}(\bar{\mu},\bar{\nu})$, with the property $W_p^{nd}(\pi,\hat{\pi})\le 2\kappa$ and $W_p^{nd}(\dot{\pi},\pi')\le 2\kappa$. In particular $\bar{\mu}\preceq_c \bar{\nu}$ and $\hat{\pi},\pi'$ have the same first marginals. We note that $$W_p(\hat{\nu},\bar{\nu})\le W_p(\hat{\nu},\nu)+W_p(\nu,\bar{\nu})\le 4\kappa.$$ Let $\zeta$ be an optimal coupling for $W_p(\hat{\nu},\bar{\nu})$. We define
\begin{align*}
\bar{\pi}(dx_1,dx_2)=\int \hat{\pi}(dx_1,dy_2) \zeta_{y_2}(dx_2)\in \Pi(\bar{\mu},\bar{\nu})
\end{align*}
and conclude
\begin{align*}
W_p^{nd}(\pi,\bar{\pi}) 
&\le W_p^{nd}(\pi,\hat{\pi})+W_p^{nd}(\hat{\pi},\bar{\pi})
\le 2\kappa+\left(\int |x_2-y_2|^p\,\zeta(dx_2,dy_2)\right)^{1/p}\le 6\kappa.
\end{align*}
This proves the claim.
\end{proof}

\begin{proof}[Proof of Lemma \ref{lemma:convexity}]
Define $\bar{\mu}=1/n\sum_{i=1}^n \mu_i$. Then $\mu_i \ll \bar{\mu}$ for all $i=1,\dots, n$ and by the Radon-Nykodym theorem there exist densities $g^1, \dots, g^n$ such that $d\mu^i =g^i d\bar{\mu}$ for all $i=1, \dots, n$. Next we note that a disintegration of $\frac{1}{n} \sum_{i=1}^n \pi^i$ is given by
\begin{align*}
\left(\frac{\frac{1}{n}\sum_{i=1}^n g^i(x_1) \pi_{x_1}^i(dx_2)}{\frac{1}{n} \sum_{i=1}^n g^i(x_1)}\right)_{x_1:\ \sum_{i=1}^n g^i(x_1)>0}
\end{align*}
and similarly 
\begin{align*}
\left(\frac{\frac{1}{n}\sum_{i=1}^n g^i(x_1) \tilde{\pi}_{ x_1}^i(dx_2)}{\frac{1}{n} \sum_{i=1}^n g^i(x_1)}\right)_{x_1:\ \sum_{i=1}^n g^i(x_1)>0}
\end{align*}
is a disintegration of $\frac{1}{n} \sum_{i=1}^n \tilde{\pi}^i$. Given $\gamma^{2,i}_{x_1}\in \Pi(\pi_{x_1}^i, \tilde{\pi}_{ x_1}^i)$ for $i=1,\dots, n$ and all  $x_1\in \cup_{i=1}^n \text{supp}(\mu^i)$ we thus conclude
\begin{align*}
\frac{\frac{1}{n}\sum_{i=1}^n g^i(x_1) \gamma^{2,i}_{x_1}(dx_2,dy_2)}{\frac{1}{n} \sum_{i=1}^n g^i(x_1)}\in \Pi\left(\left(\frac{1}{n} \sum_{i=1}^n \pi^i\right)_{x_1}, \left(\frac{1}{n} \sum_{i=1}^n \tilde{\pi}^i\right)_{x_1}\right).
\end{align*}
In particular choosing $\gamma^{2,i}_{x_1}\in \Pi(\pi_{x_1}^i, \tilde{\pi}_{ x_1}^i)$ optimal for $W_1(\pi^i_{x_1}, \tilde{\pi}^i_{x_1})$ it follows
\begin{align*}
W_1^{nd}\left(\frac{1}{n}\sum_{i=1}^n \pi^i,\frac{1}{n}\sum_{i=1}^n \tilde{\pi}^i\right)&\le
 \int \int |x_2-y_2|\,\frac{\frac{1}{n}\sum_{i=1}^n g^i(x_1) \gamma^{2,i}_{x_1}(dx_2,dy_2)}{\frac{1}{n} \sum_{i=1}^n g^i(x_1)}\,\left( \frac{1}{n}\sum_{i=1}^n g^i(x_1)\right)\bar{\mu}(dx_1)\\
 &= \frac{1}{n}\sum_{i=1}^n \int W_1(\pi^i_{x_1}, \tilde{\pi}^i_{x_1})\,\mu^i(dx_1),
\end{align*}
which proves the claim.
\end{proof}

\begin{proof}[Proof of Lemma \ref{lemma:approx2a}]
Let us denote by $\zeta\in \Pi(\mu, \tilde{\mu})$ an optimal coupling for $W_p(\mu, \tilde{\mu})$ and by $\eta\in \Pi(\nu,\tilde{\nu})$ an optimal coupling for $W_p(\nu,\tilde{\nu})$. Now we define $\hat{\rho}\in \mathcal{P}(\R^4)$ via $$\hat{\rho}(dx_1,dx_2,dy_1,dy_2)=\zeta_{x_1}(dy_1)\,\eta_{x_2}(dy_2)\,\pi(dx_1,dx_2).$$
Let $$\tilde{\pi}(dy_1,dy_2)\coloneq \int_{\R\times \R}\hat{\rho}(dx_1,dx_2,dy_1,dy_2)$$
be its projection to the third and fourth component. We compute
\begin{align*}
W_p(\pi,\tilde{\pi})&\le \left( \int |x_1-y_1|^p+|x_2-y_2|^p \,\hat{\rho}(dx,dy)\right)^{1/p}\\
&=\left( \int \left(|x_1-y_1|^p+|x_2-y_2|^p\right) \,\zeta_{x_1}(dy_1)\eta_{x_2}(dy_2)\pi(dx_1, dx_2)\right)^{1/p}\\
&= \left( W_p^p(\mu, \tilde{\mu})+W_p^p(\nu, \tilde{\nu})\right)^{1/p}\\
&\le W_p(\mu, \tilde{\mu})+W_p(\nu, \tilde{\nu}).
\end{align*}
The proof of \eqref{eq:gaoyue2} follows by use of the triangle inequality and Jensen's inequality as in \cite[proof of Prop. 4.2, p.20]{Guo:2017txa}: indeed, choosing $p=1$ yields
\begin{align*}
&\int \left|\int (y_2-y_1)\,\tilde{\pi}_{y_1}(dy_2)\right|\,\tilde{\mu}(dy_1)\\
&= \int \left|\int (y_2-y_1)\,\hat{\rho}_{y_1}(dx_1,dx_2,dy_2)\right|\,\tilde{\mu}(dy_1)\\
&\le \int \left|\int (y_2-x_2)\,\hat{\rho}_{y_1}(dx_1,dx_2,dy_2)\right|\,\tilde{\mu}(dy_1)
 + \int \left|\int (x_2-x_1)\,\hat{\rho}_{y_1}(dx_1, dx_2,dy_2)\right|\,\tilde{\mu}(dy_1)\\
&\quad + \int \left|\int (x_1-y_1)\,\hat{\rho}_{y_1}(dx_1,dx_2,dy_2)\right|\,\tilde{\mu}(dy_1)\\
&\le \int \left|y_2-x_2\right|\hat{\rho}_{y_1}(dx_1,dx_2,dy_2)\,\tilde{\mu}(dy_1)+\int \left|\int (x_2-x_1)\,\zeta_{y_1}(dx_1)\pi_{x_1}(dx_2)\right|\,\tilde{\mu}(dy_1)\\
&\quad +\int  \left|x_1-y_1\right|\, \hat{\rho}_{y_1}(dx_1,dx_2,dy_2)\,\tilde{\mu}(dy_1)\\
&\le \int \left|y_2-x_2\right|\,\eta(dx_2, dy_2)+0 + \int  \left|x_1-y_1\right|\,\zeta(dx_1,dy_1)\\
&= W_1(\nu,\tilde{\nu})+W_1(\mu, \tilde{\mu}).
\end{align*}
This concludes the proof.
\end{proof}

\begin{proof}[Proof of Lemma \ref{lemma:uniform_integrability}]
Using e.g. an argument very similar to \cite[Def. 6.8]{Villani:2009ha}, the statement
\begin{align*}
\lim_{n \to \infty} \left( \int c(x_1, x_2)\,\pi^n(dx_1, dx_2)- \int c(x_1, x_2)\, \tilde{\pi}^n(dx_1, dx_2)\right)=0
\end{align*}
for all continuous functions $c:\R^2\to \R$ such that $|c(x_1, x_2)|\le C(1+|x_1|^p+|x_2|^p)$ with $C\ge 0$
is equivalent to $\lim_{n\to \infty} W_p(\pi^n, \hat{\pi}^n)=0$, which is in turn equivalent to $\lim_{n\to \infty} W_1(\pi^n, \hat{\pi}^n)=0$ and the convergence of $p$th moments of $(\pi^n)_{n\in \N}$ and $(\hat{\pi}^n)_{n\in \N}$. But $\pi^n$ and $\tilde{\pi}^n$ have the same moments, which converge because of the assumption $\lim_{n\to \infty} W_p(\mu^n, \mu)=0=\lim_{n\to \infty} W_p(\nu^n, \nu)$. The claim follows.
\end{proof}

\begin{proof}[Proof of Lemma \ref{lem:aap_1}]
Take $\epsilon>0$. From the definition of $W_{nd}^1(\pi,\tilde{\pi})$ we can find $\gamma^1\in \Pi(\pi^1, \tilde{\pi}^1)$ and a (Borel-measurable) $\gamma^2_{(x_1,y_1)}\in \Pi(\pi_{x_1},\tilde{\pi}_{y_1})$ such that
\begin{align*}
W_1^{nd}(\pi,\tilde{\pi})&\ge   \int\left(|x_1-y_1|+\int |x_2-y_2|\,\gamma^2_{(x_1,y_1)}(dx_2, dy_2)\right)\, \gamma^1(dx_1,dy_1)-\epsilon.
\end{align*}
Now we simply write
\begin{align*}
&\left|\int \left|\int (x_2-x_1)\,\pi_{x_1}(dx_2) \right|\,\pi^1(dx_1)- \int \left| \int (y_2-y_1)\,\tilde{\pi}_{y_1}(dy_2)\right|\,\tilde{\pi}^1(dy_1)\right|\\
&= \left|\int\left(\left|\int (x_2-x_1)\,\gamma^2_{(x_1,y_1)}(dx_2,dy_2) \right|-  \left| \int (y_2-y_1)\,\gamma^2_{(x_1,y_1)}(dx_2,dy_2)\right|\right)\,\gamma^1(dx_1,dy_1)\right|\\
&\le \int \left|\int \left[(x_2-x_1)-(y_2-y_1)\right]\,\gamma^2_{(x_1,y_1)}(dx_2,dy_2)\right|\,\gamma^1(dx_1,dy_1)\\
&\le \int \left( |x_1-y_1| +\int |x_2-y_2|\,\gamma^2_{(x_1,y_1)}(dx_2,dy_2) \right)\,\gamma^1(dx_1,dy_1)\\
&\le W_1^{nd}(\pi,\tilde{\pi})+\epsilon.
\end{align*}
As $\epsilon>0$ was arbitrary, the claim follows. For the second claim we note that the triangle inequality also holds for $(\cdot)^-$ and $(\cdot)^-\le |\cdot|$, so the claim still follows.
\end{proof}

\begin{proof}[Proof of Lemma \ref{lem:aap2}]
Clearly $\pi\in \Pi(\mu,\nu)$. Fix now $\epsilon>0$. Lemma \ref{lem:aap_1} states that
\begin{align*}
\left|\int \left(\int (y_2-y_1)\,\pi^n_{y_1}(dy_2) \right)^-\,(\pi^n)^1(dy_1)- \int \left( \int (x_2-x_1)\,\pi_{x_1}(dx_2)\right)^-\,\mu(dx_1)\right|\le W_1^{nd}(\pi^n,\pi),
\end{align*}
in particular 
\begin{align*}
 \int \left( \int (x_2-x_1)\,\pi_{x_1}(dx_2)\right)^-\,\mu(dx_1)\le 2\epsilon
\end{align*}
for large $n\in \N$.
As $\mu\preceq_c \nu$ and $\pi\in \Pi(\mu,\nu)$ we have
\begin{align*}
0&= \int (x_2-x_1)\,\pi(dx_1,dx_2) \\
&=\int \left( \int (x_2-x_1)\,\pi_{x_1}(dx_2)\right)^+\,\mu(dx_1)-\int \left( \int (x_2-x_1)\,\pi_{x_1}(dx_2)\right)^-\,\mu(dx_1).
\end{align*}
In particular
\begin{align*}
\int \left| \int (x_2-x_1)\,\pi_{x_1}(dx_2)\right|\,\mu(dx_1)\le 4\epsilon.
\end{align*}
As $\epsilon>0$ was arbitrary, the claim now follows.
\end{proof}

\end{appendix}

\bibliographystyle{apalike}
\bibliography{paperslib}

\end{document}